\documentclass[10pt]{amsart}
\usepackage[labelfont=bf]{caption}

\usepackage{lineno,hyperref}

\RequirePackage{xr}
\RequirePackage[OT1]{fontenc}
\RequirePackage{amsthm, amsmath,amssymb,dsfont,bm}
\RequirePackage{algorithm}
\RequirePackage{algpseudocode}
\RequirePackage{graphicx} 
\RequirePackage[section]{placeins}
\RequirePackage{color}
\RequirePackage{chngcntr}
\RequirePackage{enumerate}
\usepackage[shortlabels]{enumitem}
\RequirePackage{longtable}
\usepackage{rotating}
\usepackage{multirow}
\usepackage{mathtools}
\usepackage[english]{babel} 
\usepackage{xcolor}

\RequirePackage[OT1]{fontenc} %\RequirePackage{amsthm}
\RequirePackage{amsfonts} \RequirePackage{amssymb}
\RequirePackage{bbm}

\RequirePackage{graphicx}
\RequirePackage{verbatim}
\usepackage{url}

\usepackage{subfigure}
\usepackage{mathtools} % ab hier selber eingebunden
\usepackage{bm}
\usepackage{dsfont}
\usepackage{color}

\usepackage{float}                 % für erweiterten Symbolvorrat

\RequirePackage{a4wide} 

\newcommand{\N}{ \mathbb{N} }

\newcommand{\Z}{ \mathbb{Z} }

\newcommand{\R}{ \mathbb{R} }

\newcommand{\trunc}[1]{ {\lfloor #1 \rfloor} }
\newcommand{\wh}[1]{ \widehat{ #1 } }
\newcommand{\wt}[1]{ \widetilde{ #1 } }

\newcommand{\argmax}{\operatorname*{argmax}}

\newcommand{\calA}{\mathcal{A}}

\newcommand{\calD}{\mathcal{D}}

\newcommand{\calF}{\mathcal{F}}
\newcommand{\calG}{\mathcal{G}}

\newcommand{\calJ}{\mathcal{J}}

\newcommand{\calS}{\mathcal{S}}

\newcommand{\calU}{\mathcal{U}}

\newcommand{\calV}{\mathcal{V}}

\newcommand{\eins}{{\bm 1}}
\newcommand{\matA}{{\bm A}}
\newcommand{\matB}{{\bm B}}

\newcommand{\matD}{{\bm D}}

\newcommand{\matX}{{\bm X}}

\newcommand{\matM}{{\bm M}}
\newcommand{\matP}{{\bm P}}
\newcommand{\matS}{{\bm S}}

\newcommand{\matV}{{\bm V}}

\newcommand{\vecnull}{{\bm 0}}

\newcommand{\veca}{{\bm a}}
\newcommand{\vecb}{{\bm b}}
\newcommand{\vecB}{{\bm B}}
\newcommand{\vecc}{{\bm c}}
\newcommand{\vecC}{{\bm C}}
\newcommand{\vecd}{{\bm d}}
\newcommand{\vecD}{{\bm D}}
\newcommand{\vece}{{\bm e}}

\newcommand{\vecG}{{\bm G}}

\newcommand{\vecl}{{\bm l}}
\newcommand{\vecr}{{\bm r}}

\newcommand{\vecM}{{\bm M}}
\newcommand{\vecu}{{\bm u}}

\newcommand{\vecR}{{\bm R}}
\newcommand{\vecS}{{\bm S}}

\newcommand{\vecT}{{\bm T}}
\newcommand{\vecv}{{\bm v}}

\newcommand{\vecw}{{\bm w}}
\newcommand{\vecx}{{\bm x}}

\newcommand{\vecy}{{\bm y}}
\newcommand{\vecY}{{\bm Y}}
\newcommand{\vecZ}{{\bm Z}}
\newcommand{\vecz}{{\bm z}}
\newcommand{\vecU}{{\bm U}}

\newcommand{\bfxi}{\bm \xi}
\newcommand{\bfmu}{\bm \mu}

\newcommand{\bfeps}{\bm \epsilon}

\newcommand{\bfSigma}{\bm\Sigma}

\newcommand{\EE}{{\mathbb{E}}}
\newcommand{\PP}{{\mathbb{P}}}
\newcommand{\Var}{{\mathbb{V}}}
\newcommand{\Cov}{{\mathbb{C}}}

\newcommand{\matid}{I}

\newcommand{\cadlag}{c\`adl\`ag }

\newcommand{\aarr}{  \mathfrak{a}  }
\newcommand{\barr}{  \mathfrak{b} }
\newcommand{\carr}{  \mathfrak{c} }

\newcommand{\ava} { \vecv_n^\top \veca_n }
\newcommand{\avb} { \vecv_n^\top \vecb_n }
\newcommand{\avc} { \vecv_n^\top \vecc_n }

\newcommand{\awa} { \vecw_n^\top \veca_n }
\newcommand{\awb} { \vecw_n^\top \vecb_n }
\newcommand{\awc} { \vecw_n^\top \vecc_n }

\newtheorem{theorem}{Theorem}

\newtheorem{example}{Example}
\newtheorem{remark}{Remark}
\newtheorem{lemma}{Lemma}

\begin{document}

\title[Change-Points in High-Dimensional Covariance Matrices]{\Large Testing and Estimating Change-Points in the Covariance Matrix of a High-Dimensional Time Series}

\author{{\large \sc Ansgar Steland}}

\address{RWTH Aachen University \\ Institute of Statistics \\ W\"ullnerstr. 3 \\ D-52056 Aachen \\ Germany} 

\email{steland@stochastik.rwth-aachen.de }

\date{December 2019}

\maketitle

\begin{abstract}
	This paper studies methods for testing and estimating change-points in the covariance structure of a high-dimensional linear time series. The assumed framework allows for a large class of multivariate linear processes (including vector autoregressive moving average (VARMA) models) of growing dimension  and spiked covariance models. The approach uses bilinear forms of the centered or non-centered sample variance-covariance matrix. Change-point testing and estimation are based on  maximally selected weighted cumulated sum (CUSUM) statistics. Large sample approximations under a change-point regime are provided including a multivariate CUSUM transform of increasing dimension. For the unknown asymptotic variance and covariance parameters associated to (pairs of) CUSUM statistics we propose consistent estimators. Based on  weak laws of large numbers for their sequential versions, we also consider stopped sample estimation where observations until the estimated change-point are used. Finite sample properties of the procedures are investigated by simulations and their application is illustrated by analyzing a real data set from environmetrics. \\ \ \\
\end{abstract}

\keywords{{\bf Keywords:} Big data, Change-point, CUSUM transform, Data science, High-dimensional statistics, Projection, Spatial statistics, Spiked covariance, Strong approximation, VARMA processes}

\subjclass{{\bf MSC:} 62E20, 62M10,  62H99}

\section{Introduction}

High-dimensional big data arise in diverse fields such as environmetrics, engineering and finance. From a data science viewpoint statistical methods and tools are needed, which allow to answer questions posed to the data, and mathematical results justifying their validity under mild regularity conditions. The latter especially requires asymptotics for the case that the data dimension is large in comparison to the sample size. In this paper, a high-dimensional time series is model is studied and all asymptotic results allow for increasing dimension without any constraint relative to the sample size. The proposed procedures are investigate by simulations and applied to real data from environmetrics.

We study methods for the detection of a change-point in a high-dimensional covariance matrix and estimation of its location based on a time series. The proposed procedures investigate estimated bilinear forms of the covariance matrix, in order to test for the presence of a change-point as well as to estimate its location. The bilinear forms use weighting vectors with finite $ \ell_1$- resp. $ \ell_2 $-norms which may even grow slowly as the sample size increases. 
This approach is natural from a mathematical point of view and has many applications in diverse areas: Analysis of projections onto subspaces spanned by (sparse) principal directions, infering the dependence structure of high-dimensional sensor data, e.g., from environmental monitoring, testing for a change of the autocovariance function of a univariate series or financial portfolio analysis, to mention a few. These problems have in common that the dimension $d$ can be large and may be even larger than the sample size $n$. The results of this paper allow for this case and do not impose a condition on the growth of the dimension. Multivariate versions of CUSUM statistics are also considered.

The problem to detect changes in a sequence of covariance matrices has been studied by several authors and recently gained increasing interest, although the literature is still somewhat sparse. Going beyond the binary segmentation approach, \cite{ChoFryzlewicz2015} propose a sparsified segmentation procedure where coordinate-wise CUSUM statistics are thresholded to segment the second-order structure. But these results do not cover significance testing. To test for a covariance change in a time series, \cite{GaleanoPena2007}, who also give some historical references, consider CUSUM and likelihood ratio statistics for fixed dimension $d$ assuming a parametric linear process with Gaussian errors. Their CUSUM statistics, however, require knowledge of the covariance matrix of the innovations when no change is present. \cite{BerkesGombayHorvath2009} studied unweighted and weighted CUSUM change-point tests for a linear process to detect a change in the autocovariance function, but only for a fixed lag. Further, their theoretical results are restricted to the null hypothesis of no change. Kernel methods for this problem have been studied by \cite{Steland2005} and \cite{LiZhao2013}. \cite{Aue2009} studied break detection in vector time series for fixed dimension and provide an approximation of the limiting distribution of their test statistic, an unweighted CUSUM, if $d$ is large.  
Contrary, the approach studied in this paper allows for growing dimension $d$ without any constraint such as $ d/n \to y \in (0,1) $, as typically imposed in random matrix theory, $ d = O( h(n) ) $ for some increasing function $h$, e.g., exponential growth as in \cite{Avanesov2018} (which is, however, constrained to i.i.d. samples), or (again for i.i.d. samples) asymptotics for the eigenstructure under the assumption  $ d/(n \lambda_j) = O(1) $ for the spiked eigenvalues $ \lambda_j $,
 \cite{WangFan2017}, which allows for $ d/n \to \infty $ provided the eigenvalues diverge.

It is shown that, for the imposed high-dimensional time series model, (weighted) CUSUM statistics associated to the sample covariance matrix can be approximated by (weighted) Gaussian bridge processes. Under the null hypothesis this follows from \cite{StelandSachs2017} and one can also consider an increasing number of such statistics by virtue of the results in \cite{StelandSachs2018}. The asymptotics under a change-point regime, however, is more involved and is provided in this paper. Both single CUSUM statistics and multivariate CUSUM transforms corresponding to a set of projection vectors are studied. The dimension of the time series as well as  the dimension of the multivariate CUSUM transform is allowed to grow with the sample size in an unconstrained way. The results of this paper extend \cite{StelandSachs2017, StelandSachs2018}, especially by studying weighted CUSUMs, providing refined martingale approximations and relaxing the conditions on the projection vectors.

Further, consistent estimation of the unknown variance and covariance parameters is studied without the need to estimate eigenstructures. As well known, this essentially would require conditions under which the covariance matrix can be estimated consistently in the Frobenius norm, which needs the restrictive condition $ d=o(n) $ on the dimension according to the results of \cite{LedoitWolf2004} and \cite{Sancetta2008}, or requires to assume appropriately constrained models. Estimators for the asymptotic variance and covariance parameters associated to a single resp. a set of CUSUM statistics have already been studied under the no-change hypothesis in  \cite{StelandSachs2017} and \cite{StelandSachs2018}. These estimators are now studied under a change-point model, generalized to deal with two pairs of projection vectors describing the asymptotic covariance between pairs of (weighted) CUSUMs and studied  from a sequential viewpoint which allows us to to propose  stopped-sample estimators using the given sample until the estimated change point. This is achieved by proving a uniform law of large numbers for the sequential estimators.  

Closely related to the problem of testing for a change-point is the task of estimating its location. It is shown hat the change-point estimator naturally associated to the weighted or unweighted CUSUM statistic is consistent. As a consequence, the well known iterative binary segmentation algorithm, dating back to \cite{Vostrikova1981}, can be used to locate multiple change points. 

The organization of the paper is as follows. Section~\ref{Sec:ModelAssumptions} introduces the framework, discusses several models appearing as special cases, introduces the proposed methods and discusses how to select the projection vectors.  The asymptotic results are provided in Section~\ref{Sec:Asymptotics}. They cover
strong and weak approximations for the (weighted) partial sums of the bilinear forms and for associated CUSUMs as well as consistency theorems for the proposed estimators of unknowns. Section~\ref{Sec:ChangePoint} considers the problem to estimate the change-point. Simulations are presented in Section~\ref{Sec:Simulations}. In Section~\ref{Sec:Example} the methods are illustrated by analyzing the dependence structure of ozone measurements from $444$ monitors across the United States over a five-year-period. Main proofs are given in Section~\ref{Sec:Proofs}, whereas additional material is deferred to an appendix. 

\section{Model, assumptions and procedures}
\label{Sec:ModelAssumptions}

\subsection{Notation}

 Throughout the paper $ a_{nk} \stackrel{n,k}{\ll} b_{nk} $ for two arrays of real numbers means that there exists a constant $ C < \infty $, such that $ a_{nk} \le C b_{nk} $ for all $n, k $. $ (\Omega, \calA, \PP ) $ denotes the underlying probability space  on which the vector time series is defined. $ \EE $ denotes expectation (w.r.t. $\PP $), $\Var$  the variance and $ \Cov $ the covariance. For a logical expression $E$ we let $  \eins( E ) $ denote the associated indicator function. If $A$ is a set, then $ \eins_A $ is the usual characteristic function, whereas $ \eins_n $ for $n \in \N $ denotes the $n$-vector with entries $1$ and $ \vecnull_n $ is the null $n$-vector. $ \| \cdot \|_2 $ is the vector-2 norm, $ \| \cdot \|_{\ell_p} $, $ p \in \N $, the $ \ell_p $-norm for sequences and $ \| \cdot \|_\infty $ the maximum norm for sequences or vectors. $ \| \cdot \|_{op} $ denotes the semi norm $ \| T \|_{op} = \sup_{f: \| f \| = 1} | (f,Tf) \| $ for a linear operator $T $ on a Hilbert space with inner product $ (\cdot, \cdot) $. $ X_n \Rightarrow X $ denotes weak convergence of a sequence of \cadlag processes in the Skorohod space $ D[0,1] $ equipped with the usual metric. 

\subsection{Time series model and assumptions}

Let us assume that the coordinates of the vector time series $ \vecY_{ni} = ( Y_{ni}^{(1)}, \ldots, Y_{ni}^{(d_n)} )^\top $ are given by
\begin{equation}
	\label{ModelYCoord}
	Y_{ni}^{(\nu)} = Y_{ni}^{(\nu)}( \aarr ) = \sum_{j=0}^\infty a_{nj}^{(\nu)} \epsilon_{n,i-j}, \qquad i \in \{ 1, \ldots, n \}, \nu \in \{ 1, \ldots, d_n \}, n \ge 1, 
\end{equation}
for coefficients $ \aarr = \{ a_{nj}^{(\nu)} : j \ge 0, n \ge 1 \} $ and  independent zero mean errors $ \{ \epsilon_{ni} : i \in \Z, n \ge 1 \} $ satisfying the following two assumptions.

\textbf{Assumption (D):} An array $ \aarr = \{ a_{nj}^{(\nu)} : j \ge 0, \nu \in \{1, \ldots, d_n\}, n \ge 1 \} $ of real numbers satisfies the decay condition (D), if for some $ \theta \in (0,1/2) $
\begin{equation}
	\label{AssDecayCond}
	\sup_{n \ge 1} \max_{1 \le \nu \le d_n} | a_{nj}^{(\nu)} | \stackrel{j}{\ll} \min(1,j)^{-3/4 - \theta/2}.
\end{equation}

\textbf{Assumption (E):} $ \{ \epsilon_{nk} : k \in \Z,  n \in \N \} $, is an array of independent mean zero random variables with $ \sup_{n \ge 1} \sup_{k \in \Z} \EE | \epsilon_{nk} |^{4+\delta} < \infty $ and moment arrays $
\sigma_{nk}^2 = \EE( \epsilon_{nk}^2 ),  \gamma_{nk} = \EE( \epsilon_{nk}^3 ),
$  $ 1 \le k \le n $, $ n \ge 1 $, 
satisfying
\[
\frac{1}{\ell} \sum_{i=1}^\ell i | \sigma_{ni}^2 - s_{n1}^2 | = O( \ell^{-\beta} ), \qquad  \frac{1}{\ell} \sum_{i=1}^\ell i | \gamma_{ni} - \gamma_{n}| = O( \ell^{-\beta} ),
\] 
for some $ \beta > 1 + \theta $ and sequences $ \{ s_{n1} \} $ and $ \{ \gamma_n \} $.

The assumptions on $ \sigma_{ni}^2 $ and $ \gamma_{ni} $ allow for a certain degree of inhomogeneity of the second and third moments. Especially, under the change-point model described below, where the coefficients of the linear processes change
after the change-point $ \tau = \tau_n $,  these assumptions cover weak effects of the change on the second resp. third moments. An example satisfying the conditions is given by
\[
\sigma_{ni}^2 = s_{n1}^2 + \frac{\kappa_i}{i} \Delta_{\sigma^2,ni}, 
\quad \Delta_{\sigma^2,ni} = \eins( i \le \tau ) \sigma_0^2 + 
\eins( i > \tau ) \sigma_1^2,
\] 
for  two positive constants $ \sigma_0^2 \not= \sigma_1^2 $ and $ \kappa_i \in \R $, $i \ge 1 $, with $ \kappa_i = o(i) $ and $ \sum_{i=1}^\ell  | \kappa_i | \sim l^{1+\beta} $.

\subsection{Spiked covariance model}

The spiked covariance model is a common framework to study estimation of the eigenstructure for high-dimensional data. For  $ r \in \N $ let $ \lambda_{1} > \cdots > \lambda_{r} > 0 $ and let $ \vecu_{nj} = ( u_{nj}^{(\nu)} )_{\nu=1}^{d} \in \R^{d} $, $ j \in \{ 1, \dots, r \} $, be orthonormal vectors with $ \| \vecu_{nj} \|_{\ell_1} \le C $ for $ j \in \{ 1, \ldots, r \} $. Assume that 
\begin{equation}
	\label{SpikedCovModel}
	\bfSigma_n = \sum_{j=1}^r \lambda_{j} \vecu_{nj} \vecu_{nj}^\top + \sigma^2 \matid_{d_n}.
\end{equation}
The $r$ leading eigenvalues of $ \bfSigma_n $ under model (\ref{SpikedCovModel}) are $ \lambda_{j} + \sigma^2 $, $ j \in  \{1, \ldots, r\} $, and represent spikes in the spectrum, which is otherwise flat and given by $ \sigma^2 $. The assumption that the eigenvectors are $ \ell_1 $-bounded is common in high-dimensional statistics, especially when assuming a spiked covariance model: \cite{JohnstoneLu2009} have shown that principal component analysis (PCA) generates inconsistent estimates of the leading eigenvectors if $ d/n \to y \in (0,1) $, which motivated developments on sparse PCA.  Minimax bounds for sparse PCA have been studied by \cite{BirnbaumJohnstonNadlerPaul2013} under $ \ell_q $-constraints on the eigenvectors for $ 0 < q < 2 $. For example, the simple diagonal thresholding estimator $ \wh{\vecu}_{nj}^{th} $  of the $j$th leading eigenvector $ \vecu_{nj} $ of \cite{JohnstoneLu2009} satisfies $ \EE \| \wh{\vecu}_{nj}^{th} -  \wh{\vecu}_{nj}^{th}{}^\top   \vecu_{nj} \vecu_{nj} \|_2^2 = O( n^{-1/4} ) $ and the iterated version of \cite{Ma2011} attains the optimal rate $ O( n^{(1-q/2)}  ) $, see also \cite{PJ2007}.  An $ \ell_1 $ sparseness assumption on the eigenvectors is weaker than the (joint) $k$-sparseness condition on the row support of matrix of eigenvectors imposed in \cite{CaiMaWu2015}, who study optimal estimation under the spectral norm. 

Model (\ref{SpikedCovModel}) can be described in terms of (\ref{ModelYCoord}):  Let $ c_{n,r-1+\nu}^{(\nu)} = \sigma^2 $, $ c_{n,j-1}^{(\nu)} = \lambda_{j}^{1/2} u_{nj}^{(\nu)} $, $ j \in  \{1, \ldots, r  \} $,  and $ c_{nj}^{(\nu)} = 0 $, $ j > r+d $,  $ \nu \in  \{1, \dots, d\} $. Then for $ \epsilon_t $ i.i.d $\mathcal{N}(0,1)$ the MA($r+d-1$) series
$
Y_{nt}^{(\nu)} = \sum_{j=0}^{r-1} c_{nj}^{(\nu)} \epsilon_{t-j} + \sigma \epsilon_{t-r-\nu},  \nu \in  \{1, \dots, d\}, 
$
have the covariance matrix (\ref{SpikedCovModel}). The decay condition (D) follows from  $ \sup_{1 \le n} \max_{1 \le \nu \le r} |  u_{nj}^{(\nu)} | \le \sup_{n \ge 1} \| \vecu_{nj} \|_{\ell_1} $. 

We may conclude that our methodology covers the above spiked covariance under which sparse PCA provides consistent estimates of the leading eigenvectors, which are an attractive choice for the projection vectors on which the proposed change-point procedures are based on. The literature on such consistency results is, however, not yet matured and typically assumes i.i.d. data vectors, whereas the framework studied here considers time series.

\subsection{Multivariate linear time series and VARMA processes}

The above linear process framework is general enough to host classes of multivariate linear processes and vector autoregressive models with respect to a $q$-variate noise process, $ q \in \N $. These processes are usually studied for a sequence of innovations, but since our constructions work for arrays, we consider this setting.

{\em Multivariate linear processes:} Let $ 0 = r_1 \le r_2 \le \cdots \le r_q $ be integers and define the $q$-variate innovations
\[
  \bfeps_{ni} = ( \epsilon_{n,i-r_1}, \ldots, \epsilon_{n,i-r_q} )^\top, \qquad i \ge 1, n \ge 1,
\]
based on $ \{ \epsilon_{ni} : i \ge 1, n \ge 1 \} $. If $ \epsilon_{ni} $ have homogeneous variances, then $ \EE( \bfeps_{n0} \bfeps_{nk}^\top ) \not= \vecnull $ iff. $ k \in \{ r_j - r_i : 1 \le i, j \le q  \} $, $ k \not= 0 $, such that for large enough $ r_j $, $ j \ge 2 $, the innovations are arbitrarily close to white noise.
Let $ \matB_{nj} = ( \vecb_{nj,1}, \ldots, \vecb_{nj,d_n} )^\top $, be $ (d_n \times q) $-dimensional matrices with row vectors $ \vecb_{nj,\nu} = ( b_{nj}^{(\nu,1)}, \ldots, b_{nj}^{(\nu,q)} )^\top $, $ \nu \in \{1, \ldots, d_n\} $, for $ j \ge 0 $. Then the $ d_n $-dimensional linear process 
\[
  \vecZ_{ni} = \sum_{j=0}^\infty \matB_{nj} \bfeps_{n, i-j} 
\]
has coordinates 
$
  Z_{ni}^{(\nu)} = \sum_{j=0}^\infty \sum_{\ell=1}^q b_{nj}^{(\nu,\ell)} \epsilon_{n,i-r_\ell-j}, 
$
which attain the representation
\begin{equation}
\label{MultivLinProc}
    Z_{ni}^{(\nu)} = \sum_{k=0}^\infty \left( \sum_{\ell=1}^q \eins( k \ge r_\ell ) b_{n,k-r_\ell}^{(\nu,\ell)}  \right) \epsilon_{n,i-k},
\end{equation}
$ \nu \in  \{1, \ldots, d_n\} $. If we assume that the elements $ b_{nj}^{(\nu,\ell)}  $ of the coefficient matrices $ \matB_{nj} $ satisfy the decay condition 
\[
  \max_{\nu \ge 1} \left|  b_{nj}^{(\nu,\ell)} \right| \stackrel{n}{\ll} (j+r_\ell)^{-3/4-\theta/2},
\]
then  the coefficients $ c_{nk}^{(Z,\nu)} =  \sum_{\ell=1}^q \eins( k \ge r_\ell ) b_{n,k-r_\ell}^{(\nu,\ell)}  $ of the series (\ref{MultivLinProc})  satisfy $ \sup_{n \ge 1} \max_{\nu \ge 1} | c_{nk}^{(Z,\nu)} | \ll k^{-3/4-\theta/2}  $, i.e.,  Assumption (D) holds. In this construction the lags $ r_1, \ldots, r_q $ used to define the $ q $-variate innovation process may depend on $ n $. 

We may go beyond the above near white noise $q$-variate innovations and consider $ d_n $-dimensional linear processes with mean zero innovations $ \vece_{ni} $, $ i \ge 1 $, with a  covariance matrix close to some $ \matV > 0 $: Let 
\begin{equation}
\label{MultivLinProc2}
\vecZ_{ni} = \sum_{j=0}^\infty \matB_{nj} \matP \vece_{n,i-j}, \qquad \vece_{ni} = \matV^{1/2} \bfeps_{ni}, \qquad i \ge 1, n \in \N, 
\end{equation}
where 
\begin{equation}
\label{DefInnovationsMulti}
\bfeps_{ni} = ( \epsilon_{n,i-r_1}, \ldots, \epsilon_{n,i-r_{d_n}} )^\top,  \qquad i \ge 1, n \in \N,
\end{equation} 
for $ 0 = r_1 < \cdots < r_{d_n} $,
$ \matP $ is a full rank $ q \times d_n $ matrix and $ \matB_{nj} $ are $ d_n \times q $ coefficient matrices as above, i.e., with elements satisfying the decay condition. $ \matP $ is used to reduce the dimensionality. Let
$
\matP \matV^{1/2} = \sum_{i=1}^q \pi_{ni} \vecl_{ni} \vecr_{ni}^\top
$
be the singular value decomposition of $ \matP \matV^{1/2}  $ with singular values $ \pi_{ni} $, left singular vectors $ \vecl_{ni} \in \R^q $ and right singular vectors $ \vecr_{ni} = ( r_{ni1}, \ldots, r_{nid_n} )^\top \in \R^{d} $ satisfying $ \| \vecl_{ni} \|_{\ell_2} = \| \vecr_{ni} \|_{\ell_2} = 1 $, $ i \in \{1, \ldots, d_n\} $, $ n \ge 1 $.  Then
$
\matB_{nj} \matP \matV^{1/2} = \sum_{i=1}^q \pi_{ni} \matB_{nj} \vecl_{ni} \vecr_{ni}^\top,
$
and the element at position $ (\nu, \ell) $ of the latter matrix is given by 
$
\sum_{i=1}^q \pi_{ni} \vecb_{nj,\nu}^\top \vecl_{ni} r_{ni \ell}  $ which is $ \stackrel{\nu}{\ll} j^{-3/4-\theta/2} $ if the eigenvalues and eigenvectors are bounded.
Therefore, the class of processes (\ref{MultivLinProc2}) is a special case of (\ref{ModelYCoord}). 

The case $ q = q_n \to \infty $, especially $ q = d_n $ leading to the usual definition of a $d_n$-dimensional linear process, can be allowed for when imposing the  conditions
\begin{equation}
\label{StrongerDecayConstruction}
   \max_{\nu, \mu \ge 1} \left|  b_{nj}^{(\nu,\mu)} \right| \stackrel{n}{\ll} (j+2 r_\ell)^{-3/2 -\varpi -\theta} (\nu \mu)^{\varpi(-3/2-\theta)}
  \qquad \text{and} \qquad \sup_{n \ge 1} \sum_{\ell=1}^\infty r_\ell^{-3/4-\theta/2} < \infty
\end{equation}
with $ \varpi = 0 $ and assuming that the operators $ \matB_{nj} \matP \matV^{1/2} $, $ n \ge 1 $, are trace class operators in the sense that $ \sum_i |\pi_{ni}| = O(1) $, with eigenvectors satisfying $ \| \vecl_{ni} \|_{\ell_1}, \| \vecr_{ni} \|_{\ell_1} \stackrel{n,i}{\ll} 1  $. For $ q = d_n $ we let $ \matP = \matid $ such that $ \vecl_{ni} = \vecr_{ni} $ are the eigenvectors and $ \pi_{ni} $ the eigenvalues of $ \matV $. 
Then 
$
 \sup_{n \ge 1} \max_{\nu \ge 1} \left| c_{nj}^{(Z,\nu)} \right| \ll j^{-3/4-\theta/2} \sum_{\ell=1}^\infty r_\ell^{-3/4-\theta/2}  \ll j^{-3/4-\theta/2} 
$ 
verifying (D), as shown in the appendix. The $ \ell_1 $ constraint on the eigenvectors can be omitted when imposing the stronger condition $ \sup_{\nu \ge 1} \sum_{j=1}^{d_n} | b_{n,k-r_\ell}^{(\nu, j)} |^2 \stackrel{n}{\ll} (k+r_\ell)^{-3-\theta} $ on the coefficient matrices. For details see the appendix.

{\em VARMA Models:} Let us consider a $d_n$-dimensional zero mean VARMA($p,r$) process  
\[
  \vecY_{ni} = \matA_{n1} \vecY_{n,i-1} + \ldots + \matA_{np} \vecY_{n,i-p} + \matM_{n1} \bfeps_{n,i-1} + \ldots + \matM_{nr} \bfeps_{n,i-r} + \bfeps_{ni},
\]
with colored $d_n$-variate innovations as in (\ref{MultivLinProc2}). $ \matA_{n1}, \ldots, \matA_{np} $ and $ \matM_{n1}, \ldots, \matM_{nr} $ are $(d_n \times d_n) $ coefficient matrices. Let us assume that each of these coefficient matrices  satisfies
(\ref{StrongerDecayConstruction}) with $ \varpi = 1 $
for some $ \delta > 0$, when denoting its elements by $ b_{nj}^{(\nu,\ell)} $, $ 1 \le \nu, \ell \le d_n $.
Recall that the process is stable, if $ \text{det}( \matid_{d_n} - \matA_{n1} z - \ldots - \matA_{np} z^p ) \not= 0 $ for $ |z| \le 1 $. Then the operator $ \matA(L) = \matid_{d_n} - \sum_{j=1}^p \matA_{nj} L^j $, where $L$ denotes the lag operator, is invertible, the coefficient matrices, $ \matD_{nj} $, of $ \boldsymbol{\Psi}(L) = \matA(L)^{-1} $ are absolutely summable, and one obtains the MA representation 
$
	\vecY_{ni} = \left(\sum_{j=0}^\infty \matD_{nj} L^j \right)\left( \sum_{k=1}^r \matM_{nk} L^k \right) \bfeps_{n,i-j} = \sum_{j=0}^\infty \boldsymbol{\Phi}_{nj} \bfeps_{n,i-j}.
$  
As well known, the coefficient matrices, $ \boldsymbol{\Phi}_{nj}  $,  can be calculated using the recursion
$
  \boldsymbol{\Phi}_{n0} = \matid_{d_n},  \boldsymbol{\Phi}_{nj} = \matM_{nj} + \sum_{k=1}^j \matA_{nk} \boldsymbol{\Phi}_{n,j-k},  j \ge 1,
$
where $ \matM_{nj} = \vecnull $ for $ j > q $. Using these formulas one can show that the coefficient matrices, $ \boldsymbol{\Phi}_{nj}  $, of the MA representation satisfy (\ref{StrongerDecayConstruction}) when denoting its elements by $ b_{nj}^{(\nu,\ell)} $, and therefore the VARMA coordinate processes $ Y_{ni}^{(\nu)} $, $ 1 \le \nu \le d_n $, with innovations (\ref{DefInnovationsMulti}) satisfy the decay condition (D). 

Another interesting class of time series to be studied in future work are factor models, which are of substantial interest in econometrics. For detection of changes resp. breaks we refer to \cite{BreitungEickmeier2011}, \cite{HanInoue2015} and \cite{HorvathRice2019}, amongst others.

\subsection{Change-Point Model and Procedures}

The change-point model studied in this paper considers a change of the coefficients defining the linear processes. Nevertheless, all procedures neither require their knowledge nor their estimation. 
So let $ \barr = \{ b_{nj}^{(\nu)} : j \ge 0, \nu \in \{1, \ldots, d_n \}, n \ge 1  \} $ and $ \carr = \{ c_{nj}^{(\nu)} : j \ge 0, \nu \in \{ 1, \ldots, d_n\}, n \ge 1  \} $ be two different coefficient arrays satisfying the decay assumption and put
\begin{align*}
\bfSigma_{n0} &= \bfSigma_{n}( \barr ) = \Var( \vecY_{n}( \barr ) ), \qquad
\bfSigma_{n1} = \bfSigma_{n}( \carr ) = \Var( \vecY_{n}( \carr ) ).
\end{align*}
It is further assumed that $ \barr $ and $ \carr $ are such that 
\begin{equation}
\label{ChangeCond1}
\barr \not= \carr  \Rightarrow \bfSigma_{n0} \not= \bfSigma_{n1}, \qquad n \in \N.
\end{equation}

We will study CUSUM type procedures based on quadratic and bilinear forms of sample analogs of those variance-covariance matrices, in order to detect a change from $  	\bfSigma_{n0}  $ to $ \bfSigma_{n1}  $.
Let $ \mathcal{V}_n = \{ (\vecx_n, \vecy_n) \in \R^{d_n} \times \R^{d_n} :  \vecx_n^\top \bfSigma_{n0} \vecy_n \not= \vecx_n^\top \bfSigma_{n1} \vecy_n \}. $ Assumption (\ref{ChangeCond1}) ensures that $ \mathcal{V}_n \not= \emptyset $.

The change-point model for the high-dimensional time series is now as follows. For some change-point $ \tau \in \{ 1, \ldots, n \} $ it holds
\begin{equation}
\label{CPM}
\vecY_{ni} = \vecY_{ni}( \barr ) \eins( i \le \tau ) + \vecY_{ni}( \carr ) \eins( i > \tau ),
\qquad 1 \le i \le n, 
\end{equation}
with underlying error terms $ \epsilon_{ni} $, $ 1 \le i \le n $,  $ n \ge 1 $, satisfying Assumption (E). Our results on estimation of $ \tau $, however, assume that the change occurs after a certain fraction of the sample by requiring that
\begin{equation}
\label{CPAssumption}
\tau = \trunc{n \vartheta},
\end{equation}
for some $ \vartheta \in (0,1) $. We are interested in testing the change-point problem
\[
H_0: \tau = n \qquad \text{versus} \qquad H_1 : \tau < n,
\]
which implies a change in the second moment structure of the vector time series $ \vecY_{n1}, \ldots, \vecY_{nn} $ when $ H_1 $ holds,
and in estimation of the change-point to locate the change. Under the null hypothesis the covariance matrix of $ \vecY_{ni} $, $1 \le i \le n $, is given by
$
\Var( \vecY_{n1}( \barr ) ) =  \bfSigma_{n0}, %= \bfSigma_{n0}( \barr ),
$
whereas it changes unter the alternative hypothesis from $ \bfSigma_{n0} $ to
$
\Var( \vecY_{n,\tau+1}( \carr ) ) = \bfSigma_{n1}. % = \bfSigma_{n1}( \carr ).
$
If $ ( \vecv_n, \vecw_n) \in \calV_n $, then the change is present in the sequence of the associated quadratic forms, 
$
\sigma_{n}^2[k] = \vecv_n^\top \Var( \vecY_{nk} ) \vecw_n,  1 \le k,
$
which change  from  $  \vecv_n^\top \bfSigma_{n0} \vecw_n $ to $  \vecv_n^\top \bfSigma_{n1} \vecw_n $ if $ \tau < n $, and the change-point test below will be based on an estimator of that bilinear form. A natural condition to ensure that this relationship holds asymptotically, yielding consistency of the proposed test, is 
\begin{equation}
\label{ThereisAChangeCondition}
\inf_{n \ge 1} | \Delta_n | > 0, \qquad \Delta_n :=  \vecv_n^\top \bfSigma_{n0} \vecw_n - \vecv_n^\top \bfSigma_{n1} \vecw_n.
\end{equation}
We shall, however, also discuss in Section~\ref{SubSecStrongApprox} more general conditions for the detectability of a change.

To introduce the proposed procedures, define the  partial sums of the outer products $ \vecY_{ni} \vecY_{ni}^\top $, 
\[
\matS_{nk} = \sum_{i \le k} \vecY_{ni} \vecY_{ni}^\top , \qquad k \ge 1,
\]
such that $ k^{-1} \matS_{nk}  $ is the sample variance-covariance matrix using the data $ \vecY_{n1}, \ldots, \vecY_{nk} $.
Let 
\[
U_{nk} = \vecv_n^\top \matS_{nk} \vecw_n, \qquad k \ge 1, n \ge 1.
\]
Consider the CUSUM-type statistic,  
\[
	C_n 
	= C_n( \vecv_n, \vecw_n ) 
	= \max_{1 \le k < n} \frac{1}{\sqrt{n}  } \left| \vecv_n^\top \left( \vecS_{nk} - \frac{k}{n} \vecS_{nn} \right) \vecw_n  \right|.
\]
The large sample approximations for $ C_n$ obtained in  \cite{StelandSachs2017} under $ H_0 $, and generalized in this paper, imply that $ C_n $ can be approximated by a Brownian bridge process, $ B^0 $. Hence, we can reject the null hypothesis of no change at the asymyptotic level $ \alpha \in (0,1) $, if
\begin{equation}
\label{CPTestRule}
T_n >K^{-1}_{1-\alpha}, \quad  \quad T_n = T_n( \vecv_n, \vecw_n )  = \max_{1 \le k < n} \frac{1}{\wh{\alpha}_n(\barr)\sqrt{n}} \left|  \vecv_n^\top \left( \vecS_{nk} - \frac{k}{n} \vecS_{nn} \right) \vecw_n  \right|,
\end{equation}
where $ \wh{\alpha}_n( \barr ) $ is a consistent estimator for the asymtotic standard deviation $ \alpha_n(\barr) $ associated to the series  $ \vecv_n^\top \matS_{nk} \vecw_n $, and $ K^{-1}_u $ is the $ u $-quantile, $ u \in (0,1) $, of the Kolmogorov distribution function,
$K(z) = 1 - \sum_{i=1}^\infty (-1)^{i-1} \exp( - 2 i^2 z^2 )  $, $z \in \R $. One may also use a weighted CUSUM test
\[
  C_n(g) 
  = \max_{1 \le k < n} \frac{1}{\sqrt{n} g(k/n) } \left| \vecv_n^\top \left( \vecS_{nk} - \frac{k}{n} \vecS_{nn} \right) \vecw_n  \right|
\]
for some weight function $g$, whose role is to compensate for the fact that the centered cumulated sums get small near the boundaries. The results of Section~\ref{SubSecStrongApprox} provide large sample approximations for a large class of weighting functions. An attractive choice would be the weight function $ g(t) = \sqrt{t(1-t)} $, but the corresponding supremum of the standardized Brownian bridge, $ B^0(t) / \sqrt{t(1-t)} $, $ 0 < t < 1 $, is not well defined due to the law of the iterated logarithm (LIL), requiring to use Gumbel-type extreme value asymptotics, \cite{CsoergoeHorvath1997}, known to converge slowly, see, e.g., \cite{Ferger2018}. For a discussion of the class of proper weight functions ensuring that $ B^0(t) / g(t) $ is a.s. finite we refer to \cite{CsoergoeHorvath1993}. One may use the weight function $ [t(1-t)]^\beta $ for some $ 0 < \beta < 1/2 $ or any weight function $ g $ satisfying
\begin{equation}
\label{AssumptionWeightFunction}
 %\text{$g$ continuous with } 
 g(t) \ge C_g [ t(1-t) ]^\beta, \ 0 \le t \le 1, \ \  \text{for some constant $ C_g $.}
\end{equation}
Therefore, one rejects the no-change null hypothesis, if 
\begin{equation} 
	\label{WeightedTest}
	T_n(g) > q_g(1-\alpha),
\end{equation}
where $ q_g $ denotes the quantile function of the law of $ \sup_{0<t<1} |B^0(t)|/g(t) $.   As studied in \cite{Ferger2018}, one may also standardize the unweighted CUSUM statistic by its maximizing point, i.e., substitute $ g(k/n) $ by $ \sqrt{ \wh{\tau}_n(1-\wh{\tau}_n) } $. The associated Brownian bridge standardized by its argmax attains a density wich has been explicitly calculated in \cite{Ferger2018}.

When the assumption that the vector time series has mean zero is in doubt, one may modify the above procedures by taking the cumulated outer products of the centered series, $ \wt\vecS_{nk} = \sum_{i \le k} (\vecY_{ni} - \overline{\vecY}_{n} ) (\vecY_{ni}-\overline{\vecY}_{n})^\top $, where $ \overline{\vecY}_n = \frac{1}{n} \sum_{i=1}^n \vecY_{ni} $. The associated weigthed CUSUM statistics are then given by 
\[
\wt{C}_n(g) = \max_{1 \le k < n} \frac{1}{\sqrt{n} g(k/n) } \left| \vecv_n^\top \left( \wt\vecS_{nk} - \frac{k}{n} \wt\vecS_{nn} \right) \vecw_n  \right|, \quad \wt T_n(g) =  \frac{\wt{C}_n(g)}{\wh{\alpha}_n(\barr) },
\]
and the null hypothesis is rejected using the rule (\ref{WeightedTest}) with $ T_n(g) $ replaced by $ \wt T_n(g) $.

To estimate the unkown change-point $ \tau $, we propose to use the estimator
\[
	\wh{\tau}_n = \argmax_{1 \le k < n} \frac{1}{g(k/n) n } \left|  \vecv_n^\top \left( \vecS_{nk} - \frac{k}{n} \vecS_{nn} \right) \vecw_n   \right|.
\]
Based on the estimator $ \wh{\tau}_n $ of the change-point, one may also estimate the nuisance parameter $ \alpha_n^2( \barr ) $ by $ \wh{\alpha}_{\wh{\tau}_n}^2( \barr ) $. 

For $L$ pairs of projection vectors $ \vecv_{nj}, \vecw_{nj} $, $ j \in  \{1, \dots, L\} $, consider the associated CUSUM transform
\[
  \vecC_n = \left( C_n( \vecv_{nj}, \vecw_{nj} ) \right)_{j=1}^L, \qquad \vecT_n = \left( T_n( \vecv_{nj}, \vecw_{nj} ) \right)_{j=1}^L.
\]
Observe that this transform differs from the  transform studied in \cite{WangSamworth2018}, where  the statistics are calculated coordinate-wise  and the transform is given by the corresponding $d $ CUSUM trajectories. 

We wish to test the null hypothesis of no change w.r.t. to $ \{ \vecv_n, \vecw_n \} $ 
\[
  H_0: \vecv_{nj}^\top \Var( \vecY_{n\tau} ) \vecw_{nj} = \vecv_{nj}^\top \Var( \vecY_{n,\tau+1} ) \vecw_{nj}, j \in  \{1, \ldots, L\},
\]
against the alternative hypothesis that, induced by a  change at $ \tau < n $, at least one bilinear form changes (assuming the projections are appropriately selected), 
\[
  H_1 : \exists j \in \{ 1, \ldots, L \} :  \vecv_{nj}^\top \Var( \vecY_{n\tau} ) \vecw_{nj} \not= \vecv_{nj}^\top \Var( \vecY_{n,\tau+1} ) \vecw_{nj}.
\]
As a global (omnibus) test one may reject $ H_0$ at the asymptotic significance level $ \alpha $, if 
\begin{equation}
\label{TestMultiv}
  Q_n = (\vecT_n - \bfmu^*_n  )^\top (\wh{\bfSigma}^\vecB_n)^{-} (\vecT_n - \bfmu^*_n ) > q_{vm}(1-\alpha).
\end{equation}
Here $ Q_n $ is a non-standard quadratic form, as it is based on the CUSUMs instead of a multivariate statistic which is asymptotically normal,
 $ \bfmu_n^* = \left(  \max_{1 \le k < n} \EE \max_{1 \le k < n} | \overline{B}^0(k/n) / g(k/n) | \right)_{j=1}^L $, $ (\wh{\bfSigma}^\vecT_n)^{-} $ is the Moore-Penrose generalized inverse of  $ \wh{\bfSigma}^\vecT_n = \left( \wh{\beta}_n^2(j,k)  \wh{\beta}_n^{-1}( k, k ) \wh{\beta}_n^{-1}( j, j ) \right)_{1 \le j \le L \atop 1 \le k \le L} $ and $ q_{mv}(p) $ denotes the $p$-quantile of the simulated distribution of $ Q_n $ using a Monte Carlo estimate of $\EE \max_{1 \le k < n} | \overline{B}^0(k/n) / g(k/n) |  $; the estimators $ \wh{\beta}_n^2(j,k)  $ of the asymptotic covariance of the $j$th and $k$th coordinate of the CUSUM transform $ \vecC_n$ are defined in the next section, calculated from a learning sample. It is worth mentioning that the statistic $ Q_n $ can be used to test for a change in the subspace $ \text{span} \{ \vecv_{n1}, \dots, \vecv_{nL}  \} $ by putting $ \vecw_{ni}  = \vecv_{ni} $, $ i \in  \{1, \ldots, n\}  $.

\subsection{Choice of the projections}

The question arises how to choose the projection vectors $ \vecv_n, \vecw_n $. Their choice may depend on the application. Here are some examples.

\begin{example} (Change of sets of covariances as in gene expression time series) \\
	Time series gene expression studies investigate the gene expression levels of a large number of genes measured at several time points, in order to identify and analyze activated genes and their relationship in a biological process, see \cite{BarJosephGitterEtAl2012}. Going beyond the expression levels and analyzing the dependence structure of gene expression is of interest. For example, a group of genes may be uncorrelated to others or the rest of the genome, but interactions inducing correlations may start after an external stimulus. To analyze two groups, e.g., the first $p$ and the last $q$ variables, one may use 
	$ \vecv_n = q^{-1} ( \eins_q^\top, \vecnull_{d_n-q}^\top )^\top \quad \text{and} \quad \vecw_n = p^{-1}( \vecnull_{d_n-p}^\top, \eins_p^\top )^\top, $ 	corresponding to
	$ \sigma_{n}^2[k] = \frac{1}{pq} \sum_{j=1}^q \sum_{\ell=d_n-p+1}^{d_n} \Cov( Y_{nk}^{(j)}, Y_{nk}^{(\ell)} ), $
	the average covariance between the first $q$ and the last $p$ coordinates. Further, in order to compare the first $q$ variables with the remaining ones, one could use $ \vecw_n = ( \vecnull_q^\top, (1/2)^r, (1/3)^r, \ldots,1/ (d_n-q+1)^r )^\top $  with $ r > 1 $.
\end{example}

\begin{example} (Spatial clustered sensors)\\
	Suppose that the $ d_n $  observed variables represent sensors of $r$ clusters or groups, e.g., sensors spatially distributed over $r$ geographic regions such as states. Such a classification is given by a partition $ \cup_{i=1}^r \calJ_{ni} = \{ 1, \dots, d_n \}$ with pairwise disjoint sets $ \emptyset \not= \calJ_{ni} $, $ i =1, \ldots, r $. To analyze the within-region and between-region covariance structures, one may consider the orthogonal sytem given by the vectors $ \vecv_{ni} = | \calJ_{ni} |^{-1} ( \eins_{\calJ_{ni}}( \nu ) )_{\nu=1}^{d_n} $, $ i = 1, \ldots, r $. Our results allow for the case of region-wise infill asymptotics where $ | \calJ_{ni} | $ increases with the sample size. In our data example, the grouping is, however, determined by a sparse PCA instead of using geographic locations.
\end{example}

\begin{example} (Change in the autocovariance function (ACVF) of a stationary time series) \\
	Our high-dimensional time series model also allows to analyze the ACVF of a stationary time series.	
	Let $ X_i^{(n)} = \sum_{j=0}^\infty c_{nj} \epsilon_{i-j}  $ be a stationary linear time series with coefficients $ \{ c_{nj} : j \ge 0 \} $, $ n \ge 1 $, satisfying Assumption (D) and define
	\[
	Y_i^{(\nu)} = X_{n,i+\nu}, \qquad i \in \N, \nu = 1, \ldots, n - d_n.
	\]
	Then $ \vecY_{ni} = ( Y_{ni}^{(1)}, \ldots, Y_{ni}^{(d_n)} )^\top $ is a special case of model (1), % (\ref{ModelYCoord}),
	and the change-point model (9) % (\ref{CPM}) 
	analyzes a change of the coefficients of $ X_t^{(n)} $ in terms of the ACVF $ \gamma_n(h)= \EE( X_1^{(n)} X_{1+h}^{(n)} ) $ up to the lag $ d_n $ respectively a change of the ACVF due to a change of the underlying coefficients.
	Since then the sample covariance matrix consists of the sample autocovariance estimators, the proposed CUSUM tests consider a weighted averages of them and taking unit vectors for $ \vecv_n, \vecw_n $ leads to a procedure closely related to the CUSUM test studied in \cite{BerkesGombayHorvath2009}.  Changes in autocovariances have also been studied by \cite{NaLeeLee2011} from a parametric point of view and by \cite{Steland2005} and \cite{LiZhao2013} using kernel methods.
\end{example}

\begin{example} (Financial portfolio analysis) \\
	In financial portfolio optimization one is given a stationary time series of returns $ \vecY_{nt}$ of $ d_n $ assets and seeks a portfolio vector representing the number of shares to hold from each asset. The variance-minimizing portfolio $ \vecw_n^* $ is obtained by minimizing the portfolio risk $ \Var( \vecw_n^\top \vecY_n ) $ under the constraint $ \eins^\top \vecw_n = 1 $. In order to  keep transactions costs moderate, sparsity constraints can be added, see, e.g., \cite{BrodieEtAl2009} where a $ \ell_1 $-penalty term is added. For bounds and confidence intervals of the risk $ \vecw_n^*{}^\top \bfSigma_n \vecw_n^* $ of the optimal portfolio see \cite{Steland2018}. 
\end{example}

In some applications selecting them from a known basis may be the method of choice.  In low- and high-dimensional multivariate statistics it is, however, a common statistical tool to project data vectors onto a lower dimensional subspace spanned by (sparse)  directions (axes) $ \vecv_n^{(1)}, \ldots, \vecv_n^{(K)} $. These directions can be obtained from a fixed basis or by a (sparse) principal component analysis using a learning sample. The projection is determined by the new coordinates
$ \vecv_n^{(i)}{}^\top \vecY_n $, $ i \in  \{1, \ldots, K\} $, for simplicity also called projections, and represent a lower dimensional compressed approximation of $ \vecY_n $. The uncertainty of its coordinates, i.e., of its position in the subspace, can be measured by the variances $ \vecv_n^{(i)}{}^\top \bfSigma_n \vecv_n^{(i)}{} $. Clearly, it is of interest to test for the presence of a change-point in the second moment structure of these new coordinates by analyzing the bilinear forms $   \vecv_n^{(i)}{}^\top \bfSigma_n \vecv_n^{(j)}{} $, $ 1 \le i, j \le K $. Also observe that one may analyze the spectrum, since for eigenvectors $ \vecv_n^{(i)} $ the associated eigenvalue is given by $  \vecv_n^{(i)}{}^\top \bfSigma_n \vecv_n^{(i)}{}  $.

The question under which conditions PCA or sparse PCA is consistent has been studied by various authors. The classic Davis-Kahan theorem, see \cite{DavisKahan1970} and \cite{WangSamworth2015} for a statistical version, relates this to consistency of the sample covariance matrix in the Frobenius norm, which generally does not hold under high-dimensional regimes  without additional assumptions, and minimal-gap conditions on the eigenvalues. Standard PCA is known to be inconsistent, if $ d/n \to y  \in (0,\infty] $, where here and in the following discussion a possible dependence of $d$ on $n$ is suppressed. Under certain spiked covariance models consistency can be achieved, see \cite{JohnstoneLu2009} if $ d = o(n) $,  \cite{Paul2007} under the condition $ d/n \to \gamma \in (0,1) $ and \cite{JungMarron2009} for $n$ fixed and $ d \to \infty $. Sparse principal components, first formally studied by \cite{JollifeEtAl2003} using lasso techniques, are strongly motivated by data-analytic aspects, e.g., by simplifying their interpretation, since linear combinations found by PCA typically involve all variables. Consistency has been studied under different frameworks, usually assuming additional sparsity constraints on the true eigenvectors (to ensure that their support set can be identfied) and/or growth conditions on the eigenvalues (to ensure that the leading eigenvalues are dominant in the spectrum). We refer to  \cite{ShenShenMarron2013} for simple thresholding sparse PCA when  $n$ is held fixed and $ d \to \infty $, \cite{BirnbaumJohnstonNadlerPaul2013} for results on minimax rates when estimating the leading eigenvectors under $ \ell_q $-constraints on the eigenvectors and fixed eigenvalues, whereas  \cite{CaiMaWu2015} provide minimax bounds assuming at most $k$ entries of the eigenvectors are non-vanishing and \cite{WangFan2017} derives asymptotic distributions allowing for diverging eigenvalues and $ d/n \to \infty $. 

To avoid that a change is not detectable because it takes place in a subspace of the orthogonal complement of the chosen projection vectors, a simple approach used in various areas is to take random projections. For example, one may draw the projection vectors from a fixed basis or, alternatively, sample them from a distribution such as a Dirichlet distribution or an appropriately transformed Gaussian law. Random projections of such kind are also heavily used in signal processing and especially in compressed sensing, by virtue of the famous distributional version of the Johnson-Lindenstrauss theorem, see \cite{JohnsonLindenstraussJoram1984}. This theorem  states that any $n$ points in a Euclidean space can be embedded into $ O( \varepsilon^2 \log(1/\delta) ) $ dimensions such that their distances are preserved up to $ 1 \pm \varepsilon $, with probability larger than $ 1 - \delta $. This embedding can be constructed with $ \ell_0 $-sparsity $ O( \varepsilon^{-1} \log(1/\delta)) $ of the associated projection matrix, see \cite{KaneNelson2014}. 

This discussion is continued in the next section after Theorem~\ref{TestConsistency} and related to the change-point asymptotics established there. 

\section{Asymptotics}
\label{Sec:Asymptotics}

The asymptotic results comprise approximations of the CUSUM statistics and related processes by maxima of Gaussian bridge processes, consistency of Bartlett type estimators of the asymptotic covariance structure of the CUSUMs, stopped sample versions of those estimators and consistency of the proposed change-point estimator.

\subsection{Preliminaries}

To study the asymptotics of the proposed change-point test statistics both under $ H_0 $ and $ H_1 $,  we consider the two-dimensional partial sums,
\begin{equation}
\label{Def_vecU}
\vecU_{nk} = \left( \begin{array}{cc} \vecU_{nk}^{(1)} \\  \vecU_{nk}^{(2)} \end{array}   \right) 
= \sum_{i \le k} \left( \begin{array}{cc} Y_{ni}( \avb ) Y_{ni}( \awb ) \\ Y_{ni}( \avc ) Y_{ni}( \awc ) \end{array}  \right)
\end{equation}
and their centered versions,
\begin{equation}
\label{Def_vecD}
\vecD_{nk} = \vecU_{nk} - \EE( \vecU_{nk} ),
\end{equation}
for $ k, n \ge 0 $, where for brevity $ Y_{ni}( \avb ) $ and $ Y_{ni}( \avc ) $ are defined by \[ Y_{ni}( \vecz^\top \aarr ) = \sum_{j=0}^\infty \sum_{\nu=1}^{d_n} z_{n\nu} a_{nj}^{(\nu)} \epsilon_{i-j} \] for $ \vecz \in \{ \vecv, \vecw \} $ and $ \aarr \in \{ \barr, \carr \} $.

Introduce the filtrations $ \calF_{nk} = \sigma( \epsilon_{ni} : i \le k ) $, $ k \ge 1 $, $ n \in \N $. In Lemma~\ref{LemmaMartingaleApprox} it is shown that $ \vecD_{nk} $ can be approximated by a $ \calF_{nk} $-martingale array with asymptotic covariance parameter $ \beta_n^2( \barr, \carr ) $ defined in Lemma~\ref{CorollaryBetaSq}, see (\ref{FormulaBetaSq}).  Denote by $ \alpha_n^2( \aarr ) = \beta_n^2( \aarr, \aarr ) $, for $ \aarr \in \{ \barr, \carr \} $, the associated asymptotic variance parameter.

As a preparation, let $ \vecB_n(t) = (\vecB_{n}^{(1)}(t), \vecB_{n}^{(2)}(t))^\top $, $ t \ge 0 $, be a two-dimensional mean zero Brownian motion with variance-covariance matrix 
\begin{equation}
\label{AsCovMatrix}
\left( \begin{matrix}  \Var( \vecB_{n}^{(1)} ) & \Cov( \vecB_{n}^{(1)}, \vecB_{n}^{(2)} ) \\ \Cov( \vecB_{n}^{(1)}, \vecB_{n}^{(2)} ) & \Var( \vecB_{n}^{(2)} ) \end{matrix} \right) = 
\left( \begin{matrix}  \alpha_n^2( \barr) & \beta_n^2( \barr, \carr ) \\  \beta_n^2( \barr, \carr )  & \alpha_n^2( \carr) \end{matrix}  \right), \qquad n \ge 1.
\end{equation}
For $n \ge 1 $ define the Gaussian processes
\begin{align}
\label{DefGaussPro}
G_n(t) &=  \vecB_{n}^{(1)}(t) \eins( t \le \tau ) + [ \vecB_{n}^{(1)}( \tau ) + ( \vecB_{n}^{(2)}(t) - \vecB_{n}^{(2)}(\tau) ) ] \eins( t > \tau ), \qquad t \ge 0, \\
\nonumber
\overline{G}_n( t ) &= \frac{1}{\sqrt{n}} G_n( t n ), \qquad t \in [0,1].
\end{align}
Before the change, $ G_n $ is the Brownian motion $ \vecB_{n}^{(1)} $ with variance $ \alpha_n^2( \barr )  $ and after the change it behaves as the Brownian motion $ \vecB_{n}^{(2)} $ with start in $ \vecB_{n}^{(1)}( \tau ) $ and variance $ \alpha_n^2( \carr) $. 
Further define
\begin{align*}
G_n^0( k ) &= G_n( k ) - \frac{k}{n} G_n( n ), \qquad k \le n, n \ge 1, \\
\overline{G}_n^0(t) &= \overline{G}_n( t ) - t \overline{G}_n(1), \qquad t \in [0,1].  
\end{align*}
As shown in the appendix, it holds
\[
\Cov( G_{n}(s), G_n(t) ) = \left\{ \begin{array}{ll} \min(s,t) \alpha_n^2( \barr ), &  s,t \le \tau $ or $ s \le \tau < t, \\
 \min(s-\tau, t-\tau) \alpha_n^2( \carr ), &  \tau \le s, t. \end{array} \right. 
\]
\begin{equation}
\label{Cov_Gn0_prechange}
  \Cov( G_n^0(s), G_n^0(t) )= \left\{ \begin{array}{cc} 
  \left(  \min( s, t ) - \frac{st}{n} \right)  \alpha_n^2( \barr ), &  s,t \le \tau $ or $ s \le \tau < t, \\
  \left(  \min( s - \tau, t - \tau ) - \frac{st}{n} \right)  \alpha_n^2( \carr ), & \tau \le s, t.
  \end{array}\right.
\end{equation}

\subsection{Change-point Gaussian approximations}
\label{SubSecStrongApprox}

Closely related to the CUSUM procedures are the following \cadlag processes: Define
\[
	\calD_n(t) = n^{-1/2} \vecv_n^\top ( \matS_{n,\trunc{nt}} - \trunc{nt} \EE  (\matS_{nn}) ) \vecw_n, \qquad t \in [0,1], n \ge 1,
\]
and the introduce the associated bridge process
\[
\calD_n^0(t) = \calD_n\left( \frac{\trunc{nt}}{n} \right) - \frac{\trunc{nt}}{n} \calD_n(1), \qquad t \in [0,1].
\]
Observe that its expectation is 
$
\EE (\calD_n^0( k/n )) = \frac{1}{\sqrt{n}} \left( \sum_{i=1}^k \sigma_n^2[i] - \frac{k}{n} \sum_{i=1}^n \sigma_n^2[i]  \right),
$
and vanishes, if $ \sigma_n^2[1] = \cdots = \sigma_n^2[n] $. But a non-constant series $ \sigma_n^2[i] $, $ i \in \{1, \ldots, n\} $, may lead to $\EE (\calD_n^0( k/n ))  \not= 0 $. This particularly holds for the change-point model.
Our results show that $ \calD_n(t) $ ($ \calD_n^0(t) $) can be approximated by a Brownian (bridge) process and lead to a FCLT under weak regularity conditions, and the same holds true for weighted version of theses \cadlag processes for nice weighting functions $g$. 

Define for $  k \ge 1, n \ge 1, $ 
\begin{align*}
	U_{nk} & = \vecv_n^\top \matS_{nk} \vecw_n, \\
	D_{nk} &= U_{nk} - \EE( U_{nk} )  = \vecv_n^\top ( \matS_{nk} - \EE  (\matS_{nk}) ) \vecw_n,
\end{align*}
and
\[
m_n(k) := \EE \left( U_{nk} - \frac{k}{n} U_{nn} \right)
= \left\{
\begin{array}{cc}
\frac{k(n-\tau)}{n} \Delta_n, & \qquad k \le \tau, \\
\tau \frac{n-k}{n} \Delta_n, & \qquad k > \tau.
\end{array}
\right.
\]

The following theorem extends the results of \cite{StelandSachs2017,StelandSachs2018} and justifies the proposed tests (\ref{CPTestRule}) and (\ref{WeightedTest}) when combined with the results of the next section on consistency of the asymptotic variance parameters. 
This and all subsequent results consider the basic time series model (\ref{ModelYCoord}), but all results hold for the multivariate linear processes and VARMA models introduced in Section~2  under the conditions discussed there. 

\begin{theorem}
	\label{BasicStrongApprox}
	Suppose that $ \{ \epsilon_{ni} : i \in \Z, n \ge 1 \} $ satisfies Assumption (E). 
	Let $ \vecv_n, \vecw_n $ be weighting vectors with  $ \ell_1 $-norms  satisfying
	\begin{equation}
	\label{AssumptionL1Projections}
	\| \vecv_n \|_{\ell_1}  \| \vecw_n \|_{\ell_1} = O( n^{\eta} ), \qquad \text{for $ 0 \le \eta \le (\theta-\theta')/4$ for some $ 0 < \theta' < \theta $},
	\end{equation} 
	and let $ \barr = \{ b_{nj}^{(\nu)} \} $ and $ \carr = \{ c_{nj}^{(\nu)} \} $  be coefficients satisfying Assumption (D). If the change-point model (\ref{CPM}) holds, then, for each $n$,  one may redefine, on a new probability space, the vector time series together with a two-dimensional mean zero Brownian motion $ \{ \vecB_n(t) : t \in [0,1] \} $ with coordinates $ \vecB_{n}^{(i)}(t) $, $ t \in [0,1] $, $ i = 1, 2 $, characterized by the covariance matrix (\ref{AsCovMatrix}) associated to the parameters $ \alpha_n^2( \barr ), \alpha_n^2( \carr ) $, assumed to be bounded away from zero, and $ \beta_n^2( \barr, \carr ) $, such that for some constant $C_n $ the following assertions hold true almost surely:
	\begin{itemize}
		\item[{\rm (i)}] $ \| \vecD_{nt} - \vecB_n(t) \|_2 \le C_n t^{1/2-\lambda} $, $ t > 0 $.
		\item[{\rm (ii)}] $ \max_{1 \le k < n} \| \vecD_{nk} - \frac{k}{n} \vecD_{nn} - [ \vecB_{nk} - \frac{k}{n} \vecB_n(n) ]  \|_2 \le 2 C_n n^{1/2-\lambda} $, $ n \ge 1 $.
		\item[{\rm (iii)}] $ \max_{1 \le k < n} \frac{1}{\sqrt{n}}  | D_{nk} - \frac{k}{n} D_{nn} - G_n^0(k) | \le   6 \sqrt{2} C_n n^{-\lambda} $, $ n \ge 1 $.
		\item[{\rm (iv)}] $ \left| \max_{1 \le k < n} \frac{1}{\sqrt{n}}  | D_{nk} - \frac{k}{n} D_{nn} | - | \max_{1 \le k < n}  \frac{1}{\sqrt{n}}  | G_n^0(k) | \right|  \le 6 \sqrt{2} C_n n^{-\lambda} $, $ n \ge 1 $.
		\item[{\rm (v) }] $ \max_{1 \le k < n} \frac{1}{\sqrt{n}}  | U_{nk} - \frac{k}{n} U_{nn} - [ m_n(k) + G_n^0(k)] | \le  6 \sqrt{2} C_n n^{-\lambda} $, $ n \ge 1 $.
		\item[{\rm (vi) }] $ \left|  \max_{1 \le k < n} \frac{1}{\sqrt{n}}  | U_{nk} - \frac{k}{n} U_{nn} | - \max_{1 \le k < n} \frac{1}{\sqrt{n}} | m_n(k) + G_n^0(k) | \right| \le 6 \sqrt{2} C_n n^{-\lambda} $, $ n \ge 1 $.
	\end{itemize}
	If $ C_n n^{-\lambda} = o(1) $, then we also have
	\begin{itemize}
		\item[{\rm (vii)}] $ \sup_{t \in [0,1]} \left| \calD_n(t) - [\mu_n(t) + \overline{G}_n( \trunc{nt}/n ) ] \right| = o(1) $, a.s., as $ n \to \infty $,
		\item[{\rm (viii)}] $ \sup_{t \in [0,1]} \left| \calD_n^0(t) - [\mu_n(t) + \overline{G}_n^0( \trunc{nt}/n ) ] \right| = o(1) $, a.s., as $ n \to \infty $,
	\end{itemize}
	where $ \mu_n(t) = \trunc{nt}/{n} (1-\tau/{n}) \Delta_n \eins( t \le \tau/n ) + \tau/n (1- \trunc{nt}/n) \Delta_n \eins( t > \tau/n) $, $ t \in [0,1] $.
	Further, provided the weight function $g$ satisfies (\ref{AssumptionWeightFunction}), the corresponding above assertions hold in probability, if $ C_n n^{-\lambda} = o(1) $. Especially, 
	\begin{equation}
		\label{MainForWeightedCUSUM1}
	  \max_{1 \le k < n} \frac{1}{\sqrt{n} g(k/n)}  \left| U_{nk} - \frac{k}{n} U_{nn} - \left[  m_n(k) + G_n^0(k) \right]  \right| = o_{\PP}(1)
	\end{equation}
	and
	\begin{equation}
	\label{MainForWeightedCUSUM}
	  \left|  \max_{1 \le k < n} \frac{1}{\sqrt{n} g(k/n)}  \left| U_{nk} - \frac{k}{n} U_{nn} \right| - \max_{1 \le k < n} \frac{1}{\sqrt{n}g(k/n)} \left| m_n(k) + G_n^0(k) \right| \right| = o_{\PP}(1).
	\end{equation}
\end{theorem}

\begin{remark} {\rm Provided  the original probability space,  $ (\Omega, \calA, \PP) $,  is rich enough to carry an additional uniform random variable, the strong approximation results of Theorem~\ref{BasicStrongApprox} can be constructed on  $ (\Omega, \calA, \PP) $.}
\end{remark}

When there is a change, the drift term $ m_n $ yields the consistency of the test. 

\begin{theorem}
	\label{TestConsistency}
	Under the assumptions of Theorem~\ref{BasicStrongApprox} and  (\ref{ThereisAChangeCondition}), 
	$
	\max_{1 \le k < n} \frac{1}{\sqrt{n}} \left| U_{nk} - \frac{k}{n} U_{nn} \right| \to \infty,  n \stackrel{\PP}{\to} \infty. $
\end{theorem}

Note that Theorem~\ref{BasicStrongApprox} holds without the conditions (\ref{CPAssumption}) and (\ref{ThereisAChangeCondition}). To discuss conditions of detectability of a change, observe that the drift of the approximating Gaussian process in (\ref{MainForWeightedCUSUM1}) is given by
\[
   H_n( k/n ) = H_n(k/n;  \tau/n, \Delta_n,  g ) = \sqrt{n} \Delta_n \left[ \frac{k(n-\tau) }{  n^2 g(k/n) } \eins( k \le \tau )
  + \tau \frac{(n-k)}{ n^2 g(k/n)} \eins(k > \tau) \right]. 
\]
If this function is asymptotically constant, especially if $ \Delta_n \not= 0 $ for all $n$ but $ \sqrt{n} \Delta_n = o(1) $ (which implies $ |\alpha_n^2(\barr) - \alpha_n^2(\carr) | = o(1) $ by (\ref{ChangeCond1}) and Lemma~\ref{CorollaryBetaSq}) and $ \tau /n \to \vartheta \in (0,1) $, then the change is asymptotically not detectable, since the asymptotic law is the same as under the null hypothesis.  Now assume $ \tau/n \to \vartheta $. A change $ \vartheta $ located in a measurable set $ A \subset (0,1) $ with positive Lebesgue measure is detectable and changes the asymptotic law, if
$
  H_n \to h^*,  n \to \infty,
$
for some function $h^* \not= 0$ on $A$, since then the asymptotic law is given by $ \sup_{0<t<1} | [h^*(t) + B^0(t)]/g(t) | $, or if
$
	 H_n \stackrel{\PP}{\to} \infty,  n \to \infty,  on\ [\vartheta,1),
$
cf. Theorem~\ref{TestConsistency}. The  case $ H_n \to h^* $ corresponds to a local alternative such as $ \bfSigma_{n1} = \bfSigma_{n0} + \boldsymbol{\Delta}_n / \sqrt{n} $ for some $ d_n \times d_n $ matrix $ \boldsymbol{\Delta}_n $ such that $ \lim_{n \to \infty}  \vecv_n^\top \boldsymbol{\Delta}_n  \vecw_n $ exists. For example, if in the spiked covariance model (\ref{SpikedCovModel}) a new local spike term of the form $ n^{-1/2} \lambda_{r+1} \vecu_{n,r+1} $ appears after the change-point, then $ \boldsymbol{\Delta}_n  =  \lambda_{r+1} \vecu_{n,r+1} $ and $ \Delta_n = \lambda_{r+1} \vecv_n^\top \vecu_{n,r+1} \vecw_n^\top \vecu_{n,r+1} $.  Condition (\ref{ThereisAChangeCondition}) is then satisfied, if the weighting vectors are not asymptotically orthogonal to the direction of the new spike. 

Observe that $ H_n( \cdot, \tau, \Delta_n; g ) $ is linear in $ \Delta_n = \vecv_n^\top( \bfSigma_{n0} - \bfSigma_{n1}) \vecw_n $. Clearly, $ |\Delta_n| $ is maximized if $ \vecv_n = \vecw_n $ is a leading eigenvector of  $ \bfSigma_{n0} - \bfSigma_{n1} $. This can be seen from the spectral decomposition $ \boldsymbol{\Delta}_n = \sum_{i=1}^s \phi_{ni} \boldsymbol{\delta}_{ni} \boldsymbol{\delta}_{ni}^\top $, where $ \boldsymbol{\delta}_{ni} $ are the eigenvectors and $ \phi_{ni} $ the eigenvalues. When there is no knowledge about the change, e.g., in terms of the $ \phi_{ni} $ and/or $ \boldsymbol{\delta}_{ni}  $ or in terms of the model coefficients $ c_{nj}^{(\nu)} $, it makes sense to select $ \vecv_n, \vecw_n  $ from a known basis or as leading (sparse) eigenvectors of $ \bfSigma_{n0} $, estimated from a learning sample, in order to obtain a procedure which is capable to react, if the dominant part of the eigenstructure of the covariance matrix changes. Clearly, a change in the orthogonal complement of chosen projection vectors is not detectable. This can be avoided by considering, in addition, random projection(s). 

For the CUSUM statistics based on the centered time series we have the following approximation result.

\begin{theorem} 
\label{BasicStrongApproxCentered}
	Let the original probability space be rich enough to carry an additional uniform random variable. Assume the conditions of  Theorem~\ref{BasicStrongApprox} and the strengthended decay condition $ \sup_{n \ge 1} \max_{1 \le \nu \le d_n} | c_{nj}^{(\nu)} | \ll (j \vee 1)^{-1-\theta} $ for some $ \theta > 0 $ hold.
	 Suppose that  the vector time series is centered at the sample averages $ \wh{\mu}_\nu = \frac{1}{n} \sum_{i=1}^n Y_{ni}^{(\nu)} $, before applying the CUSUM procedures, leading to the statistics $ \wt{C}_n(g) $ and $ \wt{T}_n(g) $.
	Then assertions (i) and (ii) of Theorem~\ref{BasicStrongApprox} hold true with an additional error term $  o_{\PP}( n^{1/2} )  $ and (iii)-(vi) with an additional $ o_{\PP}(1) $ term. Finally, (vii) and (viii) hold in probability, if $ C_n n^{-\lambda} = o(1) $. 
\end{theorem}

The above theorems assume that the projection vectors $ \vecv_n $ and $ \vecw_n $ have uniformly bounded  $ \ell_1 $-norm. When standardizing by a homogenous estimator  $ \wh{\alpha}_n = \wh{\alpha}_n( \vecv_n, \vecw_n ) $, i.e. satisfying
\begin{equation}
	\label{ScalingProp}
	\wh{\alpha}_n( x \vecv_n, y \vecw_n ) = xy \wh{\alpha}_n( \vecv_n, \vecw_n ) 
\end{equation}
for all $ x, y > 0 $, one can relax the conditions on the projections $ \vecv_n, \vecw_n $.

\begin{theorem} 
	\label{ApproxGrowing}
	Suppose that $ \{ \epsilon_{ni} : i \in \Z, n \ge 1 \} $ satisfies Assumption (E). Assume that 
	\begin{equation}
		\label{AssumptionL2Projections}
		\sup_{n \ge 1} d_n^{-1/2} \| \vecv_n \|_{\ell_2}, 	\sup_{n \ge 1} d_n^{-1/2} \| \vecw_n \|_{\ell_2} < \infty
	\end{equation} 
	or	there are non-decreasing sequences $ \{ a_n \}, \{ b_n \} \subset (0,\infty ) $ with
	\begin{equation}
	\label{AssumptionL1Growing}
	\sup_{n \ge 1} a_n^{-1} \| \vecv_n \|_{\ell_1}, \sup_{n \ge 1} b_n^{-1} \| \vecw_n \|_{\ell_1} < \infty.
	\end{equation}
	Suppose that the estimator $ \wh{\alpha}_n = \wh{\alpha}_n( \vecv_n, \vecw_n ) $ used by $ T_n(g; \vecv_n, \vecw_n ) $ is ratio consistent and homogenous.
	Further, let $ \barr = \{ b_{nj}^{(\nu)} \} $ and $ \carr = \{ c_{nj}^{(\nu)} \} $  be coefficients satisfying Assumption (D). If the change-point model (\ref{CPM}) holds,
	then, under the construction of Theorem~\ref{BasicStrongApprox} with $ C_n n^{-\lambda} = o(1) $, (vi) holds and we have for any weight function $g $ satisfying (\ref{AssumptionWeightFunction})
	\begin{equation}
		\label{MainForWeightedCUSUM2}
		\left| T_n(g) - \max_{1 \le k < n} \frac{1}{g(k/n)} \bigg| \frac{m_n(k)}{\sqrt{n}}+ \overline{B}_n^0(k/n) \biggr| \right| = o_{\PP}(1),
	\end{equation}	
	where $ \overline{B}_n^0(t) = \alpha_n^{-1}(\barr) \overline{G}_n^0(t) $, $ t \in [0,1] $.
\end{theorem}

By Theorem~\ref{BasicStrongApprox}, statistical properties of the CUSUM statistic $ C_n(g) $ can be approximated by those of  $ \max_{1 \le k < n} \frac{1}{\sqrt{n} g(k/n)} \left| m_n(k) + G_n^0(k) \right| $. In view of Theorem~\ref{ApproxGrowing}, for the standardized CUSUM statistic $ T_n(g) $ one replaces $ G_n^0 $ by a process which is a Browninan bridge with covariance function $ \min(s,t) -st$ up to $ \tau $ and  $ (\min(s,t) -st) \alpha_n^2(\carr) / \alpha_n^2( \barr) $ after the change. Especially, under the null hypothesis $ H_0 $ of no change, we have $ m_n(k) = 0 $, for all $k$ and $n$, and $ | \alpha^2_n( \barr) - \alpha^2( \carr ) | =o(1) $ by (\ref{ChangeCond1}) and Lemma~\ref{CorollaryBetaSq}. Then the asymptotics of the change-point procedures is governed by a standard Brownian bridge. 
Theorems~\ref{BasicStrongApprox}, \ref{ApproxGrowing} and  \ref{BasicStrongApproxCentered} (under the strenghtened decay condition) imply FCLTs.

\begin{theorem}
\label{FCLT}
	 (FCLT) 
	 If $ \beta_n^2( \barr, \carr ) \to \beta^2( \barr, \carr ) $, $ \alpha_n^2(\aarr) \to \alpha^2(\aarr) >0 $ for $ \aarr \in \{ \barr, \carr \} $, $ \Delta_n \to \Delta > 0 $ and $ \tau/n \to \vartheta \in (0,1) $, as $ n \to \infty $, then under the conditions of Theorem~\ref{BasicStrongApprox} (viii) or Theorem~\ref{ApproxGrowing} it holds
	\[
	  \calD_n^0 \Rightarrow \mu + \overline{G}^0, \qquad n \to \infty,
	\]
	with $ \mu(t) = t(1-\vartheta) \Delta \eins( t \le \vartheta ) + \vartheta(1-t) \Delta \eins( t > \vartheta) $, $ t \in [0,1] $, 
	in the Skorohod space $ D[0,1] $, for some Gaussian bridge process $ \overline{G}^0 $ defined on  $[0,1] $ with $ \Cov( \overline{G}^0(s), \overline{G}^0(t) ) = ( \min(s,t) - st) \alpha^2( \barr) $if $ s, t \le \vartheta $ or $ s \le \vartheta < t $,
	and $ \Cov( \overline{G}^0(s), \overline{G}^0(t) ) = ( \min(s-\vartheta, t-\vartheta) - st ) \alpha^2(\carr) $, if $ \tau \le s, t $.
	Further, if $ \vecv_n, \vecw_n $ are weighting vectors satsfying (\ref{AssumptionL1Projections}), (\ref{AssumptionL2Projections}) or (\ref{AssumptionL1Growing}) and if the constructions of Theorem~\ref{BasicStrongApprox} and Theorem~\ref{ApproxGrowing}, respectively, hold with $ C_n n^{-\lambda}  = o(1)$, 
	then for any weight function $g$ which satisfies (\ref{AssumptionWeightFunction}) we have 
	\[
	  T_n(g), \wt{T}_n(g) \Rightarrow \sup_{0 < t < 1} \frac{ |\mu(t) + B^0(t)| }{g(t)}, \qquad n \to \infty, \quad \text{in $ D[0,1] $.}
	\] 
\end{theorem}

\subsection{Multivariate CUSUM approximation}

Let us now consider $L = L_n \in \N $ CUSUM statistics $ \vecC_n(g) = ( C_{n1}(g), \ldots, C_{nL_n}(g) )^\top $ where
\[
C_{nj}(g) = C_{n}( \vecv_{nj}, \vecw_{nj}; g ) = \max_{1 \le k < n} \frac{1}{\sqrt{n} g(k/n) } 
\left| \vecv_{nj}^\top \left( \matS_{nk} - \frac{k}{n} \matS_{nn}   \right) \vecw_{nj} \right|,
\]
$ j \in \{1, \ldots, L_n\} $, defined for $ L_n $ pairs $ (\vecv_{nj}, \vecw_{nj} ) $, $ j \in  \{1, \ldots, L_n \} $, of projection vectors. When using no weights, i.e., $ g(x) = 1 $, $ x \in [0,1] $, the corresponding quantities are denoted $ \vecC_n = ( C_{n1}, \ldots, C_{nL_n} )^\top $.

Let  $ \vecB_n(t) = ( \vecB_n^{(1)}(t), \ldots, \vecB_{n}^{(2L_n)} )^\top $, $ t \ge 0 $, be a $2L_n$-dimensional mean zero Brownian motion with covariance matrix 
\begin{equation}
\label{MultivCovStructure}
 \bfSigma^\vecB_n = \left( \bfSigma^\vecB_{nij} \right)_{1 \le i \le L_n \atop 1 \le j \le L_n} 
\end{equation}
with  blocks
\[
\bfSigma^\vecB_{nij} = \left( \begin{array}{cc} \beta_n^2(\barr, i, \barr, j) & \beta_n^2( \barr, i, \carr, j ) \\  \beta_n^2( \carr, i, \barr, j ) & \beta_n^2( \carr, i, \carr, j) \end{array} \right), \qquad 1 \le i,j \le L_n, 
\]
where, for brevity, 
$
\beta_n^2( \barr, i, \carr, j ) = L_n^{-\iota} \beta_n^2( \barr, \vecv_{ni}, \vecw_{ni}, \carr, \vecv_{nj}, \vecw_{nj} )
$ with  $ \iota = \eins( L_n \to \infty ) $, $ i,j  \in \{1, \ldots, L\} $. Also put $ \alpha_n^2(\aarr, i) = L_n^{-\iota} \beta_n^2( \aarr, \vecv_{ni}, \vecw_{ni},  \aarr, \vecv_{ni}, \vecw_{ni}) $, $ \aarr \in \{ \barr, \carr \} $, $ i  \in \{1, \ldots, L\} $, see (\ref{FormulaBetaSq}). Define the processes
\begin{align*}
\vecG_n(t) &= \vecB_n^{(1)} \eins( t \le \tau ) + [ \vecB_n^{(1)}( \tau ) + ( \vecB_n^{(2)}(t ) - \vecB_n^{(2)}(\tau) ] \eins( t > \tau ), 
\qquad t \ge 0, \\
  \vecG_n^0(k) &= \vecG_n(k) - \frac{k}{n} \vecG_n(n), \qquad k \le n, n \ge 1,
\end{align*}
where $ \vecB_n^{(1)}(t) = \left( \vecB_{n,2j-1}(t)  \right)_{j=1}^{L_n} $, $ \vecB_n^{(2)}(t) = \left( \vecB_{n,2j}(t) \right)_{j=1}^{L_n} $
and $ \vecG_n^0(k) = \left( \vecG_{nj}^0(k)  \right)_{j=1}^{2L_n} $.

\begin{theorem}
\label{MultivCUSUMApprox}
	Suppose that $ \{ \epsilon_{ni} : i \in \Z, n \ge 1 \} $ satisfies Assumption (E). 
 Let $ \vecv_{nj}, \vecw_{nj} $, $ j \in \{1, \ldots, L_n\} $, be weighting vectors satisfying (\ref{AssumptionL1Projections}) uniformly in $ j $, and let $ \barr = \{ b_{nj}^{(\nu)} \} $ and $ \carr = \{ c_{nj}^{(\nu)} \} $  be coefficients satisfying Assumption (D). Then, under the change-point model  (\ref{CPM}), one can redefine, for each $n $, on a new probability space, the vector time series together with a $2L_n$-dimensional mean zero Brownian motion $ \vecB_n = ( \vecB_{ nj} )_{j=1}^{2L_n} $ with covariance function given by (\ref{MultivCovStructure}), such that
 \begin{equation}
 \label{StrongApproxMultiv}
   \left\|  L_n^{-\iota/2}\vecC_n - \left( \max_{1 \le k < n} \frac{1}{L_n^{\iota/2}\sqrt{n}} \left| m_{nj}(k) + \vecG_{nj}^0(k) \right|_2 \right)_{j=1}^{L_n} \right\|_\infty \le 6 \sqrt{2} C_n. n^{-\lambda},
 \end{equation}
and for a weight function $ g $ satisfying (\ref{AssumptionWeightFunction}), for any $ \delta > 0 $
 \begin{equation}
 \label{StrongApproxMultivWeighted}
 \max_{j \le L_n} \PP \left( \left| L_n^{-1/2} C_n( \vecv_{nj}, \vecw_{nj}; g ) - \max_{1 \le k < n} \frac{1}{\sqrt{n}g(k/n)} | m_{nj}( k ) - \vecG_{nj}^0(k) | \right| > \delta \right) = o(1),
 \end{equation}
 where $ m_{nj}(k) = \frac{k(n-\tau)}{n} \Delta_n(j) \eins(  k \le \tau ) +  \tau \frac{n-k}{n} \Delta_n(j)  \eins( \tau < k \le n)  $ with $ \Delta_n(j) = (\vecv_{nj}^\top \bfSigma_{n0} \vecw_{nj} - \vecv_{nj}^\top \bfSigma_{n1} \vecw_{nj} ) $, $ j \in  \{1, \ldots, L_n\} $. 
 \end{theorem}

Observe that under $H_0$ the asymptotic covariance matrix of the approximating process and hence of $ \vecC_n $ is given by $ \bfSigma^\vecB_{n,H_0} = \Var( \vecB_n^{(1)} ) = \bfSigma_n^\vecB( \barr ) $, whose diagonal is given by the elements $ \alpha_n^2(\barr,1), \dots, \alpha_n^2(\barr,L_n) )^\top $  and off-diagonal elements by $ \sigma_n^2( \barr, i, \barr, j )  $, $ 1 \le i \not= j \le L_n $. For fixed $L$ the results of the next section show that $ \bfSigma^\vecB_n( \barr )  $ can be estimated consistently, providing a justification for the test (\ref{TestMultiv}) when $ \bfSigma^\vecB_n( \barr ) $ is regular. 

\subsection{Full-sample and stopped-sample estimation of $ \alpha_n^2(\barr) $ and $\beta_n^2(\barr, \carr )$}

Let us now discuss how to estimate the parameter $ \alpha_n^2( \barr ) $ for one pair $ (\vecv_n, \vecw_n) $ of projection vectors, which is used in the change-point test statistic for standardization, and the  asymptotic covariance parameters $ \beta_n^2( j,k ) = \beta_n^2( \barr, \vecv_{nj}, \vecw_{nj}, \barr, \vecv_{nk}, \vecw_{nk} ) $ for two pairs $( \vecv_{nj}, \vecw_{nj}) $ and $ (\vecv_{nk}, \vecw_{nk} ) $, which arise in the multivariate test for a set of projections.  If there is no change, one may use the proposal of \cite{StelandSachs2017}. But under a change these estimators are inconsistent. The common approach is therefore to use a learning sample for estimation. Alternatively, one may estimate the change-point and use the data before the change. The consistency of that approach follows quite easily when establishing a uniform weak of large numbers of the sequential (process) version of the estimators which uses the first $k$ observations, $ \vecY_{n1}, \ldots, \vecY_{nk} $, where $ k $ is a fraction of the sample size $n$ so that $ k = \trunc{nu} $ for $ u \in (0,1] $:

Fix $ 0 < \varepsilon < 1 $ and define for $ u \in [\varepsilon,1] $
\[
\wh{\alpha}_n^2(u)  = \wh{\Gamma}_n(u; 0) + 2 \sum_{h=1}^m w_{mh} \wh{\Gamma}_n( u; h ),
\]
where
\[
\wh{\Gamma}_n(u;h) = \frac{1}{\trunc{nu}} \sum_{i=1}^{\trunc{nu}-h} [ \vecv_n^\top \vecY_{ni} \vecw_n^\top \vecY_{ni} - \wh{c}_{\trunc{nu}}][ \vecv_n^\top \vecY_{n,i+|h|} \vecw_n^\top \vecY_{n,i+|h|} - \wh{c}_{\trunc{nu}} ],
\]
for $ |h| \le m $, with $ \wh{c}_\trunc{nu} = \trunc{nu}^{-1} \sum_{j=1}^{\trunc{nu}}  \vecv_n^\top \vecY_{nj} \vecw_n^\top \vecY_{nj }  $. The estimators
$ \wh{\beta}_n^2(j,k) $ and $ \wh{\Gamma}_n(u;h,j,k)  $, $ 1 \le j,k \le K $, corresponding to two pairs of projection vectors, are defined analogously, i.e.,
\begin{equation}
\label{DefSigmaTwoProj}
  \wh{\beta}_n^2(j,k) 
 = \wh{\Gamma}_n(u; 0, j, k) + 2 \sum_{h=1}^m w_{mh} \wh{\Gamma}_n( u; h, j, k )
\end{equation}
with
\begin{equation}
\label{DefSigmaTwoProj}
  \wh{\Gamma}_n(u;h,j,k) = \frac{1}{\trunc{nu}} \sum_{i=1}^{\trunc{nu}-h} [ \vecv_{nj}^\top \vecY_{ni} \vecw_{nj}^\top \vecY_{ni} - \wh{c}_{\trunc{nu},j}][ \vecv_{nk}^\top \vecY_{n,i+|h|} \vecw_{nk}^\top \vecY_{n,i+|h|} - \wh{c}_{\trunc{nu},k} ]
\end{equation}
for $ 1 \le i, j \le L $ with $ \wh{c}_{\trunc{nu},j} = \trunc{nu}^{-1} \sum_{i=1}^{\trunc{nu}}  \vecv_{nj}^\top \vecY_{ni} \vecw_{nj}^\top \vecY_{ni } $, $ j \in \{1, \ldots, L\} $.

The weights are often defined through a kernel function $w(x) $ via $ w_{mh} = w( h/b_m ) $ for some bandwidth parameter $ b_m $. For a brief discussion of common choices see \cite{StelandSachs2017}. 

The following theorem establishes the uniform law of large numbers. Especially, it shows that $  \wh{\alpha}_n^2( \ell/n ) $ is consistent for $ \alpha^2( \barr) $ if $ \ell \le \tau $, whereas for $ \ell > \tau $ a convex combination of $\alpha^2( \barr) $  and $ \alpha^2( \carr) $ is estimated. A similar result applies to the estimator of the asymptotic covariance parameter.

\begin{theorem} 
	\label{LLN_EstAlpha_Uniform}
		Assume that $ m \to \infty $ with $ m^2/n = o(1) $, as $ n \to \infty $, and the weights $ \{ w_{mh} \} $ satisfy
	\begin{itemize}
		\item[{\rm (i)}]  $ w_{mh} \to 1 $, as $ m \to \infty $, for all $ h \in \Z $, and
		\item[{\rm (ii)}] $ 0 \le w_{mh} \le W < \infty $, for some constant $W$, for all $m \ge 1 $, $ h \in \Z $.
	\end{itemize}
	If the innovations $ \epsilon_{ni} = \epsilon_i $ are i.i.d. with $ \EE( \epsilon_1^8 ) < \infty $, 
	$ c_{nj}^{(\nu)} = c_j^{(\nu)} $, for all $j$ and $ n \ge 1 $, satisfy the decay condition
	\[
	\sup_{1 \le j} | c_j^{(\nu)} | \ll (j \vee 1)^{-(1+\delta)}
	\]
	for some $ \delta > 0 $, and $ \vecv, \vecw \in \ell_1 $, then under the change-in-coefficients model (\ref{CPM})  with $ \tau = \trunc{n \vartheta } $, 
	$ \vartheta \in (0,1) $,  it holds for any $ 0 < \varepsilon < \vartheta $ 
	\[
	\sup_{u \in [\varepsilon, 1]} | \wh{\alpha}_n^2(u) - \alpha^2(u) | \stackrel{\PP}{\to} 0,
	\]
	as $ n \to \infty $, where
	$
	\alpha^2(u) = \alpha^2( u; \barr, \carr)  = \eins(u \le \vartheta) \alpha^2(\barr) + \eins( u > \vartheta ) \bigl( (\vartheta/u) \alpha^2(\barr) + (1-\vartheta/u) \alpha^2(\carr) \bigr),
	$
	for $ u \in  [\varepsilon, 1] $. Further, 
	\[
		\sup_{u \in [\varepsilon, 1]} | \wh{\beta}_n^2(u,j,k) - \beta^2(u,j,k) | \stackrel{\PP}{\to} 0,
	\]
	where for $ u \in  [\varepsilon, 1] $
	$
\beta^2(u,j,k) = \beta^2( u; j, k, \barr, \carr)  = \eins(u \le \vartheta) \beta^2(\barr, j, k) + \eins( u > \vartheta ) \bigl( (\vartheta/u) \beta^2(\barr,j,k) + (1-\vartheta/u) \beta^2(\carr,j,k) \bigr),
	$ for $ 1 \le j, k \le L $, as defined in Lemma~\ref{CorollaryBetaSq}.
\end{theorem}

Let us now suppose we are given a consistent estimator $ \wh{\tau}_n $ of the unknown change-point; in the next section we make a concrete proposal. In order to estimate the parameter $ \alpha^2(\barr) $ it is natural to use the above estimator using all observations classified by the estimator as belonging to the pre-change period. This means, we estimate $ \alpha^2(\barr) $ by $ \wh{\alpha}_{\wh{\tau}_n}^2 $.
The following result shows that this estimator is consistent under weak conditions.

\begin{theorem} 
	\label{Alpha2Consistency}
	Suppose that $ \wh{\tau}_n $ is an estimator of $ \tau $ satisfying $ \wh{\tau}_n / n \in [\varepsilon,1] $ a.s. and 
	$
	\left| \frac{\wh{\tau}_n}{n} - \vartheta \right| \stackrel{\PP}{\to} 0,
	$
	as $ n \to \infty $. Then
	\[
	| \wh{\alpha}_{\wh{\tau}_n}^2 - \alpha^2(\barr) |  \stackrel{\PP}{\to} 0, \qquad n \to \infty.
	\]
\end{theorem}

\section{Change-point estimation}
\label{Sec:ChangePoint}

In view of the change-point test statistic studied in the previous section, it is natural to estimate the change-point $ \widehat{\tau}_n $ by
\[
\wh{\tau}_n = \argmax_{1  \le k < n} | \wh{\calU}_n(k) |, \qquad
\wh{\calU}_n(k) =  \frac{1}{g(k/n) n} \left(  U_{nk} - \frac{k}{n} U_{nn} \right), \qquad 1 \le k \le n, n \ge 1.
\]
 (By convention, $ \argmax_{x \in \mathcal{D}} f(x) $ denotes the smallest maximizer of some function $ f : \mathcal{D} \to \R $.)

The expectation $ m_n(k) $ of $ U_{nk} - \frac{k}{n} U_{nn} $ is a function of $ \Delta_n $, and we assume that the limit
\begin{equation}
\label{DeltaNLimit}
\Delta = \lim_{n \to \infty} \Delta_n, \qquad i = 0,1,
\end{equation}
exists. To proceed, we need further notation. Put
\begin{equation}
\label{Def_Mnk}
\calU_n(k) = \EE( \wh{\calU}_n(k) ) = 
\left\{
\begin{array}{cc}
\frac{k(n-\tau)}{g(k/n) n^2} \Delta_n, &\qquad k \le \tau, \\
\tau \frac{n-k}{g(k/n) n^2} \Delta_n, & \qquad k > \tau,
\end{array}
\right.
\end{equation}
and introduce the associated rescaled functions
\begin{align}
\label{Def_mhat_cont}
\wh{u}_n(t) &= \wh{\calU}_n( \trunc{nt} ), \qquad t \in [0,1], \\
u_n(t) & = \calU_n( \trunc{nt} ), \qquad t \in [0,1], 
\end{align}
and
\begin{equation}
\label{Def_m_cont}
u(t) = \frac{t}{g(t)} (1-\vartheta) \Delta \eins(t \le \vartheta) + \vartheta \frac{1-t}{g(t)} \Delta \eins( t > \vartheta ), \qquad t \in [0,1].
\end{equation}
If $ g = 1 $, then for  $ \Delta > 0 $ the function $ u(t) $ is strictly increasing on $ [0, \vartheta] $ and strictly decreasing on $ [\vartheta,1] $, and for $ \Delta < 0 $ the same holds for $ | u(t) | $. The same applies for any weight function $g$  such that
\begin{equation}
\label{CPWeightCondition}
 \text{ $ g $ is continuous, $ t/g(t) $ increasing on $ [0,\vartheta] $ and $ (1-t)/g(t) $  decreasing on $ [\vartheta,1] $.} 
\end{equation}
  Obviously, this holds for a large class of functions $ g $ whatever the value of the true change-point. Hence, we expect that the maximizers of $ | u_n( t ) | $, $ t \in [0,1] $, and its estimator $ | \wh{u}_n( t ) | $, $ t \in [0,1] $, converge to the true change-point $ \vartheta $. But the maximizers $ \wh{\tau}_n $ of $ | \wh{\calU}_n | $ and $ \wh{t}_n $ of $ | \wh{u}_n | $ are related by
\begin{equation}
\label{RelationMaximizers}
\wh{\tau}_n = \argmax_{1 \le k \le n} \wh{\calU}_n(k) = \argmax_{1 \le k \le n} \wh{u}_n( k/n) = n \argmax_{t \in \{1/n, \ldots, 1\}} \wh{u}_n(t) = n \wh{t}_n 
\end{equation}
Therefore, since $ \wh{u}_n $ is constant on $ [k/n,(k+1)/n) $, $ k \in \{1, \ldots, n-1 \} $ and vanishes on $ [0,1/n) $,
$ \wh{t}_n \stackrel{\PP}{\to} \vartheta $, as $ n \to \infty $ implies $ \frac{\trunc{\wh{\tau}_n}}{n} \stackrel{\PP}{\to} \vartheta $, as $ n \to \infty $.

A martingale approximation and Doob's inequality provide the following uniform convergence.

\begin{theorem} 
	\label{UnifConvergence}
	Let $g$ be a weight function satisfying (\ref{AssumptionWeightFunction}) and (\ref{CPWeightCondition}). If (\ref{DeltaNLimit}) holds, then 
	\begin{equation}
	\label{UnifConvM1}
	\max_{ 1 \le k < n } | \wh{\calU}_n(k) - \calU_n(k)  | \stackrel{\PP}{\to} 0,  \qquad n \to \infty,
	\end{equation}
	\begin{equation}
	\label{UnifConvM2}
	\sup_{t \in [0,1]} | \wh{u}_n(t) - u(t) |  \stackrel{\PP}{\to} 0,  \qquad n \to \infty.
	\end{equation}
\end{theorem}

The consistency of the change-point estimator $ \wh{\tau}_n $ follows now easily from the above results.

\begin{theorem} 
	\label{ConsistencyCPEstimator} Under the assumptions of Theorem~\ref{UnifConvergence} and the change-point alternative model (\ref{CPM}) with $ \tau = \lfloor n \vartheta \rfloor $, $ \vartheta \in ( \varepsilon, 1) $ for some $ \varepsilon \in (0,1) $, we have
	\[
	\frac{ \wh{\tau}_n}{ n } \stackrel{\PP}{\to} \vartheta, \qquad n \to \infty.
	\]
\end{theorem}

\section{Simulations}
\label{Sec:Simulations}

To investigate the statistical performance of the change-point tests a change from a family of AR($\rho_\nu$) series to a family of (shifted) MA($r$) series, which are, at lag $0$, independent, was examined: We assume that these series, $ Y_{ni}^{(\nu)} $, are defined as follows. Fix $ r \in \N $ and let
\[
\text{pre-change ($ i \le \tau $):}\ Y_{ni}^{(\nu)} = \rho_{\nu} Y_{n,i-1}^{(\nu)} + \epsilon_{i-1}, \quad \text{after-change ($ i > \tau $):} \
 Y_{ni}^{(\nu)}  =  \sum_{j=0}^r \theta_r^{(\nu)} \epsilon_{i-j-(\nu-1)r},
\] 
with $ \rho_\nu = 0.5 \nu/d $,  for $ \nu \in \{1, \ldots, d\} $ and $ n \ge 1 $, i.i.d. standard normal $ \epsilon_t $  and $ \theta_j^{(\nu)} = (1-0.1 \cdot j)  \sqrt{(1-\rho_{\nu}^2)^{-1} / s_\theta^2 } $, $s_\theta^2 = \sum_{k=0}^4 (1-0.1\cdot k)^2    $, $ j \in \{0, \ldots, r = 4\} $, so that the marginal variances of the $ d $ time series do not change. The asymptotic variance parameter, $ \alpha_n $, was estimated with lag truncation $ m = \lceil n^{1/3} \rceil $ justified by simulations not reported here, using three sampling approaches: (i) Learning sample of size $ L = 500 $,  (ii) full in-sample estimation and (iii) stopped in-sample estimation using the modified rule $ \wt{\tau}_n = \max( \trunc{n/4}, \min( 1.15 \cdot \wh{\tau}_n , n ) ) $. Although this modification may lead to some bias, the actual number of observations was increased, since otherwise the sample size for estimation may be too small.

Both a fixed and a random projection were examined. The case of a fixed projection vector was studied by using $ \vecv_n = \vecw_n = (1/d, \ldots, 1/d)^\top $. Random projections were generated by drawing from a Dirichlet distribution, % $ \vecv_n, \vecw_n \sim  \text{Dir}(d,\eins_d)$, 
such that the projections have unit $ \ell_1 $ norm and expectation $ d^{-1} \eins $, in order to study the effect of random perturbations around the fixed projections. 

Table~\ref{SimPower1} provides the rejection rates for  $ n = 100 $ and dimensions $ d \in \{ 10, 100, 200 \} $  when the change-point is given by $ \tau = \trunc{n \vartheta} $ with $ \vartheta \in \{ 0.1, 0.25, 0.5, 0.75, 0.9 \} $, to study changes within the central $ 50\% $ of the data as well as early and late changes. First, one can notice that the power is somewhat increasing in the dimension but quickly saturates. The results for stopped-sample and in-sample estimation are quite similar. The unweighted CUSUM procedure has very accurate type I error rate if a learning sample is present, whereas the weighted CUSUM overreacts somewhat under the null hypothesis. For stopped-sample and in-sample estimation the unweighted procedure is conservative, whereas the weighted CUSUM keeps the level quite well with only little overreaction. Although the unweighted CUSUM operates at a smaller significance level, it is more powerful than the weighted procedure when the change occurs in the middle of the sample, but the weighted CUSUM performs better for early changes. The results for a random projection are very similar.

The accuracy and power of the global test related to the CUSUM transform was examined for a change to a MA model after half of the sample for the sample size $ n = 500 $. The design of this study is data-driven as the principal directions calculated for the ozone data set were used in  addition to random projections. The global test based on the weighted CUSUM transform using the weight function $ g(t) = [t(1-t)]^{\beta} $ with $ \beta=0.3 $ was fed with the first $r$ poejctions for various values of $r$. The results are provided in Table~\ref{SimPowerGlobalTest}. Each entry is based on $ 1,000 $ runs. According to these figures, the proposed global test is accurate in terms of the significance level and quite powerful.

\begin{small}
	\begin{table}[h]
		\centering
		\begin{tabular}{ccc}
			\hline
			$r$ & level & power  \\ 
			\hline
			$ 2 $ & $ 0.045 $ & $ 0.9 $ \\
			$ 3 $ & $0.05$  &  $ 0.85 $  \\
			$ 4 $ & $0.042 $ &$ 0.9 $ \\
			$ 7 $ & $0.037 $ & $  0.884 $  \\
			\hline
		\end{tabular}
		\caption{Simulated level and power of the global test based on the CUSUM transform.}
		\label{SimPowerGlobalTest}
	\end{table}
\end{small}

\section{Data example}
\label{Sec:Example}

To illustrate the proposed methods, we analyze $ n = 1826 $ daily observations of 8 hour maxima of ozone concentration collected at $ d = 444 $ monitors in the U.S.. The data corresponds to the 5-year-period from January 2010 to December 2014. We analyze mean corrected data, see \cite{SchweinbergerEtAl2017}, namely residuals obtained after fitting cubic splines to the log-transformed data, in order to correct level and seasonal ups and downs. 

\begin{small}
	\begin{table}
		\centering
		\caption{Simulated power for the sample size $ n = 100 $ for fixed projection and a random projection for dimension $ d =  10,100,200 $ and different change-point locations. The entries for $ \vartheta = 1 $ provide simulated type I error rates.}
\label{SimPower1}
		\begin{tabular}{lrrrrrrrr}
			\hline
			Fixed projection \\ % &		\multicolumn{8}{c}{Fixed projection} \\
			& $\vartheta$ & 10 & 100 & 200 &$ \vartheta$ & 10 & 100 & 200 \\ 
			Method & \multicolumn{4}{c}{\underline{Unweighted CUSUM}}  &  \multicolumn{4}{c}{\underline{Weighted CUSUM}}  \\	
			& $\vartheta$ & 10 & 100 & 200 &$ \vartheta$ & 10 & 100 & 200 \\ 
			& 0.10 & 0.03 & 0.02 & 0.02 & 0.10 & 0.15 & 0.14 & 0.14 \\ 
			&0.25 & 0.31 & 0.34 & 0.34 & 0.25 & 0.35 & 0.37 & 0.38 \\ 
			L=500	&0.50 & 0.70 & 0.75 & 0.77 & 0.50 & 0.55 & 0.59 & 0.61 \\ 
			&0.75 & 0.39 & 0.46 & 0.46 & 0.75 & 0.33 & 0.40 & 0.38 \\ 
			&0.90 & 0.07 & 0.08 & 0.08 & 0.90 & 0.09 & 0.08 & 0.09 \\ 
			&1.00 & 0.05 & 0.06 & 0.05 & 1.00 & 0.09 & 0.11 & 0.10 \\[1ex] % \hline
			& 0.10 & 0.14 & 0.14 & 0.09 & 0.10 & 0.79 & 0.98 & 0.99 \\ 
			&0.25 & 0.79 & 0.88 & 0.88 & 0.25 & 0.86 & 0.93 & 0.92 \\ 
			stopped-sample		&0.50 & 0.90 & 0.93 & 0.93 & 0.50 & 0.70 & 0.72 & 0.71 \\ 
			&0.75 & 0.29 & 0.30 & 0.29 & 0.75 & 0.17 & 0.17 & 0.18 \\ 
			&0.90 & 0.03 & 0.02 & 0.02 & 0.90 & 0.04 & 0.03 & 0.04 \\ 
			&1.00 & 0.02 & 0.01 & 0.01 & 1.00 & 0.07 & 0.07 & 0.08 \\[1ex]  %\hline
			&0.10 & 0.14 & 0.13 & 0.11 & 0.10 & 0.79 & 0.98 & 0.99 \\ 
			&0.25 & 0.79 & 0.87 & 0.86 & 0.25 & 0.86 & 0.93 & 0.92 \\ 
			in-sample		&0.50 & 0.90 & 0.92 & 0.93 & 0.50 & 0.69 & 0.73 & 0.72 \\ 
			&0.75 & 0.29 & 0.30 & 0.29 & 0.75 & 0.17 & 0.17 & 0.18 \\ 
			&0.90 & 0.03 & 0.02 & 0.02 & 0.90 & 0.04 & 0.04 & 0.04 \\ 
			&1.00 & 0.02 & 0.02 & 0.02 & 1.00 & 0.08 & 0.08 & 0.07 \\ 
			\hline
			Random projection \\ % &		\multicolumn{8}{c}{Fixed projection} \\
Method & \multicolumn{4}{c}{\underline{Unweighted CUSUM}}  &  \multicolumn{4}{c}{\underline{Weighted CUSUM}}  \\	
& $\vartheta$ & 10 & 100 & 200 &$ \vartheta$ & 10 & 100 & 200 \\ 
			&0.10 & 0.02 & 0.02 & 0.02 & 0.10 & 0.14 & 0.14 & 0.14 \\ 
			&0.25 & 0.32 & 0.36 & 0.36 & 0.25 & 0.36 & 0.39 & 0.39 \\ 
			$L=500$		&0.50 & 0.71 & 0.77 & 0.76 & 0.50 & 0.54 & 0.60 & 0.61 \\ 
			&0.75 & 0.39 & 0.46 & 0.46 & 0.75 & 0.33 & 0.39 & 0.39 \\ 
			&0.90 & 0.08 & 0.08 & 0.08 & 0.90 & 0.09 & 0.09 & 0.10 \\ 
			&1.00 & 0.05 & 0.05 & 0.05 & 1.00 & 0.10 & 0.10 & 0.09 \\[1ex] %\hline
			&0.10 & 0.13 & 0.15 & 0.11 & 0.10 & 0.80 & 0.98 & 0.99 \\ 
			&0.25 & 0.80 & 0.86 & 0.87 & 0.25 & 0.85 & 0.92 & 0.92 \\ 
			stopped-sample		&0.50 & 0.90 & 0.93 & 0.93 & 0.50 & 0.70 & 0.72 & 0.72 \\ 
			&0.75 & 0.28 & 0.29 & 0.30 & 0.75 & 0.18 & 0.17 & 0.17 \\ 
			&0.90 & 0.03 & 0.03 & 0.03 & 0.90 & 0.04 & 0.04 & 0.04 \\ 
			&1.00 & 0.02 & 0.02 & 0.02 & 1.00 & 0.08 & 0.07 & 0.07 \\[1ex] %\hline
			&0.10 & 0.14 & 0.14 & 0.11 & 0.10 & 0.80 & 0.98 & 0.98 \\ 
			&0.25 & 0.79 & 0.87 & 0.87 & 0.25 & 0.86 & 0.92 & 0.92 \\ 
			in-sample		&0.50 & 0.91 & 0.93 & 0.93 & 0.50 & 0.71 & 0.72 & 0.73 \\ 
			&0.75 & 0.28 & 0.30 & 0.30 & 0.75 & 0.17 & 0.18 & 0.16 \\ 
			&0.90 & 0.03 & 0.03 & 0.02 & 0.90 & 0.04 & 0.04 & 0.04 \\ 
			&1.00 & 0.02 & 0.02 & 0.02 & 1.00 & 0.08 & 0.07 & 0.08 \\ 
			\hline
		\end{tabular}
	\end{table}
\end{small}

The data of the first year was used to calculate a sparse PCA. We use the method of  \cite{ErichsonEtAl2018} to get sparse directions $ \vecv_i $ instead of  \cite{CaiMaWu2015}, since, according to the latter authors, their estimators leading to minimax rates are computationally infeasible.The sparse PCA was conducted as follows: Denote the  $ 365 \times 444 $ data matrix by $ \matX $.  \cite{ErichsonEtAl2018}  propose to calculate  an orthonormal matrix $ \matA $ and a sparse matrix $ \matB = ( \vecv_1, \ldots, \vecv_d )  $ solving
\[
\min_{\matA, \matB} (1/2) \| \matX - \matX \matB \matA^\top \|_F^2 + \psi( \matB ), \quad \matA^\top \matA = \matid, 
\]
where we used an elastic net regularization $ \psi(\matB ) = \lambda_1 \| \matB \|_{\ell_1} + \lambda_2 \| \matB\|_{\ell_2} $ with parameters $ \lambda_1 = 0.025 $, $ \lambda_2 = 0.1 $. This analysis shows that
the supports $ \calS_i = \{ j : \vecv_{ij} \not= 0 \} $ of the leading six projections, where $ \vecv_i = (v_{i1}, \ldots, v_{id})^\top $ for $ i \in \{1, \ldots, 6\} $, correspond to a spatial segmentation which eases interpretation. 
Figure~\ref{fig:usa-ozone-spca} shows the geographic locations of these supports.

\begin{figure}
	\centering
	\includegraphics[width=0.6\linewidth]{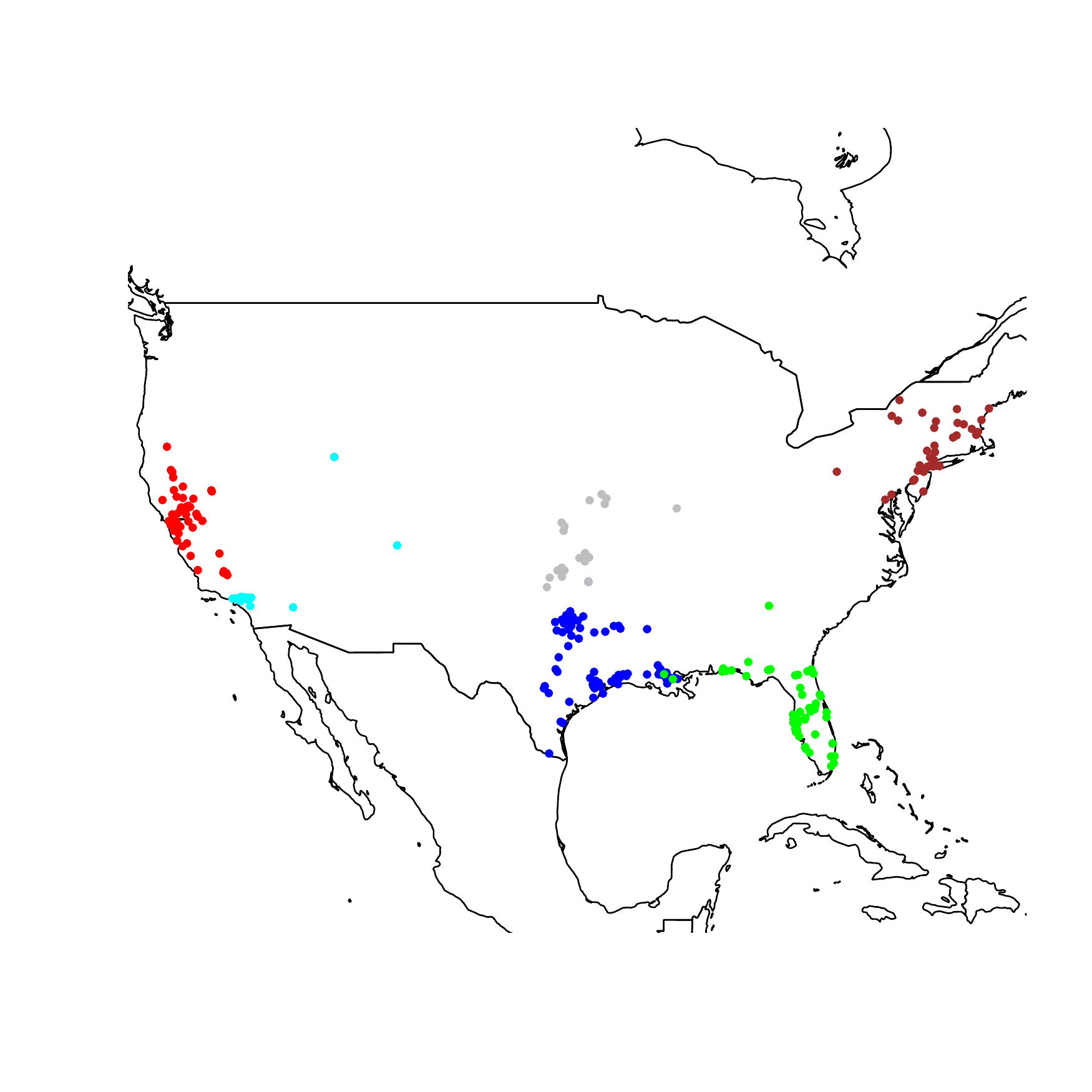}
	\caption{Sparse principal component analysis of ozone residuals from $444$ monitors across the U.S.. Depicted are the locations of the supports $ \calS_1, \ldots, \calS_6 $ of the leading principal directions $ \vecv_1, \ldots, \vecv_6 $.}
	\label{fig:usa-ozone-spca}
\end{figure}

 The data of the years 2011 to 2014, providing the test sample $ \vecY_{n1}, \ldots, \vecY_{nn} $ with $ n = 1462$, was now analyzed using the leading directions as projection vectors. The proposed change-point tests were  applied to test for the presence of changes in the (co-) variances $ \Cov( \vecv_k^\top \vecY_{ni}, \vecv_\ell^\top \vecY_{ni} ) $, $ i \in  \{1, \ldots, n\} $, for $ 1 \le k \le \ell \le 7$. The asymptotic variance parameter was estimated using both the full in-sample and stopped-sample approach.  The application of  the unweighted CUSUM approach revealed no significances at the usual levels.
\section{Proofs}
\label{Sec:Proofs}

The proofs are based on martingale approximations, which require several additional results and technical preparations. These results extend and complement the results obtained in \cite{StelandSachs2017}.

\subsection{Preliminaries}

For an arbitray array of coefficients $ \aarr = \{ a_{nj}^{(\nu)} : j \ge 0, 1 \le \nu \le d_n, n \ge 1 \} $ and vectors
$ \vecv_n = ( v_{n1}, \ldots, v_{nd_n})^\top $ and $ \vecw_n = (w_{n1}, \ldots, w_{nd_n})^\top $ with finite $ \ell_1 $-norm, i.e., 
$  \| \vecv_n \|_{\ell_1}, \| \vecw_n \|_{\ell_1} < \infty $, define
\[
f_{0,0}^{(n)}( \aarr, \vecv_n, \vecw_n ) = \sum_{\nu,\mu=1}^{d_n} v_{n\nu} w_{n\mu} a_{nj}^{(\nu)} a_{nj}^{(\mu)}, \quad 
f_{l,j}^{(n)}( \aarr ) = \sum_{\nu,\mu=1}^{d_n} v_{n\nu} w_{n\mu} [ a_{nj}^{(\nu)} a_{n,j+l}^{(\mu)} + a_{nj}^{(\mu)} a_{n,j+l}^{(\nu)} ]
\]
for $ j \in \{0, 1, \cdots\} $ and $ l \in  \{1, 2, \cdots\} $.
Put  
$
\wt{f}_{\ell,i}^{(n)}( \aarr, \vecv_n, \vecw_n ) = \sum_{j=i}^\infty f_{\ell,j}^{(n)}( \aarr, \vecv_n, \vecw_n ), 
$ 
for $\ell, i \in  \{0, 1, 2, \cdots \}$.

Introduce for coefficients $ \aarr $ satisfying Assumption (D) and vectors $ \vecv_n $ and $ \vecw_n $ the $ \calF_{nk} $-martingales
\[
M_k^{(n)}( \aarr, \vecv_n, \vecw_n ) = \wt{f}_{0,0}^{(n)}( \aarr, \vecv_n, \vecw_n ) \sum_{i=0}^k ( \epsilon_{ni}^2 - \sigma_i^2) + \sum_{i=0}^k \epsilon_{ni} \sum_{\ell=1}^\infty 
\wt{f}_{\ell,0}^{(n)}( \aarr, \vecv_n, \vecw_n ) \epsilon_{n,i-j}, \qquad k \ge 0,
\]
which start in $ M_0^{(n)} = 0 $, for each $ n \ge 0 $.  Put
\[
S_{n',m'}^{(n)}( \aarr, \vecv_n, \vecw_n ) = \sum_{i=m'+1}^{m'+n'} ( Y_{ni}( \ava ) Y_{ni}( \awa ) - \EE[ Y_{ni}( \ava ) Y_{ni}( \awa ) ] ), \qquad m', n' \ge 0.
\]
Notice that, by definitions (\ref{Def_vecU}) and (\ref{Def_vecD}),
\begin{equation}
\label{EqualForApprox}
S_{k,0}^{(n)}( \barr, \vecv_n, \vecw_n)  = \vecD_{nk}^{(1)}, \qquad S_{k,0}^{(n)}( \carr, \vecv_n, \vecw_n )  = \vecD_{nk}^{(2)},
\end{equation}
for $ k \ge 1 $ and $  n \ge 1 $, where $ \vecD_{nk} = ( \vecD_{nk}^{(1)}, \vecD_{nk}^{(2)}) $. 
For brevity introduce the difference operator
\begin{align*}
\delta M_{m'+n'}^{(n)}( \aarr, \vecv_n, \vecw_n ) & = M_{m'+n'}^{(n)}( \aarr, \vecv_n, \vecw_n ) - M_{m'}^{(n)}( \aarr, \vecv_n, \vecw_n )  \\
&= \wt{f}_{0,0}^{(n)}( \aarr, \vecv_n, \vecw_n ) \sum_{i=m'+1}^{m'+n'} ( \epsilon_{ni}^2 - \sigma_{ni}^2) + \sum_{i=m'+1}^{m'+n'} \epsilon_{ni} \sum_{\ell=1}^\infty 
\wt{f}_{\ell,0}^{(n)}( \aarr, \vecv_n, \vecw_n ) \epsilon_{n,i-\ell}, 
\end{align*}
for $ k, n \ge 1 $, which takes the lag $ n' $ forward difference at $ m' $.
Notice that for $ m' = 0 $ 
\[
\delta M_k^{(n)}( \aarr, \vecv_n, \vecw_n ) = \wt{f}_{0,0}^{(n)}( \aarr, \vecv_n, \vecw_n ) \sum_{i=1}^k ( \epsilon_{ni}^2 - \sigma_{ni}^2) + \sum_{i=1}^k \epsilon_{ni} \sum_{\ell=1}^\infty  \wt{f}_{\ell,0}^{(n)}( \aarr, \vecv_n, \vecw_n ) \epsilon_{n,i-\ell}, \ k, n \ge 1,
\]
coincides with the martingale $ M_k^{(n)}( \aarr, \vecv_n, \vecw_n ) $. A direct calculation shows that 
\begin{align}
\label{CovMartingale_shifted}
& \Cov( \delta M_{m'+n'}^{(n)}( \barr, \vecv_n, \vecw_n ), \delta M_{m'+n'}^{(n)}( \carr, \wt\vecv_n, \wt\vecw_n )  )  \\ \nonumber
& \ =  \wt{f}_{0,0}^{(n)}( \barr, \vecv_n, \vecw_n ) \wt{f}_{0,0}^{(n)}( \carr, \wt\vecv_n, \wt\vecw_n ) 
\sum_{j=1}^{n'} ( \gamma_{n,m'+j} + \sigma_{n,m'+j}^4 ) 
 \\ \nonumber & \qquad \qquad
+ \sum_{j=1}^{n'} \sum_{\ell=1}^\infty \wt{f}_{\ell,0}^{(n)}( \barr, \vecv_n, \vecw_n ) \wt{f}_{\ell,0}^{(n)}( \carr, \wt\vecv_n, \wt\vecw_n ) \sigma_{n,m'+j}^2 \sigma_{n,m'+j-\ell}^2,
\end{align}
for $ n', m' \ge 0 $ and $ n \ge 1 $.

\subsection{Martingale approximations}

The following lemma provides an explicit formula for the asymptotic covariance parameter  related to the two CUSUMs, $  \beta_n^2( \barr, \vecv_n, \vecw_n, \carr, \wt\vecv_n, \wt\vecw_n )  $, using different pairs $ (\vecv_n, \vecw_n) $ and $ (\wt\vecv_n, \wt\vecw_n) $ of weighting vectors, abbreviated as $ \beta_n^2( \barr, \carr) = \beta_n^2( \barr, \vecv_n, \vecw_n, \carr, \vecv_n, \vecw_n ) $. Especially, it follows from these results that  the asymptotic variance of a single CUSUM detector under the no-change null hypothesis, $ \alpha^2( \aarr ) = \beta_n^2( \aarr, \aarr ) $, satisfies
\[
  \alpha_n^2( \aarr ) \approx \frac{1}{n} \Var( D_{nn} ).
\]
The following general results hold under a mild condition on the error terms and especially show that (\ref{CovMartingale_shifted}) can be approximated by $ n' \beta_n^2( \barr, \vecv_n, \vecw_n, \carr, \wt\vecv_n, \wt\vecw_n ) $ at the rate $ (n')^{1-\theta} $, uniformly in $ n $ and $ m' $, cf. \cite[(3.18)]{StelandSachs2017} and \cite{Kouritzin1995}. The proof extends these latter results and improves the bounds, but it is  technical and thus deferred to the appendix. The improved bounds show that the $ \ell_1 $-norms of the weighting vectors may grow slowly without sacrificing the convergence of the second moments, cf. the verification of (II) and (III) in the proof of Theorem~\ref{BasicStrongApprox}. 

\begin{lemma}
	\label{CorollaryBetaSq}
	Let $ \epsilon_{ni} $, $ i \in \Z $, be independent with variances $ \sigma_{ni}^2 $ and third moments $ \gamma_{ni} $  satisfying  
%\begin{center}
%	\begin{minipage}[b]{.4\textwidth}
%		\vspace{-\baselineskip}
		\begin{equation}
		\frac{1}{n'} \sum_{i=1}^{n'} i | \sigma_{ni}^2 - s_{n1}^2 | \stackrel{n,n'}{\ll}  (n')^{-\beta}, \label{ConditionOnVariances}
		\end{equation}
%	\end{minipage}% 
%%	\hfill\hfill and\hfill
%\qquad and \quad
%	\begin{minipage}[b]{.4\textwidth}
%		\vspace{-\baselineskip}
		\begin{equation}
		\frac{1}{n'}  \sum_{i=1}^{n'} i |\gamma_{ni} - \gamma_n | \stackrel{n,n'}{\ll} (n')^{-\beta} \label{ConditionOnGammas}
		\end{equation}
%	\end{minipage}
%\end{center}
%	\begin{equation} 
%	\label{ConditionOnGammas}
%	\frac{1}{n'}  \sum_{i=1}^{n'} i |\gamma_{ni} - \gamma_n | \stackrel{n,n'}{\ll} (n')^{-\beta}  
%	\end{equation} 
	for constants $ s_{n1}^2 \in (0,\infty) $ and $ \gamma_n \in \R $ for some $ 1 < \beta < 2 $ with $  1 + \theta  < \beta $.  Then for $ n, n' \ge 1 $, with $ K_n = \| \vecv_n \|_{\ell_1} \| \vecw_n \|_{\ell_1} \| \wt\vecv_n \|_{\ell_1} \| \wt\vecv_n \|_{\ell_1} $,
	\begin{equation}
	\label{KonvCovMart_rate}
	\left| \Cov( M_{n'}^{(n)}( \barr, \vecv_n, \vecw_n ), M_{n'}^{(n)}( \carr, \wt\vecv_n, \wt\vecw_n ) )  - (n') \beta_n^2( \barr, \vecv_n, \vecw_n, \carr, \wt\vecv_n, \wt\vecw_n )
	\right| \stackrel{n,n'}{\ll} K_n  (n')^{1-\theta}, 
	\end{equation}
	and for $ n, n' \ge 1 $ and $ m'  \ge 0 $
	\begin{equation}
	\label{KonvCovMart_rate2}
	\left| \Cov( \delta M_{m'+n'}^{(n)}( \barr, \vecv_n, \vecw_n), \delta M_{m'+n'}^{(n)}( \carr, \wt\vecv_n, \wt\vecw_n ) )  - (n') \beta_n^2( \barr,  \vecv_n, \vecw_n, \carr, \wt\vecv_n, \wt\vecw_n )
	\right| \stackrel{n,n',m'}{\ll} K_n  (n')^{1-\theta},
	\end{equation}
	if 
	\begin{equation}
	\label{FormulaBetaSq}
	\beta_n^2( \barr, \vecv_n, \vecw_n, \carr, \wt\vecv_n, \wt\vecw_n ) = \wt{f}_{0,0}^{(n)}( \barr, \vecv_n, \vecw_n ) \wt{f}_{0,0}^{(n)}( \carr ) ( \gamma_n -  s_{n1}^4)  
	+ s_{n1}^4 \sum_{\ell=1}^{\infty} \wt{f}_{\ell,0}^{(n)}( \barr, \vecv_n, \vecw_n ) \wt{f}_{\ell,0}^{(n)}( \carr, \wt\vecv_n, \wt\vecw_n ).
	\end{equation}
\end{lemma}

\begin{lemma} 
	\label{LemmaMartingaleApprox}
	Let $ \{ \epsilon_{nk} : k \ge 1, n \ge 1 \} $ be independent mean zero random variables with
	variances $ \sigma_{nk}^2 $ and third moments $ \gamma_{nk} $ satisfying Assumption (E). Let $ \aarr $ be
	coefficients satisfying the decay condition (D). Then we have for $ n', m' \ge 0 $ and $ n \ge 1 $ 
	\begin{equation}
	\label{MartingaleApprox1}
	\EE( S_{n',m'}^{(n)}( \aarr, \vecv_n, \vecw_n ) - \delta M_{m'+n'}^{(n)}( \aarr, \vecv_n, \vecw_n ) )^2 \stackrel{n,m',n'}{\ll} \| \vecv_n \|_{\ell_1}^2 \| \vecw_n \|_{\ell_1}^2  (n')^{1-\theta}.
	\end{equation}
	Further, for $ k \ge 1 $ and $ n \ge 1 $
	\begin{align}
	\label{MartingaleApproxD1}
	\EE( \vecD_{nk}^{(1)} - \delta M_k^{(n)}( \barr )  )^2 & \stackrel{n,k}{\ll} \| \vecv_n \|_{\ell_1}^2 \| \vecw_n \|_{\ell_1}^2  k^{1-\theta}, \\
	\label{MartingaleApproxD2}
	\EE( \vecD_{nk}^{(2)} - \delta M_k^{(n)}( \carr )  )^2 & \stackrel{n,k}{\ll}  \| \vecv_n \|_{\ell_1}^2 \| \vecw_n \|_{\ell_1}^2 k^{1-\theta}, 	  
	\end{align}
	such that
	\begin{equation}
	\label{MartingaleApproxDvec}
	\EE \| \vecD_{nk} - \delta \vecM_{k}^{(n)} \|_2^2 \stackrel{n,k}{\ll} \| \vecv_n \|_{\ell_1}^2 \| \vecw_n \|_{\ell_1}^2  n^{1-\theta}.
	\end{equation}
	(\ref{MartingaleApproxD1}), (\ref{MartingaleApproxD2}) and (\ref{MartingaleApproxDvec}) also hold (with obvious modifications), if $ \vecD_{nk}^{(1)} = S_{k,0}^{(n)}( \barr, \vecv_n, \vecw_n ) $ and $ \vecD_{nk}^{(2)} = S_{k,0}^{(n)}( \carr, \wt\vecv_n, \wt\vecw_n ) $ for  two pairs of weighting vectors, where the bound in (\ref{MartingaleApproxDvec}) then is given by $ \max \{  \| \vecv_n \|_{\ell_1}^2 \| \vecw_n \|_{\ell_1}^2,  \| \wt\vecv_n \|_{\ell_1}^2 \| \wt\vecw_n \|_{\ell_1}^2 \} n^{1-\theta}  $. 
\end{lemma}

\begin{proof} See appendix. 
\end{proof} 

The next lemma studies the conditional covariances of the approximating martingales. It generalizes \cite[Lemma~2.2]{StelandSachs2018}  to  the change-point model and two different pairs of projection vectors.

\begin{lemma} 
	\label{KonvKovBeta}
	Suppose that the conditions of Lemma~\ref{CorollaryBetaSq} hold and $ \beta_n^2( \barr, \carr ) $ is as defined there. Then it holds for $ m', n' \ge 0 $ and $ n \ge 1 $ with $ K_n = \| \vecv_n \|_{\ell_1}  \| \vecw_n \|_{\ell_1} \| \wt\vecv_n \|_{\ell_1}  \| \wt\vecw_n \|_{\ell_1} $ 
	\begin{align*}
	 E_{n'}^{(n)} &= \left\|
	\EE\left[
	( \delta M_{m'+n'}^{(n)}( \barr, \vecv_n, \vecw_n ) )( \delta M_{m'+n'}^{(n)}( \carr, \wt\vecv_n, \wt\vecw_n ) )
	\mid \calF_{n,m'} \right] - n' \beta_n^2( \barr, \vecv_n, \vecw_n, \carr, \wt\vecv_n, \wt\vecw_n ) 
	\right\|_{L_1} \\
	& \stackrel{n,m',n'}{\ll} K_n (n')^{1-\theta/2} 
	\end{align*}
	and
	\begin{align*}
	& \left\|
	\EE
	\left[
	( S_{m',n'}^{(n)}( \barr, \vecv_n, \vecw_n ) )( S_{m',n'}^{(n)}( \carr, \wt\vecv_n, \wt\vecw_n ) )
	\mid \calF_{n,m'} \right] - n' \beta_n^2( \barr, \vecv_n, \vecw_n, \carr, \wt\vecv_n, \wt\vecw_n ) 
	\right\|_{L_1}  \\ \qquad &
	\stackrel{n,m',n'}{\ll} K_n (n')^{1-\theta/2}.
	\end{align*}
\end{lemma}

\begin{proof} See appendix.
\end{proof}

\subsection{Proofs of Subsection~\ref{SubSecStrongApprox}}

After the above preparations, we are now in a position to show  Theorem~\ref{BasicStrongApprox}.

\begin{proof}[Proof of Theorem~\ref{BasicStrongApprox}]
	Put 
	\begin{equation}
	\label{DefXiProof}
	\bfxi_{i}^{(n)} = \bfxi_{i}^{(n)}( \avb, \avc ) = \left( \begin{array}{cc} Y_{ni}( \avb ) Y_{ni}( \awb ) - \EE[ Y_{ni}( \avb ) Y_{ni}( \awb ) ]\\ Y_{ni}( \avc ) Y_{ni}( \awc ) - \EE[  Y_{ni}( \avc ) Y_{ni}( \awc ) ]\end{array}  \right),
	\end{equation}
	such that 
	$
	\matD_{nk} = \sum_{i \le k} \bfxi_{i}^{(n)},
	$
	for $ k \ge 1 $ and $ n \ge 1 $. Let us consider the bivariate
	extension of the sums $ S_{n',m'}^{(n)} $,
	\[
	\vecS_{n',m'}^{(n)} = ( S_{n',m'}^{(n)}(1), S_{n',m'}^{(n)}(2) )^\top = \sum_{k=m'+1}^{m'+n'} \bfxi_{k}^{(n)}, \qquad m', n' \ge 0.
	\]
Introduce the conditional covariance operators
	\[
	\vecC_{n',m'}^{(n)}( \vecu ) = \EE[ \vecu^\top S_{n',m'}^{(n)} S_{n',m'}^{(n)} | \calF_{n,m'} ], \qquad \vecu \in \R^2, \ n, n', m' \ge 1, 
	\]
	and the unconditional covariance operator associated to the Brownian motion $ \vecB^{(n)} $,
	\[
	\vecT^{(n)}( \vecu ) = \EE( \vecu^\top \vecB_n \vecB_n ),  \qquad \vecu \in \R^2, \ n \ge 1.
	\]		
	We shall verify \cite[Th.~1]{Philipp1986}, namely the validity of the following conditions: For $ m' \ge 0 $, $ n' \ge 1 $, 
	\begin{itemize}
		\item[(I)] $ \sup_{j \ge 1} \EE \| \bfxi_{j}^{(n)} \|_2^{2+\delta} < \infty $ for some $ \delta > 0 $.
		\item[(II)] For some $ \varepsilon > 0 $ it holds
		\[
		\EE \|  \EE( \vecS_{n',m'}^{(n)} \mid \calF_{n,m'} ) \|_2 \stackrel{n,n',m'}{\ll} (n')^{1/2-\varepsilon}
		\]
		\item[(III)] There exists a covariance operator $ \vecC $, namely $ \vecT^{(n)} $, such that the conditional covariance operator 
		$ \vecC_{n',m'}^{(n)} $ converges to $ \vecC $ in the semi-norm $ \| \cdot \|_{op} $ in expectation in the sense that 		for some $ \theta' > 0 $.
		\[
		\EE \| \vecC_{n'}( \cdot \mid \calF_{n,m'} ) - \vecC( \cdot ) \|_{op} \stackrel{n,n',m'}{\ll} (n')^{1-\theta'}.
		\]
	\end{itemize} 
	{\bf \em Remark:} As the construction is for fixed $n$, one could consider $ \stackrel{n',m'}{\ll} $ in (II) and (III). But since we are interested in $ n \to \infty $ and (II) and (III) yield the  moment convergence with rate for the partial sum $ \vecS_{n,0}^{(n)} $ of interest (for large $n$), we show $ \stackrel{n,n',m'}{\ll} $ and consider the case $ n' \ge n $. This includes the real sample size $n$ and (\ref{AssumptionL1Projections}) then ensures
	the bound $ \| \vecv_n \|_{\ell_1}^2 \| \vecw_n \|_{\ell_1}^2  (n')^{-\theta/2} = O( (n')^{-\theta'/2} ) $ we shall use.
		
	Write $ \bfxi_{i}^{(n)} = ( \bfxi_{i1}^{(n)}, \bfxi_{i2}^{(n)} )^\top $, $ i \ge 1 $, and observe that $ \bfxi_{ij}^{(n)} = \| \vecv_n \|_{\ell_1} \| \vecw \|_{\ell_1} \bfxi_{nij}^* $ where $ \bfxi_{nij}^* $ is obtained from $ \bfxi_{ij}^{(n)} $ by replacing $ \vecv_n $ by  $ \vecv_n^* = \vecv_n / \| \vecv_n  \|_{\ell_1}  $ and $ \vecw_n $ by $ \vecw_n^*= \vecw_n / \| \vecw_n \|_{\ell_1} $. The $ C_r $-inequality and Cauchy-Schwarz yield
	\begin{align*}
	\EE |  \bfxi_{i1}^{(n)} |^{2+\delta} & \le \| \vecv_n \|_{\ell_1}^{2+\delta} \| \vecw_n \|_{\ell_1}^{2+\delta} \EE( | Y_{ni}( \vecv_n^*{}^\top \vecb_n ) Y_{ni}( \vecw_n^*{}^\top \vecb_n )  |  + \EE | Y_{ni}(\vecv_n^*{}^\top \vecb_n ) Y_{ni}( \vecw_n^*{}^\top \vecb_n  ) | )^{2+\delta} \\
	& \le \| \vecv_n \|_{\ell_1}^{2+\delta} \| \vecw_n \|_{\ell_1}^{2+\delta} 2^{3+\delta} \EE | Y_{ni}( \vecv_n^*{}^\top \vecb_n ) Y_{ni}( \vecw_n^*{}^\top \vecb_n  ) |^{2+\delta} \\
	& \le \| \vecv_n \|_{\ell_1}^{2+\delta} \| \vecw_n \|_{\ell_1}^{2+\delta} 2^{3+\delta} \sqrt{ \EE | Y_{ni}( \vecv_n^*{}^\top \vecb_n  ) |^{4+2\delta} } \sqrt{ \EE | Y_{ni}(  \vecw_n^*{}^\top \vecb_n  ) |^{4+2\delta} },
	\end{align*}
	and the second component is estimated analogously.
	Following the arguments in \cite[p.~343]{Kouritzin1995}, for $ \delta' \in (0,2) $ and $ \chi = \delta'/2 $, one can show that
	for $ \aarr \in \{ \barr, \carr  \} $ and $ \vecu_n \in \{ \vecv_n^*, \vecw_n^* \} $
	\begin{align*}
	\EE | Y_{ni}( {\bm u}_n' {\bm a}_n ) |^{4+\delta'} &\le  
	\sup_{n, k \ge 0} \EE | \epsilon_{nk} | \sum_{\ell=0}^\infty | a_{n\ell}^{(u)} |^{2(2+\chi)} \\  &
	 + \sup_{n,k \ge 0} \EE( \epsilon_{nk}^2 ) \left\{ \sup_{n',k' \ge 0} \EE( \epsilon_{n'k'}^2 )  \right\}^{1+\chi} 
	\sum_{\ell=0}^\infty | a_{n\ell}^{(u)} |^2 \left\{ \sum_{\ell=0}^\infty | a_{n\ell}^{(u)} |^2  \right\}^{1+\chi}, 
	\end{align*}
	where $ a_{n\ell}^{(u)} = \sum_{\nu=1}^{d_n} a_{n\ell}^{(\nu)} u_{n\nu} \ll (\max(\ell, 1))^{-3/4-\theta/2} $, uniformly in uniformly $ \ell_1 $-bounded $ \vecu_n $  and $ n \ge 1 $, such that $ \sum_{\ell=0}^\infty | a_{n\ell}^{(u)} |^2 < \infty $ and, in turn, 
	$ \sum_{\ell=0}^\infty | a_{n\ell}^{(u)} |^{2(2+\chi)} < \infty $. Eventually, we obtain for any $ n $ 
	\begin{equation}
	\label{UniformMomentsXi}
		\max_{j=1,2}  \sup_{i \ge 1} \EE | \bfxi_{ij}^{(n)} |^{2+\delta} = O( \| \vecv_n \|_{\ell_1}^{2+\delta} \| \vecw_n \|_{\ell_1}^{2+\delta} ).
	\end{equation}
	Now Jensen's inequality yields
	\[
	\EE \| \bfxi_i^{(n)} \|_2^{2+\delta} = 2^{1+\delta/2} \EE\left( \frac{1}{2} \sum_{j=1,2} [\bfxi_{ij}^{(n)} ]^2 \right)^{1+\delta/2} 
	\le 2^{\delta/2} \sum_{j=1,2} \EE | \bfxi_{ij}^{(n)} |^{2+\delta} < \infty,
	\]
	verifying (I). To show (II) recall that the martingale approximation for 
	$ \vecS_{n',m'}^{(n)} = ( S_{n',m'}^{(n)}( \barr ), S_{n',m'}^{(n)}( \carr ) )^\top $
	is given by 
	$ \delta \vecM_{n',m'}^{(n)} = ( \delta M_{m'+n'}^{(n)}( \barr ), \delta M_{m'+n'}^{(n)}( \carr ) )^\top, $
	see Lemma~\ref{LemmaMartingaleApprox}. Using 
	\[ 
	\EE( \delta M_{m'+n'}^{(n)}( \barr ) \mid \calF_{n,m'} ) = 0, \ \text{and hence} \ 
	\EE( \vecS_{n',m'}^{(n)}  \mid \calF_{n,m'} ) = \EE( \vecS_{n',m'}^{(n)} - \delta \vecM_{n',m'}^{(n)}  \mid \calF_{n,m'} ), 
	\] 
	it follows that
	\[
	\EE \| \EE( \vecS_{n',m'}^{(n)}  \mid \calF_{n,m'} )  \|_2 \stackrel{n,n',m'}{\ll} \| \vecv_n \|_{\ell_1}^2 \| \vecw_n \|_{\ell_1}^2  (n')^{1/2-\theta/2} \stackrel{n,n',m'}{\ll} (n')^{1/2-\theta'/2},
	\]
	by Lemma~\ref{LemmaMartingaleApprox} and (\ref{AssumptionL1Projections}), such that (II) holds with $ \varepsilon = \theta'/2 $. It remains to show (III). 
	Observe that
	\begin{align*}
	& \left\| \frac{1}{n'} \vecC_{n',m'}^{(n)} - \vecT^{(n)} \right\|_{op} \\
	& \qquad =  \sup_{\vecu \in \R^2, \| \vecu \|_2=1} \left|  \vecu^\top \EE\left[  \frac{\vecS_{n',m'}^{(n)}}{ \sqrt{n'} } \frac{\vecS_{n',m'}^{(n)}{}^\top }{ \sqrt{n'} } \mid \calF_{n,m'}  \right] \vecu -  \sum_{j=1,2} u_j 
	\left( \begin{array}{cc} \Cov(B_{n1}, B_{nj} ) \\ \Cov(B_{n2}, B_{nj} ) \end{array} \right)^\top \vecu \right| \\
	& \qquad = \sup_{\vecu \in \R^2, \| \vecu \|_2=1} 
	\left|
	\sum_{i,j=1}^2 u_i u_j \left( \EE\left[ \frac{S_{n',m'}^{(n)}(i)}{ \sqrt{n'} } \frac{S_{n',m'}^{(n)}(j) } { \sqrt{n'} } \mid \calF_{n,m'}   \right] - \Cov( B_{ni}, B_{nj} ) \right)
	\right|.
	\end{align*}
	Noting that $ | u_i u_j | \le \max_{k} u_k^2 \le 1 $, we obtain
	\[
	\left\| \frac{1}{n'} \vecC_{n',m'}^{(n)} - \vecT^{(n)} \right\|_{op} 
	\le 4 \max_{1 \le i, j \le  2} \left|  
	\EE\left[ \frac{S_{n',m'}^{(n)}(i)}{ \sqrt{n'} } \frac{S_{n',m'}^{(n)}(j)}{ \sqrt{n'} } \mid \calF_{n,m'}   \right]
	- \Cov( B_{ni}, B_{nj} )
	\right|.
	\]
	Therefore, (III) follows, if 
	\[
	\EE \left|
	\EE\left[ \frac{S_{n',m'}^{(n)}(i)}{ \sqrt{n'} } \frac{S_{n',m'}^{(n)}(j)}{ \sqrt{n'} } \mid \calF_{n,m'}   \right]
	- \Cov( B_{ni}, B_{nj} ) 
	\right|  
	\stackrel{n,n',m'} \ll \| \vecv_n \|_{\ell_1}^2 \| \vecw_n \|_{\ell_1}^2 (n')^{-\theta/2},
	\]
	a.s., which is shown in Lemma~\ref{KonvKovBeta}, since then assumption (\ref{AssumptionL1Projections}) ensures the estimate $ \stackrel{n,n',m'}{\ll} (n')^{-\theta'/2} $. Hence, from \cite{Philipp1986}, we may conclude that there exists a constant $ C_n $ and a universal constant $ \lambda > 0 $, such that
	\begin{equation}
	\label{StrongApproxBivariate}
	\| \vecD_{nt} - \vecB_n(t) \|_2 \le C_n t^{1/2-\lambda} \qquad t > 0,
	\end{equation}
	a.s., which implies
	\begin{equation}
	\label{StrongApproxCoords}
	| \vecD_{nt}^{(i)} - \vecB_n^{(i)}(t) | \le \sqrt{2} C_n t^{1/2-\lambda}, \qquad t > 0,
	\end{equation}
	a.s., for $ i = 1, 2 $. Recalling that  $ D_{nt} = \vecv_n^\top ( \vecS_{nt} - \EE( \vecS_{nt} ) ) \vecw_n $ where
	$ \vecS_{nt} = \sum_{i \le t} \vecY_{ni} \vecY_{ni}^\top $ satisfies
	\begin{align*}
	\vecS_{nt} & = \eins( t \le \tau ) \sum_{i \le t} \vecY_{ni}( \barr ) \vecY_{ni}( \barr )^\top 
	+ \eins( t > \tau ) \left[ \sum_{i \le \tau} \vecY_{ni}( \barr ) \vecY_{ni}( \barr )^\top  + \sum_{i=\tau+1}^t \vecY_{ni}( \carr ) \vecY_{ni}( \carr )^\top 
	\right],
	\end{align*}
	we have the following crucial representation in terms of $ \vecD_{nt} $,
	\[
	D_{nt} = \vecD_{nt}^{(1)} \eins( t \le \tau ) + [ \vecD_{n\tau}^{(1)} + \vecD_{nt}^{(2)} - \vecD_{n\tau}^{(2)} ] \eins(t > \tau  ),
	\]
	for all $t$. Since
	\begin{align*}
	& D_{nt} - \{ \vecB_n^{(1)}(t) \eins( t \le \tau ) + [ \vecB_n^{(1)}(\tau) + \vecB_n^{(2)}(t) - \vecB_n^{(2)}(\tau) ]  \} \\
	& = \quad \left(  \vecD_{nt}^{(1)} - \vecB_n^{(1)}(t)  \right) \eins( t \le \tau ) + 
%	& \qquad 
	\left( \vecD_{n\tau}^{(1)} - \vecB_n^{(1)}(\tau) + \vecD_{nt}^{(2)} - \vecB_n^{(2)}(t) - \vecD_{nt}^{(2)}(\tau) + \vecB_n^{(2)}( \tau )  \right) \eins( t > \tau ),
	\end{align*}
	(\ref{StrongApproxCoords}) yields, by definition of $ G_n $, see (\ref{DefGaussPro}), 
	\begin{align*}
	| D_{nt} - G_n(t) | & = 
	| D_{nt} - \{ \vecB_n^{(1)}(t) \eins( t \le \tau ) + [ \vecB_n^{(1)}(\tau) + \vecB_n^{(2)}(t) - \vecB_n^{(2)}(\tau) ]  \} | 
 \le 3 \sqrt{2} C_n t^{1/2-\lambda},
	\end{align*}
	for $ t > 0 $, a.s.. This implies 
	\begin{equation}
	\label{FirstApproxBound}
	\frac{1}{ \sqrt{n} } \max_{1 \le k < n} |  D_{nk} - G_n(k) | \le \frac{1}{\sqrt{n}} C_n \max_{1 \le k < n} k^{1/2-\lambda} 
	\le 3 \sqrt{2} C_n n^{-\lambda},
	\end{equation}
	as $ n \to \infty $, a.s., which in turn leads to (iii), since
	\begin{align*}
	\frac{1}{\sqrt{n}} \max_{1 \le k < n} \left|  D_{nk} - \frac{k}{n} D_{nn} - G_n^0(k) \right| 
	& = 
	\frac{1}{\sqrt{n}} \max_{1 \le k < n} \left|  D_{nk} - \frac{k}{n} D_{nn} - \left[  G_{n}(k) - \frac{k}{n} G_{n}(n)   \right] \right| 
	\le 6 \sqrt{2} C_n n^{-\lambda},
	\end{align*}
	as $ n \to \infty $, a.s., and (iv) follows from the reverse triangle inequality. Recalling that $ U_{nk} = \EE( U_{nk} ) + D_{nk} $ and $ U_{nk} - \frac{k}{n}  U_{nn} = m_n(k) + D_{nk} - \frac{k}{n} D_{nn} $, we obtain
	\[
	\frac{1}{\sqrt{n}} \max_{1 \le k < n} \left|  U_{nk} - \frac{k}{n} U_{nn} - \left[ m_n(k) + G_n^0(k) \right] \right|
	\le 6 \sqrt{2} C_n n^{-\lambda},
	\]
	as $ n \to \infty $, a.s., which shows (v). (vi) now follows easily from the reverse triangle inequality. 
%	These bounds extend easily to weights $ g(k/n) $ in the denominator, when considering the maximum over $ n_0 \le k \le n_1 $ with $ n_i = \trunc{n t_i} $, $ i = 0, 1 $, for  $ 0 < t_0 < t_1 < 1 $, since $ g(k/n) \ge g_n^\star := \min \{ g(k_0/n), g( k_1/n ) \} \to \min \{ g(t_0), g(t_1) \} > 0 $. Indeed, (\ref{FirstApproxBound}) now becomes $ n^{-1/2} \max_{n_0 \le k \le n_1} | D_{nk} - G_n(k) | / g(k/n) \le 3 \sqrt{2} (g_n^\star)^{-1} C_n n^{-\lambda} $.	Clearly, under the stated condition all above upper bounds are $ o(1) $, a.s.. 
For a weight function $g$ satisfying (\ref{AssumptionWeightFunction}) the arguments are more involved and as follows:  Let $ \gamma_n $ be a  non-decreasing sequence specified later. Then, using $ g(t) / [t(1-t)]^\beta \ge C_g $ and $ n^2 /(k(n-k)) \le 2 n/k $ for $ 1 \le k \le n/2 $, we obtain a.s.
	\begin{align*}
	& \max_{ \varepsilon n / \gamma_n \le k \le n/2} \frac{1}{\sqrt{n} g(k/n)  } \left| D_{nk} - \frac{k}{n} D_{nn} - G_n^0(k) \right|    \\
	& \qquad \le C_g^{-1}
	  \max_{ \varepsilon n / \gamma_n \le k \le n/2}   \left( \frac{n}{k} \frac{n}{n-k} \right)^\beta\frac{1}{\sqrt{n} } \left| D_{nk} - \frac{k}{n} D_{nn} - G_n^0(k) \right|  \\
	  &\qquad \le C_g^{-1} (2/\varepsilon)^\beta \gamma_n^\beta \max_{1 \le k < n} \frac{1}{\sqrt{n} } \left| D_{nk} - \frac{k}{n} D_{nn} - G_n^0(k) \right| \\
	  & \qquad\le 3 \sqrt{2} C_g^{-1} (2/\varepsilon)^\beta \gamma_n^\beta C_n n^{-\lambda}. 
	 \end{align*}
	The maximum over $ n/2 \le k \le (1-\varepsilon/\gamma_n) n $ is estimated analogously leading to
	\[
	  \max_{ \varepsilon n / \gamma_n \le k \le (1-\varepsilon/\gamma_n) n} \left( \frac{n}{k} \frac{n}{n-k} \right)^\beta \frac{1}{\sqrt{n} } \left| D_{nk} - \frac{k}{n} D_{nn} - G_n^0(k) \right|  = O(  \gamma_n^\beta C_n n^{-\lambda} ), a.s..
	\]
	The right-hand side is $ o(1) $, a.s., if we put $ \gamma_n = n^{0.1 \lambda/\beta} $. 
	Further,  the technical results of the appendix and the H\'ajek-R\'enyi inequality for martingale differences yield for any $ \delta > 0 $ the tail bound
	\begin{align*}
	  \PP\left( \max_{1 \le k \le n \varepsilon / \gamma_n} \left( \frac{n}{k} \right)^\beta \frac{1}{\sqrt{n}} \left| D_{nk} - \frac{k}{n} D_{nn} \right| \ge \delta  \right) 
	  & = O\left( \frac{ n^{2\beta-1} }{ (\delta/2)^2 } \sum_{k=1}^{ n \varepsilon / \gamma_n } k^{-2\beta}   \right) + O\left( \frac{n^{1-\theta}}{\delta^2 n} \right) \\
	  & = O\left( (\varepsilon)^{1-2\beta} \gamma_n^{2\beta-1} ( \log(n \varepsilon / \gamma_n  ) + 1)^{2\beta} \right) + o(1) \\
	  & = O\left( (\varepsilon)^{1-2\beta} n^{0.1(2\beta-1)\lambda/\beta} ( \log(n) +1 )^{2\beta} \right) +  o(1).
	\end{align*}
	The first term tends to $0$, as $ \varepsilon \to 0 $, uniformly in  $n$, since $ \beta < 1/2 $. Let $ B^0 $ be a Brownian bridge and note that $ \{ \alpha_n ^{-1}(\barr) \overline{G}_n^0(t) : 0 \le t \le \vartheta \} \stackrel{d}{=} \{ B^0(t) : 0 \le t \le \vartheta \} $.
	Using the estimates $ \sqrt{t}/t^\beta \le (\varepsilon / \gamma_n )^{1/2-\beta} $ and $ \log_2(1/t) \le \log_2(n) $ on $ t \in \calG_n = \{ 1/n, \dots, \lfloor  \varepsilon n / \gamma_n \rfloor / n \} $ the law of the iterated logarithm for the Brownian bridge, \cite[p.72]{ShorackWellner1986}, 
	 entails for $ \delta > 0 $ and $ \varepsilon / \gamma_n \le \vartheta $   (thus for large $n$)
	\begin{align*}
	  \PP\left( \max_{1 \le k \le n \varepsilon/\gamma_n } \left(\frac{n}{k}\right)^\beta \frac{| n^{-1/2} G_n^0(k) | }{  \alpha_n(\barr) } > \delta  \right) 
	  & \le \PP\left( \sup_{t \in \calG_n}  \frac{| B^0(t) |}{ \sqrt{2 t \log_2(1/t)}} > \frac{\delta}{ \alpha_n(\barr) \sqrt{2 \log_2(n)} ( \varepsilon / \gamma_n )^{1/2-\beta} } \right) \\ &= o(1),
	\end{align*}
	by our choice of $ \gamma_n $ and since $ \beta < 1/2 $.
	The corresponding tail probabilities for the maximum over $ (1-\varepsilon/\gamma_n) n \le k \le n$ are treated analogously. Combining the above estimates  shows (\ref{MainForWeightedCUSUM1}). 
\end{proof}

\begin{proof}[Proof of Theorem~\ref{TestConsistency}]
	
		See appendix.
\end{proof}

\begin{proof}[Proof of Theorem~\ref{BasicStrongApproxCentered}]
	
	See appendix.
 \end{proof}

\begin{proof}[Proof of Theorem~\ref{ApproxGrowing}]
	Since $ \wt\vecv_n = a_n^{-1} \vecv_n $ and $ \wt\vecw_n = b_n^{-1} \vecw_n $ satisfy property (\ref{AssumptionL1Projections}) and  \[ D_{nk}( g; a_n^{-1} \vecv_n, b_n^{-1} \vecw_n ) = a_n^{-1} b_n^{-1} D_{nk}( g;  \vecv_n, \vecw_n ), \] we may conclude that $ T_n(g; \vecv_n, \vecw_n) = T_n(g; \wt\vecv_n, \wt\vecw_n) $. Consequently, all approximations for $ T_n $ carry over. In particular, we obtain under the conditions of Theorem~\ref{BasicStrongApprox}, cf. (\ref{MainForWeightedCUSUM}), 
	\[ \left| T_n(g; \vecv_n, \vecw_n) - \max_{1 \le k < n} \frac{1}{g(k/n)} \biggl| \frac{m_n(k)}{\sqrt{n}} + \overline{B}_n^0(k/n) \biggr| \right| = o_{\PP}(1). \]
	Note that $ \overline{B}_n = \alpha_n^{-1}( \barr ) \overline{G}_n^0(t) $ is a standard Brownian on $ [0, \vartheta] $, cf.  (\ref{Cov_Gn0_prechange}),
	whereas the scale factor changes from $1$ to $ \alpha_n( \carr ) / \alpha_n( \barr ) $ on $ (\vartheta,1] $.
	This shows (\ref{MainForWeightedCUSUM2}) for $ \ell_1$-bounded projections. The proof for uniformly $ \ell_2 $-bounded projections uses the scaling $ a_n = b_n = d_n $ and the fact that by Jensen's inequality gives $ \| \wt\vecv_n \|_{\ell_1} \le \left( \frac{1}{d_n} \sum_{\nu=1}^\infty w_{n\nu}^2 \right) $, where the sum is finite by assumption and the factor cancels by standardization, see also \cite{StelandSachs2018}.
\end{proof}

\begin{proof}[Proof of Theorem~\ref{FCLT}]
	The conditions on $g$ ensure that $ \sup_{0<t<1} | B^0(t) | / g(t) $ is well defined, see \cite{CsoergoeHorvath1993}. Further, $ \{ G_n^0(t) / \alpha^2(\barr) : 0 \le t \le \tau \}  \stackrel{d}{=} \{ B^0(t) : 0 \le t \le \tau \} $ and $  \{ G_n^0(t) / \alpha^2(\carr) : \tau < t \le 1 \}  \stackrel{d}{=} \{ B^0(t) : \tau <  t \le 1 \} $ for each $n$. Therefore, combining these facts, L\'evy's modulus of continuity, $ \omega_{B^0}(a) = \sup_{0 \le t-s \le a} |B^0(t)-B^0(s)| $, of a Brownian bridge $B^0 $, i.e. $ \lim_{a \downarrow 0} \omega_{B^0}(a) / \sqrt{2 a \log(1/a)} = 1 $, a.s.,  and the continuous mapping theorem the result follows from (\ref{MainForWeightedCUSUM})
\end{proof}

\begin{proof}[Proof of Theorem~\ref{MultivCUSUMApprox}]
	Let us stack the statistics $ D_{nk}( \vecv_{nj}, \vecw_{nj} ) $, as defined in (\ref{Def_vecD}), yielding the $ 2L_n $-dimensional random vector
	\[
	\vecD_{nk} = \left(  \begin{array}{c} \vecD_{nk}( \vecv_{n1}, \vecw_{n1}) \\ \cdots \\ \vecD_{nk}( \vecv_{nL}, \vecw_{nL} ) \end{array}  \right) = \sum_{i \le k} \bfxi_i^{(n)}, \qquad k \ge 1, \quad \bfxi_i^{(n)} = \left( \bfxi_{ni}^{(n)}(j) \right)_{j=1}^L.
	\]
	Also put $ \vecS_{n',m'}^{(n)} = \sum_{k=m'+1}^{m'+n'} \bfxi_k^{(n)} $, $ n', m' \ge 0, n \ge 1 $.
	For sparseness of notation, we use the same symbols $ \vecD_{nk}, \vecS_{n',m'}^{(n)} $ and $ \bfxi_k^{(n)} $ and note that the quantities studied here coincide with the previous definitions if $ L_n = 1 $. We work in the Hilbert space $ \R^{2L_n} $ and show (I) - (III) when $ L_n \to \infty $, so that the additional scaling with $ L_n^{-1/2} $, which can be attached to the $ \bfxi_i^{(n)} $'s or put in front of the sums, is in effect.  The equivalence of the vector norms $ \| \cdot \|_2 $ and $ \| \cdot \|_\infty $ - recall that $ \| \cdot \|_\infty \le \| \cdot \|_2 $ and $ \| \cdot \|_2 \le \sqrt{L_n} \| \cdot \|_\infty $ - and Jensen's inequality yield, in view of (\ref{UniformMomentsXi}), 
	\begin{align*}
	\sup_{i \ge 1} \EE \| L_n^{-1/2} \bfxi_i^{(n)} \|_2^{2+\delta}
	 = \sup_{i \ge 1}  \EE \left[ \frac{1}{L_n} \sum_{j=1}^{L_n} \| \bfxi_i^{(n)} \|_2^2 \right]^{(2+\delta)/2} 
	 \le \sup_{i \ge 1} \frac{1}{L_n} \sum_{j=1}^{L_n} \EE \| \bfxi_i^{(n)} \|_2^{2+\delta} < \infty,
	\end{align*}
	since the bounds for $ \EE \| \bfxi_i^{(n)} \|_2^{2+\delta}  $ obtained above and leading to (\ref{UniformMomentsXi}) are uniform in $ i \ge 1 $ and uniform over the considered sets of projection vectors and coefficient arrays. This shows (I).  (II) follows from
	\begin{align*}
	\EE \left\| \EE\left( \vecS_{n',m'}^{(n)} \, | \, \calF_{n,m'} \right)  \right\|_2
	& \le L_n^{1/2} \EE \left\| \EE\left( \vecS_{n',m'}^{(n)} \, | \, \calF_{n,m'} \right)  \right\|_\infty \\
	& \le L_n^{1/2} \EE  \left( \max_{1 \le \ell \le L}  \left | \EE\left( \vecS_{n',m'}^{(n)}(\ell)_1 \, | \, \calF_{n,m'} \right) \right| + \left| \EE\left( \vecS_{n',m'}^{(n)}(\ell)_2 \, | \, \calF_{n,m'} \right) \right| \right)  \\
	& \le  L_n^{1/2} \EE  \left( \left\|  \EE\left( \vecS_{n',m'}^{(n)}(\cdot)_1 \, | \, \calF_{n,m'} \right) \right\|_2  + \left\|  \EE\left( \vecS_{n',m'}^{(n)}(\cdot)_2 \, | \, \calF_{n,m'} \right) \right\|_2\right),
	\end{align*}
	such that $ \EE \left\| \EE\left( L_n^{-1/2} \vecS_{n',m'}^{(n)} \, | \, \calF_{n,m'} \right)  \right\|_2 \stackrel{n,n',m'}{\ll} \| \vecv_n \|_{\ell_1}^2 \| \vecw_n \|_{\ell_1}^2  (n')^{-1/2-\theta/2} \stackrel{n,n',m'}{\ll} (n')^{-1/2-\theta'/2} $, by
	the assumptions on the growth of $\| \vecv_n \|_{\ell_1}^2 \| \vecw_n \|_{\ell_1}^2 $ .
	Next consider the conditional covariance operators 
	\[ 
		C_{n'm'}^{(n)}( \vecu ) = \EE\left( \vecu^\top (L_n^{-1/2} \vecS_{n',m'}^{(n)} ) (L_n^{-1/2} \vecS_{n',m'}^{(n)} | \calF_{n,m'} ) \right), \qquad  \vecu \in \R^{2 L_n}, 
	\] 
	and the covariance operator $ \vecT^{(n)}( \vecu ) = \EE(  \vecu^\top \vecB_n \vecB_n ) $, $ \vecu \in \R^{2L_n} $. We need to estimate the operator norm of their difference and use Lemma~\ref{KonvKovBeta} and similar arguments as in the proof of Theorem~2.2 of \cite{StelandSachs2018}. 
	Denote the $ \nu $th coordinate of $ \vecS_{n',m'}^{(n)} $ corresponding to the weighting vectors $ \vecv_n(\nu) $ and $ \vecw_n( \nu ) $ 
	by $ \vecS_{n',m'}^{(n)}(\nu) $ and let  \[ C_{n'm'}^{(n)}( \nu, \mu ) = \EE(  (L_n^{-1/2} \vecS_{n',m'}^{(n)}(\nu) ) (L_n^{-1/2} \vecS_{n',m'}^{(n)}(\mu) | \calF_{n,m'}  ) ). \]
	By Lemma~\ref{KonvKovBeta}
	\[
	  \EE \max_{1 \le \nu, \mu \le 2 L_n} \left| C_{n'm'}^{(n)}( \nu, \mu ) - \EE( \vecB_n(\nu) \vecB_n(\mu ) ) \right| \ll L_n^{-1} K_n (n')^{\theta/2} \ll L_n^{-1} (n')^{-\theta'/2},
	\]
	where $ \EE( \vecB_n(\nu) \vecB_n(\mu ) ) = L_n^{-1} \beta_n^2( \barr, \vecv_n(\nu), \vecw_n(\nu), \carr, \vecv_n(\mu), \vecw_n(\mu ) ) $. 
	 Using the well known estimate $ | \sum_{i,j} a_{ij} x_i x_j | \le L_n \| \vecx \|_2^2 \max_{i,j} |a_{ij} | $ for $ \vecx = (x_1, \ldots, x_{L_n}) \in \R^k $ and $ a_{ij} \in \R $, $ 1 \le i, j \le L_n $, we therefore obtain 
	\begin{align*}
	\EE \left\| (n')^{-1}C_{n'm'}^{(n)} - \vecT^{(n)}  \right\|_{op} 
	& = \EE \sup_{\vecu \in \R^{2L_n}, \| \vecu \|_2 = 1} \left| \vecu^\top\left( (n')^{-1}C_{n'm'}^{(n)} - \vecT^{(n)}  \right) \vecu  \right|  
	 \stackrel{n}{\ll} (n')^{-\theta'/2},
	\end{align*}
	which establishes condition (III). 	Hence, from \cite{Philipp1986}, we may conclude that there exists a constant $ C_n $ and a universal constant $ \lambda > 0 $, such that on a new probability space for an equivalent version of $ \vecD_{nt} $ and a Brownian motion as described in the theorem
	\[
 	\| \vecD_{nt} - \vecB_n(t) \|_2 \le C_n t^{1/2-\lambda} \qquad t > 0,
	\]
	a.s.. The proof can now be completed along the lines of the proof of Theorem~\ref{BasicStrongApprox} with $ (\vecG_n, \vecG_n^0) $ instead of $ (G_n, G_n^0) $ by arguing coordinate-wise leading to
	\[
	  \left| L_n^{-1/2} C_n( \vecv_{nj}, \vecw_{nj} ) - \max_{1 \le k < n} \frac{1}{\sqrt{n}} | m_{nj}( k ) - G_{nj}^0(k) \right| \le 6 \sqrt{2} C_n n^{-\lambda},
	\]
	where the upper bound does not depend on $j$, which establishes (\ref{StrongApproxMultiv}). For a positive weight function a similar bound applies when considering CUSUMs taking the maximum over $ \{ n_0, \dots,  n_1 \} $ for $ n_i = \trunc{nt_i } $, $ i = 1, 2 $. For a weight function $g$ satisfying  (\ref{AssumptionWeightFunction}) and CUSUMs taking the maximum over $ \{ 1, \ldots, n-1 \} $ the required LIL tail bound and the martingale approximation used to apply the H\'ajek-R\'enyi inequality do not depend on $1 \le j \le L_n $ or $ L_n $, such that 
	\[
		\max_{j \le L_n} \PP \left( \left| L_n^{-1/2} C_n^g( \vecv_{nj}, \vecw_{nj} ) - \max_{1 \le k < n} \frac{1}{\sqrt{n}g(k/n)} | m_{nj}( k ) - G_{nj}^0(k) \right| > \delta \right) = o(1),
	\]
	for any $ \delta > 0 $.
\end{proof}

\subsection{Consistency of nuisance estimators}

\begin{proof}[Proof of Theorem~\ref{LLN_EstAlpha_Uniform}] 
Fix $ 0 < \varepsilon < \vartheta $. We can and will assume that $n$ is large enough to ensure that $ \trunc{n\varepsilon} \ge 1 $ and $ \trunc{n \vartheta} > h $. 
	Denote by $ \wh{\Gamma}_n(h;d) $ the estimator $ \wh{\Gamma}_n(h) $ regarding the dimension $d$ as a formal parameter such that $ \wh{\Gamma}_n(h) = \wh{\Gamma}_n(h;d) \bigr|_{d=d_n} $. In the same vain we proceed for  $ \wh{\beta}_n^2 $ and all other statistics arising below and write $ \wh{\beta}_n^2(d) $ etc. The assertion will then follow by showing that the consistency is uniform in the dimension $d$.   
	By assumption $ z_{ni}^{(j)} = \vecv_{nj}^\top \vecY_{ni} \vecw_{nj}^\top \vecY_{ni} - \EE(\vecv_{nj}^\top \vecY_{ni} \vecw_{nj}^\top \vecY_{ni}) $ satisfies $ z_{ni}^{(j)} = \vecv_{nj}^\top \vecY_{ni}( \barr ) \vecw_{nj}^\top \vecY_{ni}( \barr ) =: z_{ni}^{(j)}( \barr)  $ for $ i \le \tau $ and $ z_{ni}^{(j)}  = \vecv_{nj}^\top \vecY_{ni}( \carr ) \vecw_{nj}^\top \vecY_{ni}( \carr ) =: z_{ni}^{(j)}( \carr ) $ if $ i > \tau $. Put $ \xi_{ni}^{(j)} = z_{ni}^{(j)} - \EE( z_{ni}^{(j)} ) $ and again let $ \xi_{ni}^{(j)}( \barr ) = \xi_{ni}^{(j)}  $, if $ i \le \tau $, and $ \xi_{ni}^{(j)}( \carr ) = \xi_{ni}^{(j)} $, if $ \tau < i \le n $.  By Lemma~\ref{CorollaryBetaSq} and Lemma~\ref{LemmaMartingaleApprox}, $ \beta_n^2(j,k) = n^{-1} \Cov( \vecv_{nj}^\top \matS_{nn} \vecw_{nj}, \vecv_{nk}^\top \matS_{nn} \vecw_{nk}) + R_n $ with $ \EE(R_n^2) = O(n^{-\theta} ) $. Combining this with (\ref{UnifBoundedSumMoments}), we obtain
	$
	 \beta_n^2(j,k) = \sum_{h \in \Z} \EE\left( \xi_{n0}^{(j)} \xi_{n,|h|}^{(k)} \right) + R_n + o(1).
	$
	Without loss of generality we fix $ (j,k) = (1,2) $ and show that $ \sum_{h \in \Z} \wh{\Gamma}_n(h, 1, 2) - \sum_{h \in \Z} \EE( \xi_{n0}^{(1)} \xi_{n,|h|}^{(2)} )  = o(1) $, as $ n \to \infty $, where
	\[
	\wt{\Gamma}_n(u;h) = \wt{\Gamma}_n(u;h,d) = \frac{1}{\trunc{nu}} \sum_{i=1}^{\trunc{nu}-h} \xi_{ni}^{(1)}  \xi_{n,i+h}^{(2)}.
	\]
	Here and in the sequel we omit the dependence of $\wt{\Gamma}_n(u;h) $ and related quantities (namely $ \wt{\Gamma}_n(u;h,d) $ and $ \Gamma(u;h,d) $  introduced below) on $ 1,2 $, for sake of readability. 
	
	Observe that for $ h \ge 0 $
	\begin{align*}
	& \wt{\Gamma}_n(u;h,d) \\
	& = \eins( u \le \vartheta) \frac{1}{\trunc{n u }} \sum_{i=1}^{\trunc{nu}-h} 
	\xi_{ni}^{(1)}( \barr ) \xi_{n,i+h}^{(2)}( \barr ) %\\ & \qquad 
	+ \eins( u > \vartheta) \biggl\{ \frac{\trunc{n\vartheta}-h}{\trunc{nu}}  \frac{1}{\trunc{n \vartheta}-h} \sum_{i=1}^{\trunc{n\vartheta}-h} \xi_{ni}^{(1)}(\barr) \xi_{n,i+h}^{(2)}(\barr)  \\ 
	&  + \frac{h}{\trunc{nu}} \frac{1}{h} \sum_{i=\trunc{n\vartheta}-h+1}^{\trunc{n\vartheta}}  \xi_{ni}^{(1)}(\barr) \xi_{n,i+h}^{(2)}(\carr) 
	%\\ 	& \qquad \quad 
	+ \frac{\trunc{nu}-\trunc{n\vartheta}-h}{\trunc{nu}} \frac{1}{n-\trunc{n\vartheta}-h}  \sum_{i=\trunc{n\vartheta}+1}^{\trunc{nu}-h} 
	\xi_{ni}^{(1)}( \carr ) \xi_{n,i+h}^{(2)}( \carr ) \biggr\}.
	\end{align*}
	Define for $ |h| \le m_n $ and $ \aarr \in \{ \barr, \carr \} $
	\[
	\Gamma(h, d, \aarr ) = \EE\left( \xi_{ni}^{(1)}(\aarr)  \xi_{n,i+|h|}^{(2)}(\aarr) \right).	  
	\]
	Then for $ 0 \le h \le m_n $,
	\begin{align*}
	& \EE(  \wt{\Gamma}_n(u;h,d) ) \\  &= \eins( u \le \vartheta ) \frac{\trunc{nu}-h}{\trunc{nu}}  \Gamma(h, d, \barr )  
	% \\	& \quad 
	+ \eins( u > \vartheta ) \biggl( \frac{\trunc{n\vartheta}-h}{\trunc{nu}} \Gamma(h, d, \barr) + \frac{h}{\trunc{nu}} \EE\left( \xi_{ni}^{(1)}(\barr)  \xi_{n,i+h}^{(2)}(\carr) \right) \\
	& + \frac{\trunc{nu}-\trunc{n\vartheta}-h}{\trunc{nu}} \Gamma(h, d, \carr) 
	\biggr).
	\end{align*}
	Using $ |h| \le m_n = o(n) $ and $ | \frac{ \trunc{na} }{ \trunc{nb} } - a/b | = O( b| \trunc{na}/n - a| + a | \trunc{nb}/n - b| ) = O(1/n) = o(m_n^{-1} ) $ for $ 0 < \varepsilon \le  a, b $, uniformly in $ a, b $,  we obtain
	\[
	\EE(  \wt{\Gamma}_n(u;h) ) = \eins( u \le \vartheta ) \Gamma(h, \barr )  + \eins( u > \vartheta)( (\vartheta/u) \Gamma(h,\barr) + (1 - \vartheta/u) \Gamma(h,\carr) + o( m_n^{-1} )),
	\]
	as $ n \to \infty $, for $ |h| \le m_n $, where the $ o(1) $ term is uniform in $ |h| \le m_n $ and $ u \in [\varepsilon,1] $. Consequently,
	\begin{align*}
	& \sum_{|h| \le m_n} w_{mh} \EE(  \wt{\Gamma}_n(u;h,d) ) 
	 = \eins(u \le \vartheta) \sum_{|h| \le m_n} w_{mh} \Gamma(h,d,\barr)  \\
	& \qquad  + \eins( u > \vartheta) \left(  (\vartheta/u) \sum_{|h| \le m_n} w_{mh} \Gamma(h,d,\barr) + (1-\vartheta/u)  \sum_{|h| \le m_n} w_{mh} \Gamma(h,d,\carr)  \right) + o(1),
	\end{align*}
	as $ n \to \infty $, where the $ o(1) $ term is uniform in $ d \in \N $ and $ u \in [\varepsilon,1] $, such that
	\begin{equation}
	\label{EstimationBn}
	\sup_{u \in [\varepsilon,1]} \sup_{d \in \N} \max_{|h| \le m_n} \left| \sum_{|h| \le m_n} w_{mh} \EE(  \wt{\Gamma}_n(u;h,d) ) - \sum_{| h | \le m_n} w_{mh}  \Gamma(u;h,d) \right| = o(1),
	\end{equation}
	as $ n \to \infty $, where 
	\[ 
	\Gamma(u;h,d) =  \eins(u \le \vartheta) \Gamma(h,d,\barr) + \eins( u > \vartheta)  ( (\vartheta/u) \Gamma(h,d,\barr)  + (1-\vartheta/u)  \Gamma(h,d,\carr) )
	\]
	for $ u \in [\varepsilon, 1] $.
	As in \cite[Th.~4.4]{StelandSachs2017} one can show that 
	\begin{equation} 
	\label{UnifBoundedSumMoments}
	\sum_{h \in \Z} \sup_{d \in \N} \left| \EE\left( \xi_1^{(1)}(\aarr) \xi_{1+h}^{(2)}(\aarr) \right) \right| < \infty 
	\end{equation}
	for $ \aarr \in \{ \barr, \carr \} $ as well as $ \beta^2(\barr) = \sum_{h \in \Z} \Gamma(h, d, \barr) $ and $ \beta^2(\carr) = \sum_{h \in \Z} \Gamma(h, d, \carr) $. This implies 
	\begin{equation}
	\label{ConsistencySeriesConverges}
	\sup_{u \in [\varepsilon,1]}  \sup_{d \ge 1} \sum_{h \in \Z} | \Gamma(u;h,d) | < \infty,
	\end{equation}
	since
	$
	\sum_{h \in \Z} | \Gamma(u;h,d) | \le 2 \sum_{h \in  \Z} | \Gamma(h,d,\barr) | + \sum_{h \in  \Z} | \Gamma(h,d,\carr) |.
	$ Therefore, we may further conclude that 
	\begin{align*}
	\beta^2(u;d) := & \sum_{h \in \Z} \EE(  \wt{\Gamma}_n(u;h,d) ) 
	= \eins(u \le \vartheta) \sum_{h \in \Z} \Gamma(h,d,\barr)  \\
	& \qquad  + \eins( u > \vartheta) \left(  (\vartheta/u) \sum_{h \in \Z} \Gamma(h,d,\barr) + (1-\vartheta/u)  \sum_{h \in \Z}  \Gamma(h,d,\carr)  \right) + o(1),
	\end{align*}
	yielding the representation
	\begin{equation}
	\label{RepresAlpha2}
	\beta^2_n(u;d)  = \sum_{h \in \Z} \Gamma(u;h,d) + o(1)
	\end{equation}
	%where $  \Gamma(u;h,d) = \eins(u \le \vartheta) \Gamma(h,d,\barr) +  \eins( u > \vartheta) [ (\vartheta/u) \Gamma(h,d,\barr) + (1-\vartheta/u) \Gamma(h,d,\carr) ] $, 
	as well as
	\begin{equation}
	\label{RepresAlpha3}
	\beta_n^2(u;d) = \beta^2(u;d) + o(1),
	\end{equation}
	as $ n \to \infty $, uniformly in $ d \in \N $ and $ u \in [\varepsilon,1] $, where
	\[
	\beta^2(u;d) = \sigma^2( u; d, \barr, \carr)  = \eins(u \le \vartheta) \beta^2(\barr) + \eins( u > \vartheta ) ( (\vartheta/u) \beta^2(\barr) + (1-\vartheta/u) \beta^2(\carr) ).
	\]
	The arguments used in the proof of \cite[Th.~4.4]{StelandSachs2017}  to obtain (A.11) therein show that, if applied to the subseries $ \{ \xi_{ni}^{(j)} : 1 \le i \le \trunc{n\vartheta} \} $ and $ \{ \xi_{ni}^{(j)} : \trunc{n \vartheta}+1 \le i \le n-h \} $,
	\begin{equation}
	\label{ConsistencyOrder1}
	\left\| \sum_{i=1}^{\trunc{nu}} \xi_{ni}^{(j)}( \barr ) \right\|_{L_2}^2 = C_1 \trunc{nu}, u \le \vartheta,
	\quad \text{and} \quad
	\left\| \sum_{i=\trunc{n\vartheta}+1}^{\trunc{nu}-h} \xi_{ni}^{(j)}( \carr ) \right\|_{L_2}^2 = C_2\left( \trunc{nu} - \trunc{n\vartheta} \right), u > \vartheta,
	\end{equation}
	for constants $ C_1, C_2 < \infty $ not depending on $ h $, $ j = 1,2 $. Hence
	\begin{align}
	\label{ConsistencyOrder2} 
	\left\| \sum_{i=1}^{\trunc{nu}-h} \xi_{ni}^{(j)} \right\|_{L_2} 
%	&
%	 \le  \eins( u \le \vartheta ) \left\| \sum_{i=1}^{\trunc{n\vartheta}} \xi_{ni}^{(j)}( \barr ) \right\|_{L_2}
%	+ \eins( u > \vartheta ) \left\{ \left\| \sum_{i=1}^{\trunc{n\vartheta}} \xi_{ni}^{(j)}( \barr ) \right\|_{L_2} + \left\| \sum_{i=\trunc{n\vartheta}+1}^{\trunc{nu}-h} \xi_{ni}^{(j)}( \carr ) \right\|_{L_2} \right\} \\
	& \le C_3\left( \max( \sqrt{ \trunc{nu}}, \sqrt{\trunc{n\vartheta}} + \sqrt{\trunc{nu} - \trunc{n\vartheta} - h}) \right),
	\end{align}
	for $ j = 1, 2 $, and in turn
	\begin{equation}
	\label{UnifConvGammaTilde}
	\sup_{u \in (\varepsilon,1]} \sup_{d \in \N} \max_{|h| \le m_n} \| \wt{\Gamma}_n(u;h,d) - \EE (\wt{\Gamma}_n(u;h,d)) \|_{L_1} \le C_4 n^{-1/2}, 
	\end{equation}
	for  constants $ C_3, C_4 < \infty $. Now observe that
	$
	\wh{\Gamma}_n( u; h, d ) = \frac{1}{\trunc{nu}} \sum_{i=1}^{\trunc{nu}-h} ( \xi_{ni}^{(1)} - \overline{\xi}_n^{(1)})( \xi_{n,i+h}^{(2)} - \overline{\xi}_n^{(2)}) 
	$
	where $ \overline{\xi}_{n}^{(j)}(u) = \trunc{nu}^{-1} \sum_{i=1}^{\trunc{nu}} \xi_{ni}^{(j)} $, $ j = 1, 2 $. It holds
	\[
	\trunc{nu} \left( \wh{\Gamma}_n( u; h,d ) - \wt{\Gamma}_n( u; h,d ) \right)
	=
	- \overline{\xi}_n^{(1)}(u) \sum_{j=1}^{\trunc{nu}-h} \xi_{n,j+h}^{(2)} - \overline{\xi}_n^{(2)}(u) \sum_{j=1}^{\trunc{nu}-h}  \xi_{nj}^{(1)} -  \overline{\xi}_n^{(1)}(u) \sum_{j=1}^{\trunc{nu}} \xi_{nj}^{(2)}.
	\]
	Again decomposing the sums as \[ \sum_{j=1}^{\trunc{nu}-h} = \eins(u\le \vartheta) \sum_{j=1}^{\trunc{nu}-h} + \eins(u>\vartheta) \left\{  \sum_{j=1}^{\trunc{n\vartheta}-h} + \sum_{j=\trunc{n\vartheta}-h+1}^{\trunc{n\vartheta}} + \sum_{j = \trunc{n\vartheta}+1}^{\trunc{nu}-h} \right\} \]	and using (\ref{ConsistencyOrder1}), we obtain
	$
	\EE\left(  n | \wh{\Gamma}_n( u; h, d ) - \wt{\Gamma}_n( u; h, d ) | \right) = O(1),
	$
	uniformly over $ |h| \le m_n $, $ d \in \N $ and $ u \in [\varepsilon,1] $. For example, for $ 0 \le h \le m_n $
	\begin{align*}
	\EE \left|  \overline{\xi}_n^{(2)}(u) \sum_{j=1}^{\trunc{nu}-h}  \xi_{nj}^{(1)} \right| 
	& \le \frac{1}{\trunc{nu}} \EE \left|  \sum_{i=1}^{\trunc{nu}} \xi_{ni}^{(2)} \sum_{j=1}^{\trunc{nu}-h} \xi_{nj}^{(1)} \right| 
\le  C_1C_2\left(  \frac{\sqrt{\trunc{nu}} \sqrt{\trunc{nu}-h} }{ \trunc{nu} }   \right) 
	 = O(1),
	\end{align*}
	We may conclude that 
	$
	\sup_{d \in \N} \sup_{u \in [\varepsilon,1]}  m_n \max_{|h| \le m_n} \EE | \wh{\Gamma}_n( u; h,d ) - \wt{\Gamma}_n( u; h, d ) | = O(m_n/n) = o(1),
	$
	as $n \to \infty $, and by boundedness of the weights it follows that
	\[
	\sup_{d \in \N} \EE \left|  \sum_{|h| \le m_n} w_{mh} \wh{\Gamma}_n(h;d) - \sum_{|h| \le m_n} w_{mh} \wt{\Gamma}_n(h;d)  \right| = o(1).
	\]
	Now, having  in mind (\ref{RepresAlpha2}) and (\ref{RepresAlpha3}), decompose
	\begin{align*}
	\sum_{|h| \le m_n} w_{mh} \wt{\Gamma}_n(u;h,d) - \alpha^2(u, \barr, \carr )
	& = \sum_{|h| \le m_n} w_{mh} \left[ \wt{\Gamma}_n(u;h,d) - \Gamma(u;h,d) \right] - \sum_{|h| > m_n} w_{mh}\Gamma(u;h,d),
	\end{align*} 
	 and combine (\ref{EstimationBn}), (\ref{ConsistencySeriesConverges}) and (\ref{UnifConvGammaTilde}), see the appendix for details.
\end{proof}

\subsection{Consistency of the change-point estimators}

\begin{proof}[Proof of Theorem~\ref{UnifConvergence}]
	Observe that, by the definitions of $ U_{nk}, D_{nk} $ and $ \wt{D}_{nk} $,
	\begin{align*}
	\wh{\calU}_n( k ) - \calU_n(k) %& =  
%	\frac{1}{g(k/n) n} \left[ U_{nk} - \frac{k}{n} U_{nn} - \EE\left( U_{nk} - \frac{k}{n} U_{nn} \right) \right] \\
%	& = \frac{1}{g(k/n) n} \left( D_{nk} - \frac{k}{n} D_{nn} \right)  \\
	& = \frac{1}{g(k/n) n} \left( \wt{D}_{nk} - \frac{k}{n} \wt{D}_{nn} \right) + \frac{R_{nk}}{g(k/n)},
	\end{align*} 
	with remainder $ R_{nk} = \frac{1}{n} \left( D_{nk} - \frac{k}{n} D_{nn} - [ \wt{D}_{nk} - \frac{k}{n} \wt{D}_{nn} ]  \right) $. By (\ref{Dnk_MapproxWeighted})  %(\ref{WeightedApproxBound}) 
	we have for $ k \le n/2 $ 
	\[
	  \max_{1 \le k \le n/2} \EE \left( \frac{R_{nk}}{g(k/n)}  \right)^2 \le 2 \max_{1 \le k \le n/2} \EE \left( \left( \frac{n}{k} \right)^\beta R_{nk} \right)^2 \stackrel{n}{\ll} n^{-1-\theta} 
	\]
	and the same bound holds for $ n/2 < k < n $. Therefore for any $ \delta > 0 $
	\[
	 \PP\left( \max_{1 \le k < n} \frac{|R_{nk}|}{g(k/n)} > \delta  \right)
	 \le \PP \left( \sum_{k=1}^n \left( \frac{R_{nk}}{g(k/n)} \right)^2> \delta^2  \right) 
	 \ll n^{-\theta}.
	\]
	Hence, it suffices to show that for all $ \delta > 0 $
%	\begin{align}
%	\label{Consist1}
$	\PP\left( | \wt{D}_{nn} | > \delta n \right) = o(1)$
and
%	\label{Consist2}
$	\PP\left( \max_{1 \le k < n} | \wt{D}_{nk} | > \delta n \right) = o(1) $, 
where the first assertion follows from the latter maximal inequality.
%	\end{align}
%	as $ n \to \infty $, where (\ref{Consist1}) follows from the maximal inequality (\ref{Consist2}). 
Of couse $ \EE( \wt{D}_{nn}^2 ) = O( n ) $, since $ \wt{D}_{nn} $ is the sum of $n$ martingale differences. Now an application of Doob's maximal inequality entails
	$
	\PP\left( \max_{1 \le k < n} | \wt{D}_{nk} |^2 > \delta^2 n^2 \right) 
	= \frac{ \EE( \wt{D}_{nn}^2 ) }{ \delta^2 n^2 } = O\left( \frac{1}{n} \right),
	$
	which establishes \[ \PP\left( \max_{1 \le k < n} | \wt{D}_{nk} | > \delta n \right) = o(1)  \] %(\ref{Consist2}) 
	and in turn (\ref{UnifConvM1}). 
	Next consider 
	\begin{align*}
	\sup_{t \in [0,1]} | \wh{u}_n(t) - u(t) |
%	& \le \sup_{t \in [0,1]} | \wh{u}_n(t) - u_n(t) | + \sup_{t \in [0,1]} | u_n(t) - u(t) | \\
	& \le \max_{1 \le k < n } | \wh{\calU}_n(k) - \calU_n(k) | + \sup_{t \in [0,1]} | u_n(t) - u(t) | 
	 = \sup_{t \in [0,1]} | u_n(t) - u(t) | + o_\PP(1),
	\end{align*}
	as $ n \to \infty $, by (\ref{UnifConvM1}). Clearly, $ u_n(t) \to u(t) $ for each fixed $t$, and by monotonicity on $ [0, \vartheta] $ and $ [\vartheta, 1] $ this implies uniform convergence, since $ g $ is continuous, which completes the proof.
\end{proof}

\begin{proof}[Proof of Theorem~\ref{ConsistencyCPEstimator}] 
	Since $ \vartheta \in (0,1) $ is an isolated maximum of $ u $ and $ \wh{u}_n $ converges uniformly to $ u $, the consistency follows from well known results, see, e.g., \cite{vanderVaart1998}, by virtue of Theorem~\ref{UnifConvergence} and (\ref{RelationMaximizers}). 
\end{proof}

\appendix

\section{Additional Results and Proofs}
\label{Sec:Proofs}

This section provides technical details of several proofs and additional auxiliary results. Especially, the proofs of the asymptotic results are based on martingale approximations which require several additional results and technical preparations. These results extend and complement the results obtained in \cite{StelandSachs2017}.

\subsection{Proofs for Subsection 7.1 (Preliminaries)}

Consider the  multivariate linear time series of dimension $ q = q_n \to \infty $,
\[
\label{MultivLinProc2}
\vecZ_{ni} = \sum_{j=0}^\infty \matB_{nj} \matP \matV^{1/2} \bfeps_{n,i-j}, \qquad i \ge 1, n \in \N, 
\]
with $ \bfeps_{ni} = ( \epsilon_{n,i-r_1}, \ldots, \epsilon_{n,i-r_{d_n}} )^\top,  i \ge 1, n \in \N, $  $ 0 = r_1 < \cdots < r_{d_n} $, cf. (5), which can be written as
\[
\vecZ_{ni} = (Z_{ni}^{(1)}, \ldots, Z_{ni}^{(d_n)} )^\top, \quad Z_{ni}^{(\nu)} =  \sum_{k=0}^\infty \left( \sum_{\ell=1}^q \eins( k \ge r_\ell ) b_{n,k-r_\ell}^{(\nu,\ell)}  \right) \epsilon_{n,i-k}, \quad \nu = 1, \ldots, d_n.
\]  
Let
$
\matP \matV^{1/2} = \sum_{i=1}^q \pi_{ni} \vecl_{ni} \vecr_{ni}^\top
$
be the SVD of $ \matP \matV^{1/2}  $ with singular values $ \pi_{ni} $, left singular vectors $ \vecl_{ni} \in \R^{q_n} $ and right singular vectors $ \vecr_{ni} = ( r_{ni1}, \ldots, r_{nid_n} )^\top \in \R^{d_n} $, which satisfy $ \| \vecl_{ni} \|_{\ell_2} = \| \vecr_{ni} \|_{\ell_2} = 1 $, $ i = 1, \dots, q_n $.  Then
$
\matB_{nj} \matP \matV^{1/2} = \sum_{i=1}^{q_n} \pi_{ni} \matB_{nj} \vecl_{ni} \vecr_{ni}^\top,
$
and the element at position $ (\nu, \ell) $ of the latter matrix is given by 
$ \sum_{i=1}^q \pi_{ni} \vecb_{nj,\nu}^\top \vecl_{ni} r_{ni \ell} $. Plugging the latter formula into equation (4) leads us to
\[
Z_{ni}^{(\nu)} = \sum_{k=0}^\infty \left( \sum_{\ell=1}^{q_n} \eins( k \ge r_\ell )
\sum_{i'=1}^{q_n} \pi_{ni'} \vecb_{n,k-r_\ell,\nu}^\top \vecl_{ni'} r_{ni'\ell}  \right) \epsilon_{n,i-k},
\]
i.e. the coefficients $ c_{nk}^{(Z,\nu)} $ of the series (4) take now the form $ c_{nk}^{(Z,\nu)}  =  \sum_{\ell=1}^{q_n} \eins( k \ge r_\ell ) \sum_{i=1}^{q_n} \pi_{ni} \vecb_{n,k-r_\ell,\nu}^\top \vecl_{ni} r_{ni\ell}  $. 

\begin{lemma} 
	\label{LemmaMLP}
	Suppose that
	\begin{itemize}
		\item[(i)] $ \| \vecl_{nk} \|_{\ell_1}, \| \vecr_{nk} \|_{\ell_1} \stackrel{n,k}{\ll} 1 $ and $ \sum_{i=1}^\infty | \pi_{ni} | = O(1) $,
		\item[(ii)] $ \sup_{n \ge 1} \sup_{1 \le \nu, \mu} | b_{nj}^{(\nu,\mu)} | \stackrel{j,\ell}{\ll} (j+2r_\ell)^{-3/2-\theta} $,
		\item[(iii)] $ \sum_{\ell=1}^\infty r_{\ell}^{-3/4-\theta/2} = O(1). $
	\end{itemize}
	Then $ \sup_{n \ge 1} \max_{1 \le \nu \le d_n} c_{nk}^{(Z,\nu)}  \stackrel{k}{\ll}  k^{-3/4-\theta/2}$, i.e. (D) holds.
\end{lemma}

\begin{proof}[{\bf Proof of  Lemma~\ref{LemmaMLP}}] 
	We can assume that the constants in (i) are equal to $1$. 
	By (i) $ | r_{n k\ell} | \le 1 $ for all $n, k, \ell $. Using the inequality $ | \vecx^\top \vecy | \le \| \vecx \|_{\ell_1} \sup_j | y_j |  $ we obtain for all $ \nu \ge 1 $
	\[
	| \vecb_{k-r_\ell,\nu}^{(\nu)}{}^\top \vecl_{ni} r_{ni\ell} | \le \| \vecl_{ni} \|_{\ell_1} | r_{ni\ell} | \sup_{1 \le \nu, \mu} | b_{k-r_\ell}^{(\nu, \mu)} | \stackrel{k,\ell}{\ll} (k+r_\ell)^{-3/2-\theta/2}. 
	\]
	Combining
	\begin{equation}
	\label{estimate}
	(k+r_\ell)^{-3/2-\theta} = r_\ell^{-3/4-\theta/2} k^{-3/4-\theta/2} \left( \frac{k}{k+r_\ell}  \right)^{3/4+\theta/2} \left(  \frac{r_\ell}{k+r_\ell}  \right)^{3/4+\theta/2}
	\le r_\ell^{-3/4-\theta/2} k^{-3/4-\theta/2} 
	\end{equation}
	with $ \sum_{i=1}^\infty | \pi_{ni} | = O(1) $ now yields
	\[
	\sum_{i=1}^{q_n} |\pi_{ni} \vecb_{n,k-r_\ell,\nu}^\top \vecl_{ni} r_{ni\ell} |
	\stackrel{k,\ell}{\ll} r_\ell^{-3/4-\theta/2} (k+r_\ell)^{-3/4-\theta/2} \sum_{i=1}^\infty | \pi_{ni} |.
	\]
	Using (iii) we may conclude that the coefficients $ c_{nk}^{(Z,\nu)} $ satisfy
	\[
	\sup_{n \ge 1} \max_{\nu \ge 1}  | c_{nk}^{(Z,\nu)}  |
	\le \sum_{\ell=1}^{\infty} \left| \sum_{i=1}^{q_n}  \pi_{ni} \vecb_{n,k-r_\ell,\nu}^\top \vecl_{ni} r_{ni\ell} \right| \stackrel{k}{\ll} k^{-3/4-\theta/2} \sum_{\ell=1}^\infty r_\ell^{-3/4-\theta/2}, 
	\]	
	which completes the proof. 
\end{proof}

Whereas the conditions of  Lemma~\ref{LemmaMLP} rule out eigenvectors such as $ (1/\sqrt{d_n}, \ldots, 1/\sqrt{d_n})^\top $, the following set of conditions relaxes the assumptions on the eigenstructure by strengthening the requirements on the coefficient matrices. 

\begin{lemma}
	\label{LemmaMLP2}
	Suppose that
	\begin{itemize}
		\item[(i)] $ \sum_{k=1}^\infty | \pi_{nk} | = O(1) $ and
		\item[(ii)] $\sup_{n \ge 1} \sup_{\nu \ge 1} \sum_{j=1}^{d_n} | b_{n,k-r\ell}^{(\nu,j)} | \stackrel{k,\ell}{\ll} (k+2r_\ell)^{-3-\theta}. $
	\end{itemize}	
	Then $ \sup_{n \ge 1} \max_{1 \le \nu \le d_n} | c_{nk}^{(Z,\nu)} | \stackrel{k}{\ll}  k^{-3/4-\theta/2}$, i.e. (D) holds.
\end{lemma}

\begin{proof}[{\bf Proof of  Lemma~\ref{LemmaMLP2}}] 
	Recall that $ \| \vecl_{ni} \|_{\ell_2} = \| \vecl_{ni} \|_{\ell_2} = 1 $ for all $n,i$. 
	The proof is similar as the proof of  Lemma~\ref{LemmaMLP} noting that 
	$ \| \vecr_i \|_{\ell_2} $ implies $  | r_{i\ell} | \le 1 $ and using the estimate
	$ | \vecb_{k-r_\ell,\nu}^{(\nu)}{}^\top \vecl_i r_{i\ell} | \le \| \vecb_{k-r_\ell} \|_{\ell_2} \| \vecl_i \|_{\ell_2} | r_{i\ell} | \le \sqrt{\sum_{\mu=1}^{d_n} ( b_{k-r_\ell}^{(\nu,\mu)} )^2 } 
	\stackrel{j,\ell}{\ll} (k+r_\ell)^{-3/2-\theta/2}$.
\end{proof}

Assumption (ii) is a weak localizing condition on the coefficient matrices of the multivariate linear process, as it limits the influence of the $\mu$th innovation on the $ \nu$ th coordinate process at all lags $ j$. Observe that a sufficient condition for (ii) is to assume that
\[
\sup_{n \ge 1} \sup_{1 \le \nu, \mu} |b_{nj}^{(\nu,\mu)} | \stackrel{j,\ell}{\ll} (j+2r_\ell)^{-3/2-\theta/2} / \mu^{1/2+\delta}
\]
for some $ \delta > 0 $. 

Consider a stable VARMA model
\begin{equation}
\label{Varmamodel}
\vecY_{ni} = \matA_{n1} \vecY_{n,i-1} + \ldots + \matA_{np} \vecY_{n,i-p} + \matM_{n1} \bfeps_{n,i-1} + \ldots + \matM_{nr} \bfeps_{n,i-r} + \bfeps_{ni},
\end{equation}
as introduced in the main document with $(d_n \times d_n) $  coefficient matrices $ \matA_{n1}, \ldots, \matA_{np} $ and $ \matM_{n1}, \ldots, \matM_{nr} $ satisfying (element-wise) the decay condition (7) with $ \varpi = 1$ for some $ \delta > 0 $. The coefficient matrices, $  \boldsymbol{\Phi}_{nj} $, of the $MA(\infty) $ representation, 
\[ \vecY_{ni} =  \sum_{j=0}^\infty \boldsymbol{\Phi}_{nj} \bfeps_{n,i-j}, \]
can be calculated using the recursion
\[
\boldsymbol{\Phi}_{n0} = \matid_{d_n},  \boldsymbol{\Phi}_{nj} = \matM_{nj} + \sum_{k=1}^j \matA_{nk} \boldsymbol{\Phi}_{n,j-k},  j \ge 1,
\]
where $ \matM_{nj} = \vecnull $ for $ j > q_n $. Denote $ \matA_{nk} = ( a_{nk}^{(\nu,\mu)} )_{\nu, \mu} $ and $ \boldsymbol{\Phi}_j = ( \Phi_{nj}^{(\nu,\mu)} )_{\nu, \mu} $. 

\begin{lemma} 
	\label{VARMAcondition}
	Suppose that the cofficient matrices of the VARMA model satisfy (7) with $ \varpi = 1$, i.e. for some $ \delta > 0 $ 
	\[
	| a_{nk}^{(\nu,\mu)} |, | m_{nk}^{(\nu,\mu)} | \ll (k+2r_\ell)^{-5/2-\theta} (\nu \mu )^{-1/2-\delta}.
	\]
	Then $ | \Phi_{nj}^{(\nu,\mu)} |  \stackrel{j,\ell}{\ll} (j+2r_\ell)^{-3/2-\theta} (\nu \mu)^{-1/2-\delta} \stackrel{\nu,\mu}{\ll} (j+2r_\ell)^{-3/2-\theta} $, such that (7) is satisfied with $ \varpi = 0 $, and therefore the decay condition (D) holds for the coefficients,
	$ c_{nk}^{(Z,\nu)} = \sum_{\ell=1}^{q_n} \eins( k \ge r_\ell ) \Phi_{n,k-r_\ell}^{(\nu,\ell)} $, in the representation (4) of the multivariate linear process associated to the VARMA model (\ref{Varmamodel}).
\end{lemma}

\begin{proof}[{\bf Proof of  Lemma~\ref{VARMAcondition}}] The proof is by induction. For $ j = 0 $ the assertion follows from $ \Phi_{n0}^{(\nu,\mu)} = \matM_{n0}^{(\nu,\mu)} + \matA_{n0}^{(\nu,\mu)} $. For $j > 1 $ it suffices to show that
	\[
	\sum_{k=1}^j \sum_{s=1}^{d_n} | a_{nk}^{(\nu,s)} \Phi_{n,j-k}^{(s,\mu)} | \stackrel{j,\ell}{\ll} (j+2r_\ell)^{-3/2-\theta} (\nu \mu)^{-1/2-\delta}.
	\]
	Since for $ 1 \le k < j $
	\begin{align*}
	| a_{nk}^{(\nu,s)} | & \stackrel{k,\ell}{\ll} (k+2r_\ell)^{-5/2-\theta} ( \nu s )^{-1/2 - \delta}, \\
	| \Phi_{n,j-k}^{(s, \mu)} | & \stackrel{j,k,\ell}{\ll} (j-k+2r_\ell)^{-3/2-\theta} (\mu s)^{-1/2-\delta}, 
	\end{align*}
	we have
	\begin{align*}
	| a_{nk}^{(\nu,s)}\Phi_{n,j-k}^{(s, \mu)} |
	& \stackrel{j,k,\ell}{\ll}  (k+2 r_\ell)^{-5/2-\theta} (j-k+2r_\ell)^{-3/2-\theta} (s^2 \nu \mu )^{-1/2-\delta} \\
	& \stackrel{j,k,\ell}{\ll}  [(k+2r_\ell)(j-k+2r_\ell)]^{-5/2-\theta} (s^2 \nu \mu )^{-1/2-\delta}.
	\end{align*}
	Using the fact that $ ab \ge a+b-1 $ for $ a,b \ge 1 $, we obtain
	$ (k+2r_\ell)(j-k+2r_\ell) \ge j + 4 r_\ell -1  \ge j + 2 r_\ell $, such that
	\[
	| a_{nk}^{(\nu,s)} \Phi_{n,j-k}^{(s, \mu)} | \stackrel{j,k,\ell}{\ll}  (j+2r_\ell)^{-5/2-\theta} (s^2 \nu \mu )^{-1/2-\delta}.
	\]
	Consequently,
	\[
	\sum_{s=1}^{d_n} | a_{nk}^{(\nu,s)}  \Phi_{n,j-k}^{(s,\mu)} | \stackrel{j,k,\ell}{\ll}  (j+2r_\ell)^{-5/2-\theta}  (\nu \mu)^{-1/2-\delta} \sum_{s=1}^\infty s^{-1-\delta/2}
	\]
	and we may conclude that
	\[
	\sum_{k=1}^j \sum_{s=1}^{d_n} | a_{nk}^{(\nu,s)} | \Phi_{n,j-k}^{(s,\mu)} | \stackrel{j,\ell}{\ll} 
	j  (j+2r_\ell)^{-5/2-\theta}  (\nu \mu)^{-1/2-\delta} 
	\ll (j+2r_\ell)^{-3/2-\theta}  (\nu \mu)^{-1/2-\delta},
	\]
	such that
	\[
	| \Phi_{nj}^{(\nu,\mu)} | \le | \matM_{nj}^{(\nu,\mu)} | +  \sum_{k=1}^j \sum_{s=1}^{d_n} | a_{nk}^{(\nu,s)} \Phi_{n,j-k}^{(s,\mu)} | \stackrel{j,\ell}{\ll} (j+2r_\ell)^{-3/2-\theta} (\nu \mu)^{-1/2-\delta} \stackrel{\nu,\mu}{\ll} (j+2r_\ell)^{-3/2-\theta}.
	\]
	Using again the estimate (\ref{estimate}), we may conclude that
	$ | c_{nk}^{(Z,\nu)} | \le \sum_{\ell=1}^{q_n} | \Phi_{n,k-r_\ell}^{(\nu,\ell)} | \stackrel{\nu}{\ll} k^{-3/4-\theta/2} $.
\end{proof}

We need the following lemma.

\begin{lemma} Under Assumption (D) it holds for weighting vectors $ \vecv_n, \vecw_n, \wt\vecv_n, \wt\vecw_n $ with finite  $ \ell_1 $-norms
	\begin{align}
	\label{FCoeff1}
	& \sum_{i=1}^\infty \sum_{\ell=0}^\infty ( \wt{f}_{\ell,i}^{(n)}( \aarr, \vecv_n, \vecw_n ) - \wt{f}_{\ell,i}^{(n)}( \aarr, \vecv_n, \vecw_n ) )^2 \stackrel{n,n'}{\ll} \| \vecv_n \|_{\ell_1}^2 \| \vecw_n \|_{\ell_1}^2 (n')^{1-\theta}
	\quad \text{for $ n',n = 1, 2, \cdots$}, \\
	\label{FCoeff2}
	& \sum_{k=1}^n \sum_{r=0}^\infty [ \wt{f}_{r+k,0}^{(n)}  ( \aarr, \vecv_n, \vecw_n ) ]^2 \stackrel{n,n'}{\ll} \| \vecv_n \|_{\ell_1}^2 \| \vecw_n \|_{\ell_1}^2 (n')^{1-\theta}
	\quad \text{for $ n',n = 1, 2, \cdots$}, \\
	\label{FCoeff3}
	& \sum_{k=1}^n \sum_{\ell=0}^\infty [ \wt{f}_{\ell,k}^{(n)} ( \aarr, \vecv_n, \vecw_n ) ]^2 \stackrel{n,n'}{\ll} \| \vecv_n \|_{\ell_1}^2 \| \vecw_n \|_{\ell_1}^2 (n')^{1-\theta}
	\quad \text{for $ n',n = 1, 2, \cdots$}.
	\end{align}
	The constants arising in the above estimates do not depend on $ \aarr $.
\end{lemma}

\begin{proof} Omitted for brevity.
\end{proof}

The following lemma provides estimates needed to study a change of the coefficients.

\begin{lemma}
	\label{LemmaFcoeffs}
	Under Assumption (D) it holds for vectors $ \vecv_n, \vecw_n $ with finite $ \ell_1 $-norms:
	\begin{itemize}
		\item[(i)] $ | \wt{f}_{\ell,0}^{(n)}(\aarr, \vecv_n, \vecw_n ) | \stackrel{n}{\ll} \| \vecv_n \|_{\ell_1} \| \vecw_n \|_{\ell_1} \ell^{-3/4-\theta/2} $, $ \ell \ge 1 $.
		\item[(ii)] $ \sum_{j=1}^{n'} \sum_{\ell=j}^{\infty} [ \wt{f}_{\ell,0}^{(n)}( \barr, \vecv_n, \vecw_n ) ]^2 
		\sum_{\ell=j}^{\infty} [ \wt{f}_{\ell,0}^{(n)}( \carr, \wt\vecv_n, \wt\vecw_n ) ]^2 \stackrel{n,n'}{\ll} \| \vecv_n \|_{\ell_1}^2 \| \vecw_n \|_{\ell_1}^2  \| \wt\vecv_n \|_{\ell_1}^2 \| \wt\vecw_n \|_{\ell_1}^2 (n')^{1-\theta} $.
		\item[(iii)] 
		$ \sum_{j=1}^{n'} \sum_{\ell=j}^\infty \left| \wt{f}_{\ell,0}^{(n)}( \barr, \vecv_n, \vecw_n ) \wt{f}_{\ell,0}^{(n)}( \carr, \wt\vecv_n, \wt\vecw_n ) \right| \stackrel{n,n'}{\ll}  
		\| \vecv_n \|_{\ell_1} \| \vecw_n \|_{\ell_1} \| \wt\vecv_n \|_{\ell_1} \| \wt\vecw_n \|_{\ell_1}
		(n')^{1-\theta} $.
	\end{itemize}
\end{lemma}

\begin{proof}[{\bf Proof of  Lemma~\ref{LemmaFcoeffs}}] By virtue of Assumption (D), we have for $ j, \ell \ge 1 $ with $ C_n = C_n(\vecv_n, \vecw_n) =  \| \vecv_n \|_{\ell_1} \| \vecw_n \|_{\ell_1} $ 
	\begin{align*}
	| f_{l,j}^{(n)}( \aarr, \vecv_n, \vecw_n ) |  & \stackrel{n}{\ll} C_n j^{-3/4-\theta/2} (\ell+j)^{-3/4-\theta/2}  \\
	& \stackrel{n}{\ll} C_n j^{-1-\theta/4} j^{1/4-\theta/4} (\ell+j)^{-3/4-\theta/2} \\
	& \stackrel{n}{\ll} C_n j^{-1-\theta/4}  (\ell+j)^{-1/2-\frac34 \theta} \\
	& \stackrel{n}{\ll} C_n j^{-1-\theta/4}  \ell^{-1/2 - \frac34 \theta }.
	\end{align*}
	Hence $ \sum_{j=1}^\infty | f_{\ell,j}^{(n)}( \aarr, \vecv_n, \vecw_n ) |  \stackrel{n}{\ll} C_n \ell^{-1/2- \frac34 \theta} $.
	Since $ | f_{\ell,0}^{(n)}( \aarr, \vecv_n, \vecw_n ) |  \stackrel{n}{\ll} C_n \ell^{-3/4-\theta/2} $ for $ \ell \ge 1 $, we obtain
	\[
	| \wt{f}_{\ell,0}^{(n)}( \aarr, \vecv_n, \vecw_n ) |  \stackrel{n}{\ll} C_n \ell^{-1/2-\frac34 \theta}, \qquad \ell \ge 1.
	\]
	Next we show (ii). Put $ K_n = C_n( \vecv_n, \vecw_n ) C_n( \wt\vecv_n, \wt\vecw_n ) $. By (i) we have for $ j \ge 1 $
	\[
	\sum_{\ell=j}^\infty [ \wt{f}_{\ell,0}^{(n)}( \aarr, \vecv_n, \vecw_n) ]^2 
	\stackrel{n}{\ll} C_n^2 \sum_{\ell=j}^\infty \ell^{-1-\frac32 \theta} 
	\stackrel{n}{\ll} C_n^2 \lim_{b \to \infty} \int_j^b x^{-1-\frac32 \theta} \, dx
	\stackrel{n}{\ll} C_n^2 j^{-\frac32 \theta},
	\]
	such that
	\begin{align*}
	\frac{1}{n'} \sum_{j=1}^{n'} \sum_{\ell=j}^\infty [ \wt{f}_{\ell,0}^{(n)}( \barr, \vecv_n, \vecw_n ) ]^2 \sum_{\ell=j}^\infty [ \wt{f}_{\ell,0}^{(n)}( \carr, \wt\vecv_n, \wt\vecw_n ) ]^2 
	&  \stackrel{n,n'}{\ll} K_n^2 (n')^{-\theta} \sum_{j=1}^{n'} (1/j)^{3\theta} (1/n')^{1-\theta} \\
	& \stackrel{n,n'}{\ll} K_n^2 (n')^{-\theta} \sum_{j=1}^{n'} j^{-1-4\theta} 
	\stackrel{n,n'}{\ll}  K_n^2 (n')^{-\theta},
	\end{align*}
	which establishes (ii). To show (iii), observe that for $ j \ge 1 $ 
	\begin{align*}
	\sum_{\ell=j}^\infty \left| \wt{f}_{\ell,0}^{(n)}( \barr, \vecv_n, \vecw_n ) \wt{f}_{\ell,0}^{(n)}( \carr, \wt\vecv_n, \wt\vecw_n ) \right|
	\stackrel{n}{\ll} K_n \sum_{\ell=j}^\infty   \ell^{-1- \frac32 \theta}  
	\stackrel{n}{\ll} K_n\lim_{b\to \infty} \int_j^b x^{-1- \frac32 \theta} \, dx 
	\stackrel{n}{\ll} K_n j^{-\frac32 \theta}.
	\end{align*}
	Therefore, for $ n' \ge 1 $
	\begin{align*}
	\frac{1}{n'} \sum_{j=1}^{n'} \sum_{\ell=j}^\infty \left| \wt{f}_{\ell,0}^{(n)}( \barr, \vecv_n, \vecw_n ) \wt{f}_{\ell,0}^{(n)}( \carr, \wt\vecv_n, \wt\vecw_n ) \right|
	&  \stackrel{n,n'}{\ll} K_n  \frac{1}{n'} \sum_{j=1}^{n'}  j^{-\frac32 \theta} \\
	&  \stackrel{n,n'}{\ll} K_n (n')^{-\theta} \sum_{j=1}^{n'} (1/j)^{\frac32 \theta} (1/n')^{1-\theta} \\
	& \stackrel{n,n'}{\ll} K_n (n')^{-\theta}  \sum_{j=1}^{\infty} j^{-1-\theta/2} 
	\stackrel{n,n'}{\ll} K_n (n')^{-\theta}.
	\end{align*}
\end{proof}

\subsection{Proofs for Subsection 7.2 (Martingale Approximations)}

Recall the following definitions: Introduce for coefficients $ \aarr $ satisfying Assumption (D) and vectors $ \vecv_n $ and $ \vecw_n $ the $ \calF_{nk} $-martingales
\[
M_k^{(n)}( \aarr, \vecv_n, \vecw_n ) = \wt{f}_{0,0}^{(n)}( \aarr, \vecv_n, \vecw_n ) \sum_{i=0}^k ( \epsilon_{ni}^2 - \sigma_i^2) + \sum_{i=0}^k \epsilon_{ni} \sum_{l=1}^\infty 
\wt{f}_{l,0}^{(n)}( \aarr, \vecv_n, \vecw_n ) \epsilon_{n,i-j}, \qquad k \ge 0,
\]
which start in $ M_0^{(n)} = 0 $, for each $ n \ge 0 $.  Put
\[
S_{n',m'}^{(n)}( \aarr, \vecv_n, \vecw_n ) = \sum_{i=m'+1}^{m'+n'} ( Y_{ni}( \ava ) Y_{ni}( \awa ) - \EE[ Y_{ni}( \ava ) Y_{ni}( \awa ) ] ), \qquad m', n' \ge 0.
\]
Notice that, by definitions (16) and (17) %(\ref{Def_vecU}) and (\ref{Def_vecD}),
\begin{equation}
\label{EqualForApprox}
S_{k,0}^{(n)}( \barr, \vecv_n, \vecw_n)  = \vecD_{nk}^{(1)}, \qquad S_{k,0}^{(n)}( \carr, \vecv_n, \vecw_n )  = \vecD_{nk}^{(2)},
\end{equation}
for $ k \ge 1 $ and $  n \ge 1 $, where $ \vecD_{nk} = ( \vecD_{nk}^{(1)}, \vecD_{nk}^{(2)}) $. 
For brevity introduce the difference operator
\begin{align*}
\delta M_{m'+n'}^{(n)}( \aarr, \vecv_n, \vecw_n ) & = M_{m'+n'}^{(n)}( \aarr, \vecv_n, \vecw_n ) - M_{m'}^{(n)}( \aarr, \vecv_n, \vecw_n )  \\
&= \wt{f}_{0,0}^{(n)}( \aarr, \vecv_n, \vecw_n ) \sum_{i=m'+1}^{m'+n'} ( \epsilon_{ni}^2 - \sigma_{ni}^2) + \sum_{i=m'+1}^{m'+n'} \epsilon_{ni} \sum_{l=1}^\infty 
\wt{f}_{l,0}^{(n)}( \aarr, \vecv_n, \vecw_n ) \epsilon_{n,i-\ell}, \qquad k, n \ge 1,
\end{align*}
which takes the lag $ n' $ forward difference at $ m' $.
Notice that for $ m' = 0 $ 
\begin{equation}
\label{Martingale_m'0} 
\delta M_k^{(n)}( \aarr, \vecv_n, \vecw_n ) = \wt{f}_{0,0}^{(n)}( \aarr, \vecv_n, \vecw_n ) \sum_{i=1}^k ( \epsilon_{ni}^2 - \sigma_{ni}^2) + \sum_{i=1}^k \epsilon_{ni} \sum_{l=1}^\infty  \wt{f}_{l,0}^{(n)}( \aarr, \vecv_n, \vecw_n ) \epsilon_{n,i-\ell}, \qquad k, n \ge 1,
\end{equation}
coincides with the martingale $ M_k^{(n)}( \aarr, \vecv_n, \vecw_n ) $. A direct calculation shows that 
\begin{align}
\label{CovMartingale_shifted}
& \Cov( \delta M_{m'+n'}^{(n)}( \barr, \vecv_n, \vecw_n ), \delta M_{m'+n'}^{(n)}( \carr, \wt\vecv_n, \wt\vecw_n )  )  \\ \nonumber
& \ =  \wt{f}_{0,0}^{(n)}( \barr, \vecv_n, \vecw_n ) \wt{f}_{0,0}^{(n)}( \carr, \wt\vecv_n, \wt\vecw_n ) 
\sum_{j=1}^{n'} ( \gamma_{n,m'+j} + \sigma_{n,m'+j}^4 ) 
% \\ \nonumber & \qquad \qquad
+ \sum_{j=1}^{n'} \sum_{\ell=1}^\infty \wt{f}_{\ell,0}^{(n)}( \barr, \vecv_n, \vecw_n ) \wt{f}_{\ell,0}^{(n)}( \carr, \wt\vecv_n, \wt\vecw_n ) \sigma_{n,m'+j}^2 \sigma_{n,m'+j-\ell}^2,
\end{align}
for $ n', m' \ge 0 $ and $ n \ge 1 $.

\begin{proof}[{\bf Proof of Lemma~1}]
	For $ n, n' \ge 1 $ and $ m' \ge 0 $ put
	\begin{align*}
	\beta_{n,n',m'}^2( \barr, \vecv_n, \vecw_n, \carr, \wt\vecv_n, \wt\vecw_n ) 
	& = \wt{f}_{0,0}^{(n)}( \barr, \vecv_n, \vecw_n ) \wt{f}_{0,0}^{(n)}( \carr, \wt\vecv_n, \wt\vecw_n ) \frac{1}{n'} \sum_{j=1}^{n'} ( \gamma_{n,m'+j} - \sigma_{n,m'+j}^4 ) \\
	& \qquad + \frac{1}{n'} \sum_{j=1}^{n'} \sigma_{n,m'+j}^2 \sum_{\ell=1}^{j-1} \wt{f}_{\ell,0}^{(n)}( \barr, \vecv_n, \vecw_n ) \wt{f}_{\ell,0}^{(n)}( \carr, \wt\vecv_n, \wt\vecw_n ) \sigma_{n,m'+j-\ell}^2.
	\end{align*}
	Then
	\begin{align*}
	& \left| \Cov( \delta M_{m'+n'}^{(n)}( \barr, \vecv_n, \vecw_n ), \delta M_{m'+n'}^{(n)}( \carr, \wt\vecv_n, \wt\vecw_n ) )  - (n') \beta_{n,n',m'}^2( \barr, \vecv_n, \vecw_n, \carr, \wt\vecv_n, \wt\vecw_n )
	\right|  \\
	& \qquad \le \left( \sup_k E( \epsilon_k^2) \right)^2 \sum_{j=1}^{n'} \sum_{\ell=1}^\infty \left| \wt{f}_{\ell,0}^{(n)}( \barr, \vecv_n, \vecw_n ) \wt{f}_{\ell,0}^{(n)}( \carr, \wt\vecv_n, \wt\vecw_n ) \right|  \\
	& \qquad	\stackrel{n,n',m'}{\ll} K_n (n')^{1-\theta},
	\end{align*}
	by  Lemma~\ref{LemmaFcoeffs} (iii).
	We shall prove that one may replace $ \beta_{n,n',m'}^2( \barr, \vecv_n, \vecw_n, \carr, \wt\vecv_n, \wt\vecw_n ) $ by 
	\[
	\beta_{n,n'}^2( \barr, \vecv_n, \vecw_n, \carr, \wt\vecv_n, \wt\vecw_n ) = \wt{f}_{0,0}^{(n)}( \barr, \vecv_n, \vecw_n ) \wt{f}_{0,0}^{(n)}( \carr, \wt\vecv_n, \wt\vecw_n ) ( \gamma_n - s_{n1}^4 ) 
	+ \frac{s_{n1}^4}{n'} \sum_{j=1}^{n'} \sum_{\ell=1}^{j-1} \wt{f}_{\ell,0}^{(n)}( \barr, \vecv_n, \vecw_n ) \wt{f}_{\ell,0}^{(n)}( \carr, \wt\vecv_n, \wt\vecw_n )
	\]
	with an error term of order  $ O( K_n  \| \vecw_n \|_{\ell_1}^2  (n')^{1-\theta} )$.
	Then a further application of  Lemma~\ref{LemmaFcoeffs} (iii) shows that the range of summation for $ \ell $ can be extended to $ \N $ and gives $ | \beta_{n,n'}^2( \barr, \vecv_n, \vecw_n, \carr, \wt\vecv_n, \wt\vecw_n ) - \beta_n^2( \barr, \vecv_n, \vecw_n, \carr, \wt\vecv_n, \wt\vecw_n ) | \stackrel{n,n'}{\ll} K_n  \| \vecw_n \|_{\ell_1}^2  (n')^{1-\theta} $, such that the assertion follows.
	First observe the following fact:
	If $ \{ \alpha_n^*, \alpha_{nk} : k \ge 1 \} \subset \R $, $ n \ge 1 $, satisfy $  (n')^{-1} \sum_{i=1}^{n'} i | \alpha_{ni} - \alpha_n^* | \stackrel{n,n'}{\ll} (n')^{-\beta} $ for some $ \beta > 1 $, then 
	\[
	\sup_{m' \ge 1}  \frac{1}{n'} \sum_{i=m'+1}^{m'+n'}  i | \alpha_{ni} - \alpha_n^*  | \stackrel{n,n',m'}{\ll} (n')^{-1}  \stackrel{n,n',m'}{\ll} (n')^{-\theta} 
	\]
	This follows from
	\begin{align*}
	\frac{1}{n'} \sum_{i=m'+1}^{m'+n'}  i | \alpha_{ni} - \alpha_n^* |
	& = \frac{n'+m'}{n'} \left\{  \frac{1}{n'+m'} \sum_{i=1}^{m'+n'} i |\alpha_{ni} - \alpha_n^*| - \frac{m'}{m'+n'} \frac{1}{m'} \sum_{i=1}^{m'} i | \alpha_{ni} - \alpha_n^* |  \right\} 
	\end{align*}
	which implies 
	\begin{align*}
	\frac{1}{n'} \sum_{i=m'+1}^{m'+n'} i | \alpha_{ni} - \alpha_n^* |  & \stackrel{n,n',m'}{\ll} \frac{n'+m'}{n'} (m'+n')^{-\beta} + \frac{m'}{n'} (m')^{-\beta} \\
	& \stackrel{n,n',m'}{\ll}  (n')^{-1} (m'+n')^{1-\beta} + (n')^{-1} (m')^{1-\beta} 
	\stackrel{n,n',m'}{\ll} (n')^{-1},
	\end{align*}
	since $ \beta > 1 $.  Next observe that
	\begin{equation}
	\label{FTildeBounded}
	\wt{f}_{0,0}^{(n)}( \barr, \vecv_n, \vecw_n ), \max_{j=1, \ldots, n'} \max_{\ell = 1, \ldots, n'-1} \wt{f}_{j-\ell,0}^{(n)}( \barr, \vecv_n, \vecw_n ) \stackrel{n,n}{\ll}  \| \vecv_n \|_{\ell_1} \| \vecw_n \|_{\ell_1} ,
	\end{equation}
	and anolgous estimates hold for the triple $ (\carr, \wt\vecv_n, \wt\vecw_n) $. Because $ \theta \le \beta $ we have $ (n')^{-\beta} \le (n')^{-\theta} $ and therefore by using 
	\begin{equation} 
	\label{ConditionOnVariances}
	\frac{1}{n'} \sum_{i=1}^{n'} i | \sigma_{ni}^2 - s_{n1}^2 | \stackrel{n,n'}{\ll}  (n')^{-\beta}  
	\end{equation}
	and   
	\begin{equation} 
	\label{ConditionOnGammas}
	\frac{1}{n'}  \sum_{i=1}^{n'} i |\gamma_{ni} - \gamma_n | \stackrel{n,n'}{\ll} (n')^{-\beta}  
	\end{equation} 
	for constants $ s_{n1}^2 \in (0,\infty) $ and $ \gamma_n \in \R $ for some $ 1 < \beta < 2 $ with $  1 + \theta  < \beta $, and the decomposition
	\[
	\sigma_{n,m'+j}^4 - s_{n1}^4 = (\sigma_{n,m'+j}^2 - s_{n1}^2)(\sigma_{n,m'+j}^2 + s_{n1}^2) 
	\]
	where $ (n')^{-1} \sum_{j=1}^{n'} \sigma_{n,m'+j}^2 \ll 1 $, we obtain
	\begin{equation}
	\label{BetaStep1}
	\left| \wt{f}_{0,0}^{(n)}( \barr, \vecv_n, \vecw_n ) \wt{f}_{0,0}^{(n)}( \carr, \wt\vecv_n, \wt\vecw_n ) \frac{1}{n'} \sum_{j=1}^{n'} ( \gamma_{n,m'+j} - \sigma_{n,m'+j}^4 ) -  \wt{f}_{0,0}^{(n)}( \barr, \vecv_n, \vecw_n ) \wt{f}_{0,0}^{(n)}( \carr, \wt\vecv_n, \wt\vecw_n ) ( \gamma_n - s_{n1}^4 )  \right| 
	\stackrel{n,n'}{\ll} K_n  (n')^{-\theta}.
	\end{equation}
	Next we show that the second term of $ \beta_{n,n',m'}^2( \barr, \vecv_n, \vecw_n, \carr, \wt\vecv_n, \wt\vecw_n ) $ can be replaced by the second term of $ \beta_n^2(\barr, \vecv_n, \vecw_n, \carr, \wt\vecv_n, \wt\vecw_n) $. Use 
	\begin{align*}
	\sigma_{n,m'+j}^2 \sigma_{n,m'+\ell}^2  - s_{n1}^4  &=   (\sigma_{n,m'+j}^2 - s_{n1}^2) \sigma_{n,m'+\ell}^2  + (\sigma_{n,m'+\ell}^2 - s_{n1}^2) s_{n1}^2 
	\end{align*}
	to otain
	\begin{align*}
	& \frac{1}{n'} \sum_{j=1}^{n'} \sigma_{n,m'+j}^2 \sum_{\ell=1}^{j-1} \wt{f}_{\ell,0}^{(n)}( \barr, \vecv_n, \vecw_n ) \wt{f}_{\ell,0}^{(n)}( \carr, \wt\vecv_n, \wt\vecw_n ) \sigma_{n,m'+j-\ell}^2 -
	\frac{s_{n1}^4}{n'} \sum_{j=1}^{n'} \sum_{\ell=1}^{j-1} \wt{f}_{\ell,0}^{(n)}( \barr, \vecv_n, \vecw_n ) \wt{f}_{\ell,0}^{(n)}( \carr, \wt\vecv_n, \wt\vecw_n ) 
	\\
	& \quad = \frac{1}{n'} \sum_{j=1}^{n'} \sum_{\ell=1}^{j-1} \wt{f}_{j-\ell,0}^{(n)}( \barr, \vecv_n, \vecw_n ) \wt{f}_{j-\ell,0}^{(n)}( \carr, \wt\vecv_n, \wt\vecw_n )[ \sigma_{n,m'+j}^2  \sigma_{n,m'+\ell}^2 - s_{n1}^4] \\
	& \quad = \frac{1}{n'} \sum_{j=1}^{n'} \sum_{\ell=1}^{j-1} \wt{f}_{j-\ell,0}^{(n)}( \barr, \vecv_n, \vecw_n ) \wt{f}_{j-\ell,0}^{(n)}( \carr, \wt\vecv_n, \wt\vecw_n ) (\sigma_{n,m'+j}^2 - s_{n1}^2) \sigma_{n,m'+\ell}^2   \\
	& \qquad + \frac{1}{n'} \sum_{j=1}^{n'}  \sum_{\ell=1}^{j-1} \wt{f}_{j-\ell,0}^{(n)}( \barr, \vecv_n, \vecw_n ) \wt{f}_{j-\ell,0}^{(n)}( \carr, \wt\vecv_n, \wt\vecw_n ) (\sigma_{n,m'+\ell}^2 - s_{n1}^2) s_{n1}^2.
	\end{align*}	
	The first term of the last decomposition can be estimated as follows.
	\begin{align*}
	& \left| \frac{1}{n'} \sum_{j=1}^{n'}  \sum_{\ell=1}^{j-1} \wt{f}_{j-\ell,0}^{(n)}( \barr, \vecv_n, \vecw_n ) \wt{f}_{j-\ell,0}^{(n)}( \carr, \wt\vecv_n, \wt\vecw_n ) (\sigma_{n,m'+j}^2 - s_{n1}^2) \sigma_{n,m'+\ell}^2  \right| \\
	& \stackrel{n,n',m'}{\ll} \frac{K_n}{n'} \sum_{j=1}^{n'}   \sum_{\ell=1}^{j-1}  | \sigma_{n,m'+j}^2 - s_{n1}^2 | \sigma_{n,m'+\ell}^2 \\
	& \le \frac{K_n}{n'} \sum_{j=2}^{n'} (j-1) |\sigma_{n,m'+j}^2  - s_{n1}^2 |  \frac{1}{j-1} \sum_{\ell=1}^{j-1} |\sigma_{n,m'+\ell}^2 - s_{n1}^2| +  \frac{K_n}{n'} \sum_{j=2}^{n'} s_{n1}^2  (j-1) |\sigma_{n,m'+j}^2 - s_{n1}^2| \\
	& \stackrel{n,n',m'}{\ll} \frac{K_n}{n'} \sum_{j=2}^{n'} | \sigma_{n,m'+j}^2 + s_{n1}^2 |  (j-1)^{1-\beta} + K_n (n')^{-\beta}  
	\stackrel{n',m'}{\ll} K_n  (n')^{-\beta}  
	\stackrel{n'}{\ll} K_n  (n')^{-\theta}.
	\end{align*}
	Similarly, for the second term we have
	\begin{align*}
	&\left| \frac{1}{n'} \sum_{j=1}^{n'}  \sum_{\ell=1}^{j-1} \wt{f}_{j-\ell,0}^{(n)}( \barr, \vecv_n, \vecw_n ) \wt{f}_{j-\ell,0}^{(n)}( \carr, \wt\vecv_n, \wt\vecw_n ) (\sigma_{n,m'+\ell}^2 - s_{n1}^2) s_{n1}^2 \right| \\
	& \stackrel{n,n',m'}{\ll} \frac{K_n}{n'} \sum_{j=2}^{n'}  \sum_{\ell=1}^{j-1}  |\sigma_{n,m'+\ell}^2 - s_{n1}^2| 
	\stackrel{n,n'}{\ll} \frac{K_n}{n'} \sum_{j=2}^{n'} (j-1)^{1-\beta} 
	\stackrel{n'}{\ll} K_n (n')^{-\theta'} 
	\stackrel{n'}{\ll} K_n (n')^{-\theta}
	\end{align*}
	for $ \theta < \theta' < \beta-1 $ ($\theta' $ exists, since $ \theta < \beta - 1 $ by assumption).
	Putting things together, we arrive at
	\[
	\left| \frac{1}{n'} \sum_{j=1}^{n'} \sigma_{n,m'+j}^2 \sum_{\ell=1}^{j-1} \wt{f}_{\ell,0}^{(n)}( \barr, \vecv_n, \vecw_n ) \wt{f}_{\ell,0}^{(n)}( \carr, \wt\vecv_n, \wt\vecw_n ) \sigma_{n,m'+j-\ell}^2 -
	\frac{s_1^4}{n'} \sum_{j=1}^{n'} \sum_{\ell=1}^{j-1} \wt{f}_{\ell,0}^{(n)}( \barr, \vecv_n, \vecw_n ) \wt{f}_{\ell,0}^{(n)}( \carr, \wt\vecv_n, \wt\vecw_n )  \right|
	\ll K_n  (n')^{-\theta}.	
	\]
	Combining the latter estimate with (\ref{BetaStep1}) shows that
	\[
	| \beta_{n,n',m'}^2( \barr, \vecv_n, \vecw_n, \carr, \wt\vecv_n, \wt\vecw_n ) - \beta_{n,n'}^2( \barr, \vecv_n, \vecw_n, \carr, \wt\vecv_n, \wt\vecw_n ) | \stackrel{n,n',m'}{\ll}
	K_n  (n')^{-\theta},
	\]
	which completes the proof.
\end{proof}

\begin{proof}[{\bf Proof of Lemma~2}]
	For brevity of notation, we omit the dependence on $ \vecv_n, \vecw_n $ in notation. 
	Put $ R_{n',m'}^{(n)}(\aarr) = S_{n',m'}^{(n)}( \aarr ) - \delta M_{m'+n'}^{(n)}( \aarr )  $ and note that
	$ R_{n'}^{(n)}( \aarr ) = O_{n',m'}^{(n)}( \aarr ) + P_{n',m'}^{(n)}( \aarr ) + Q_{n',m'}^{(n)}( \aarr ) $, where
	\begin{align*}
	Q_{n',m'}^{(n)}( \aarr ) & =  \sum_{i=1}^{n'-1} \sum_{\ell=0}^{n'-i-1} \wt{f}_{\ell,i+1}^{(n)}( \aarr ) ( \sigma_{n,m'-n'-i}^2 \eins( \ell = 0) - \epsilon_{n,m'+n'-i} \epsilon_{n,m'+n'-i-\ell} ),  \\
	P_{n',m'}^{(n)}( \aarr ) & = \sum_{i=0}^\infty \sum_{\ell=0}^\infty ( \wt{f}_{\ell,i+1}^{(n)}( \aarr ) - \wt{f}_{\ell,i+n'+1}^{(n)}( \aarr ) )( \epsilon_{n,m'-i} \epsilon_{n,m'-i-\ell} - \sigma_{n,m'-i}^2 \eins(\ell=0) ), \\
	O_{n',m'}^{(n)}( \aarr ) & = - \sum_{i=0}^{n'-1} \sum_{k=n'}^\infty \wt{f}_{k-i,i+1}^{(n)}( \aarr ) \epsilon_{n,m'+n'-k} \epsilon_{n,m'+n'-i},
	\end{align*}
	for $ n', m' \ge 0 $. The result can now be shown along the lines of \cite[Lemma~2]{Kouritzin1995} noting the following facts. By independence of $ \{ \epsilon_{nk} : k \in \Z \} $, for any fixed $n$,
	\[
	\EE( Q_{n',m'}^{(n)} )^2 = \sum_{i=0}^{n'-1} [ \wt{f}_{0,i+1}^{(n)}( \aarr ) ]^2 ( \gamma_{n,m'+n'-i} - \sigma_{n,m'+n'-i}^4  )
	+ \sum_{i=0}^{n'-1} \sum_{\ell=1}^{n'-i-1} [\wt{f}_{\ell,i+1}^{(n)}( \aarr )]^2 \sigma_{n,m'-n'-i}^2 \sigma_{n,m'-n'-i-\ell}^2.
	\]
	By virtue of (\ref{FCoeff1}) - (\ref{FCoeff3}), which hold due to the decay assumption (D) on the coefficients of the vector time series, we obtain
	\[
	\EE( Q_{n',m'}^{(n)} )^2 \stackrel{n,n',m'}{\ll} \sum_{i=0}^{n'-1} [\wt{f}_{0,i+1}^{(n)}( \aarr ) ]^2 +
	\sum_{i=0}^{n'-1} \sum_{\ell=1}^{n'-i-1} [ \wt{f}_{\ell,i+1}^{(n)}( \aarr ) ]^2 \stackrel{n,n',m'}{\ll} \| \vecv_n \|_{\ell_1}^2 \| \vecw_n \|_{\ell_1}^2  (n')^{1-\theta},
	\]
	since by assumption $ \sup_{n \ge 1} \sup_{k \ge 1} | \gamma_{n,k} | < \infty $ and $ \sup_{n \ge 1 } \sup_{k \ge 1} \EE | \epsilon_{n,k} |^2 < \infty $, which entails that the rate $ (n')^{1-\theta} $ also applies to innovation arrays $ \{ \epsilon_{nk} : k \in \Z, n \in \N \} $ satisfying Assumption (E). Analogously,
	\[
	\EE( P_{n',m'}^{(n)} )^2 \stackrel{n,n',m'}{\ll} \sum_{i=1}^{\infty} \sum_{\ell=0}^\infty \left( \wt{f}_{\ell,i}^{(n)}( \aarr ) - \wt{f}_{\ell,i+n'}^{(n)}(\aarr)  \right)^2 \stackrel{n,n',m'}{\ll} \| \vecv_n \|_{\ell_1}^2 \| \vecw_n \|_{\ell_1}^2 (n')^{1-\theta}.
	\]
	Lastly, 
	\[
	\EE( O_{n',m'}^{(n)}( \aarr ) )^2 = \EE \lim_{N \to \infty} \left( \sum_{i=0}^{n'-1} \sum_{k=n'}^N \wt{f}_{k-i,i+1}^{(n)}( \aarr ) \epsilon_{n,m'+n'-k} \epsilon_{n,m'+n'-i}  \right)^2,
	\]
	and therefore Fatou's lemma leads to $ \EE(  O_{n',m'}^{(n)}( \aarr ) )^2  \stackrel{n,n',m'}{\ll} \| \vecv_n \|_{\ell_1}^2 \| \vecw_n \|_{\ell_1}^2 (n')^{1-\theta} $. 
	Repeating the arguments provided in \cite{Kouritzin1995} shows that $ R_{n'} $ can be decomposed in three terms which are bounded by the expressions listed in (\ref{FCoeff1}) - (\ref{FCoeff3}), which are $ \stackrel{n',n}{\ll} (n')^{1-\theta} $. 
	Observing that the dependence on the vectors $ \vecv_n, \vecw_n, \wt\vecv_n, \wt\vecw_n $ is only through the coefficients $ \wt{f}_{\ell,j}( \cdot ) $, the remaining assertions follow by recalling (\ref{EqualForApprox}).
\end{proof}

Observe that for i.i.d. error terms with $ \EE( \epsilon_k^2 ) = \sigma^2 $ and $ \EE( \epsilon_k^3 ) = \gamma $ for all $k$
\begin{equation}
\label{FormulaBetaSq2}
\beta_n^2( \barr, \vecv_n, \vecw_n, \carr, \wt\vecv_n, \wt\vecw_n ) = \wt{f}_{0,0}^{(n)}( \barr, \vecv_n, \vecw_n ) \wt{f}_{0,0}^{(n)}( \carr, \wt\vecv_n, \wt\vecw_n ) (\gamma + \sigma^4) 
+ \sigma^4 \sum_{\ell=1}^{\infty} \wt{f}^{(n)}_{\ell,0}( \barr, \vecv_n, \vecw_n ) \wt{f}^{(n)}_{\ell,0}( \carr, \wt\vecv_n, \wt\vecw_n ).
\end{equation}
We write $ \alpha^2_n( \aarr, \vecv_n, \vecw_n) = \beta_n^2( \aarr, \vecv_n, \vecw_n, \aarr, \vecv_n, \vecw_n ) $, $ \aarr \in \{ \barr, \carr \} $. If $ c_{nj}^{(\nu)} = c_j^{(\nu)} $, $ n, \nu \ge 1$, then for projections $ \vecv, \vecw \in \ell_1 $ these quantities do not depend on $n$.

\begin{proof}[{\bf Proof of Lemma~3}] We may apply the method of proof of \cite{Kouritzin1995}. 
	Define for a sequence  $ \aarr $ satisfying Assumption (D) the following approximation for $ M_{m'+j}^{(n)}( \aarr ) - M_{m'}^{(n)}( \aarr ) $,
	\[
	L_j( \aarr, \vecv_n, \vecw_n ) = L_{n,m',j}( \aarr, \vecv_n, \vecw_n ) = \sum_{k=m'+1}^{m'+j} \{  \wt{f}_{0,0}^{(n)}( \aarr, \vecv_n, \vecw_n )( \epsilon_{nk}^2 - \sigma_{nk}^2 ) 
	+ \sum_{\ell=1}^{k-m'-1} \wt{f}_{\ell,0}^{(n)}( \aarr, \vecv_n, \vecw_n ) \epsilon_{nk} \epsilon_{n,k-\ell} \}
	\]
	and let 
	\[
	K_j( \aarr, \vecv_n, \vecw_n ) = M_{m'+j}^{(n)}( \aarr, \vecv_n, \vecw_n ) - M_{m'}^{(n)}( \aarr, \vecv_n, \vecw_n ) - L_j( \aarr, \vecv_n, \vecw_n ) 
	\]
	denote the associated approximation error, for $ j > 0 $, and $ K_j( \aarr, \vecv_n, \vecw_n ) = 0 $, if $ j = 0 $. Then the lag $1$ martingale differences attain the representation
	\[
	\Delta M_{m'+j}^{(n)}( \aarr, \vecv_n, \vecw_n ) := M_{m'+j}^{(n)}( \aarr, \vecv_n, \vecw_n ) - M_{m'+j-1}^{(n)}( \aarr, \vecv_n, \vecw_n ) = \Delta L_j( \aarr, \vecv_n, \vecw_n ) + \Delta K_j( \aarr, \vecv_n, \vecw_n ).
	\]
	where 
	\begin{equation}
	\label{FormulaDeltaLj}
	\Delta L_j( \aarr, \vecv_n, \vecw_n ) = \wt{f}_{0,0}^{(n)}( \aarr, \vecv_n, \vecw_n )( \epsilon_{n,m'+j}^2 - \sigma_{n,m'+j}^2 ) + \sum_{\ell=1}^{j-1} \wt{f}_{\ell,0}^{(n)}( \aarr, \vecv_n, \vecw_n ) \epsilon_{n,m'+j} \epsilon_{n,m'+j-\ell}.
	\end{equation}
	Using the fact that the lag $ n' $ difference operator is the sum of first order differences, i.e.,
	$ \delta M_{m'+n'}^{(n)}( \aarr, \vecv_n, \vecw_n ) = M_{m'+n'}^{(n)}( \aarr, \vecv_n, \vecw_n )  - M_{m'}^{(n)}( \aarr, \vecv_n, \vecw_n ) = \sum_{j=1}^{n'} \Delta M_{m'+j}^{(n)}( \aarr, \vecv_n, \vecw_n )  $, we obtain
	\begin{align*}
	E_{n'}^{(n)}  & = \left\|
	\sum_{j=1}^{n'} \left\{
	\EE [ \delta M_{m'+j}^{(n)}( \barr, \vecv_n, \vecw_n ) \delta M_{m'+j}( \carr, \wt\vecv_n, \wt\vecw_n ) | \calF_{n,m'} ] - \beta_n^2( \barr, \vecv_n, \vecw_n, \carr, \wt\vecv_n, \wt\vecw_n ) \right\}
	\right\|_{L_1} \\
	& \le \left\| 
	\sum_{j=1}^{n'} \{ \EE [ \Delta L_j( \barr, \vecv_n, \vecw_n ) \Delta L_j( \carr, \wt\vecv_n, \wt\vecw_n ) | \calF_{n,m'}  ] - \beta_n^2( \barr, \vecv_n, \vecw_n, \carr, \wt\vecv_n, \wt\vecw_n ) \} 
	\right\|_{L_1} \\
	& \qquad + \left\|
	\sum_{j=1}^{n'} \EE [ \Delta K_j( \barr, \vecv_n, \vecw_n ) \Delta L_j( \carr, \wt\vecv_n, \wt\vecw_n ) | \calF_{n,m'}  ]  
	\right\|_{L_1} 
	+ \left\|
	\sum_{j=1}^{n'} \EE [ \Delta L_j( \barr, \vecv_n, \vecw_n ) \Delta K_j( \carr, \wt\vecv_n, \wt\vecw_n ) | \calF_{n,m'}  ]  
	\right\|_{L_1} \\
	& \qquad + \left\|
	\sum_{j=1}^{n'} \EE [ \Delta K_j( \barr, \vecv_n, \vecw_n ) \Delta K_j( \carr, \wt\vecv_n, \wt\vecw_n ) | \calF_{n,m'}  ]  
	\right\|_{L_1}. \\
	\end{align*}
	Let us now estimate the four terms separately.	For brevity of notation we omit the dependence on $, \vecv_n, \vecw_n, \wt\vecv_n, \wt\vecw_n $, as they are attached to $ \barr $ and $ \carr $, respectively, and enter only through the coefficients $ \wt{f}_{\ell,j}^{(n)} $. Using (\ref{FormulaDeltaLj}) we have
	\begin{align*}
	\Delta L_j( \barr ) \Delta L_j( \carr ) & = \wt{f}_{0,0}^{(n)}( \barr ) \wt{f}_{0,0}^{(n)}( \carr ) ( \epsilon_{n,m'+j}^2 - \sigma_{n,m'+j}^2 )^2 
	+ \wt{f}_{0,0}^{(n)}( \barr ) ( \epsilon_{n,m'+j}^2 - \sigma_{n,m'+j}^2 ) \epsilon_{n,m'}^2 \sum_{\ell=1}^{j-1} \wt{f}_{\ell,0}^{(n)}( \carr ) \epsilon_{n,m'+j-\ell} \\
	& +	 \wt{f}_{0,0}^{(n)}( \carr ) ( \epsilon_{n,m'+j}^2 - \sigma_{n,m'+j}^2 ) \epsilon_{n,m'}^2 \sum_{\ell=1}^{j-1} \wt{f}_{\ell,0}^{(n)}( \barr ) \epsilon_{n,m'+j-\ell} 
	+ \sum_{\ell, \ell'=1}^{j-1} \wt{f}_{\ell,0}^{(n)}( \barr ) \wt{f}_{\ell,0}^{(n)}( \carr ) \epsilon_{n,m'}^2 \epsilon_{n,m'+j-\ell}  \epsilon_{n,m'+j-\ell'}
	\end{align*}
	Noting that $ j \ge 1 $ and the sums over $ \ell, \ell' $ are non-vanishing only if $ j \ge 2 $, we obtain by independence
	of $ \{ \epsilon_{nk} : k \in \Z \} $ for $ j \ge 2 $ with the centered r.v.s. $  \bar\epsilon_{n,m'+j}^2 = \epsilon_{n,m'+j}^2 - \sigma_{n,m'+j}^2 $
	\begin{align*}
	& \EE \left[  \wt{f}_{0,0}^{(n)}( \barr ) \bar \epsilon_{n,m'+j}^2  \epsilon_{n,m'}^2 \sum_{\ell=1}^{j-1} \wt{f}_{\ell,0}^{(n)}( \carr ) \epsilon_{n,m'+j-\ell}   \bigr| \calF_{n,m'} \right]  %\\
	%	& \qquad = 
	= \wt{f}_{0,0}^{(n)}( \barr ) \epsilon_{nm'}^2 \EE(  \bar \epsilon_{n,m'+j}^2 ) \sum_{\ell=1}^{j-1}  \wt{f}_{\ell,0}^{(n)}( \carr )  \EE( \epsilon_{n,m'+j-\ell}  ) = 0,
	\end{align*}	
	a.s., since  $ \epsilon_{n,m'+j} $ and $ \epsilon_{n,m'+j-\ell}  $ are  independent if $ j \ge 2 $ and $ \ell \ge 1 $. Therefore
	\begin{align*}
	\EE[ \Delta L_j( \barr ) \Delta L_j( \carr )  \mid \calF_{n,m'}  ] 
	& = \wt{f}_{0,0}^{(n)}( \barr ) \wt{f}_{0,0}^{(n)}( \carr )( \gamma_{n,m'+j} - \sigma_{n,m'+j}^4)  
	+ \sum_{\ell=1}^{j-1} \wt{f}_{\ell,0}^{(n)}( \barr ) \wt{f}_{\ell,0}^{(n)}( \carr )  \sigma_{n,m'+j}^2 \sigma_{n,m'+j-\ell}^2,
	\end{align*}
	a.s.. Consequently, cf. (\ref{CovMartingale_shifted}), 
	\begin{align*}
	& \sum_{j=1}^{n'} \{ \EE [ \Delta L_j( \barr ) \Delta L_j( \carr ) | \calF_{n,m'}  ] - \beta_n^2( \barr, \carr ) \} \\
	& = \wt{f}_{0,0}^{(n)}( \barr ) \wt{f}_{0,0}^{(n)}( \carr ) \sum_{j=1}^{n'} ( \gamma_{n,m'+j} - \sigma_{n,m'+j}^4) 
	+ \sum_{j=1}^{n'}  \sigma_{n,m'+j}^2 \sum_{\ell=1}^{j-1} \wt{f}_{\ell,0}^{(n)}( \barr ) \wt{f}_{\ell,0}^{(n)}( \carr )  \sigma_{n,m'+j-\ell}^2 - (n') \beta_n^2( \barr, \carr ), \\
	& \ll K_n (n')^{1-\theta}.
	\end{align*}
	a.s.. A lengthy calculation shows that 
	$
	\Delta K_j( \aarr ) = \epsilon_{n,m'+j}  \sum_{\ell = j}^{\infty} \wt{f}_{\ell,0}^{(n)}( \aarr ) \epsilon_{n,m'+j-\ell},
	$
	such that, because  $ \epsilon_{n,m'+j-\ell} \epsilon_{n,m'+j-\ell'} $ is $ \calF_{n,m'} $-measurable if $ \ell, \ell' \ge j $ and $ j \ge 1 $,
	\[
	\EE[ \Delta K_j( \barr ) \Delta K_j( \carr)  \mid \calF_{n,m'} ] = \sigma_{n,m'+j}^2 
	\left( \sum_{\ell=j}^\infty \wt{f}_{\ell,0}^{(n)}( \barr ) \epsilon_{n,m'+j-\ell }  \right)
	\left( \sum_{\ell=j}^\infty \wt{f}_{\ell,0}^{(n)}( \carr ) \epsilon_{n,m'+j-\ell }  \right),
	\]
	such that the Cauchy-Schwarz inequality provides us with the bound
	\[
	\| \EE[ \Delta K_j( \barr ) \Delta K_j( \carr)  \mid \calF_{n,m'} ]  \|_{L_1} 
	\le \left( \sup_{n,k \ge 1} \sigma_{nk}^2 \right)^3 \prod_{\aarr \in \{ \barr, \carr \}} \sqrt{ \sum_{\ell=j}^\infty [ \wt{f}_{\ell,0}^{(n)}( \aarr ) ]^2 },
	\]
	a.s.. This implies
	\begin{align*}
	\left\| \sum_{j=1}^{n'}  \EE[ \Delta K_j( \barr ) \Delta K_j( \carr)  \mid \calF_{n,m'} ]  \right\|_{L_1} 
	& \le \sum_{j=1}^{n'} \|  \EE[ \Delta K_j( \barr ) \Delta K_j( \carr)  \mid \calF_{n,m'} ]  \|_{L_1}  \\
	& \stackrel{n,n'}{\ll} \sum_{j=1}^{n'}  \prod_{\aarr \in \{ \barr, \carr \}} \sqrt{ \sum_{\ell=j}^\infty [ \wt{f}_{\ell,0}^{(n)}( \aarr ) ]^2 }  
	\stackrel{n,n'}{\ll} K_n ( n' )^{1-\theta/2},
	\end{align*}
	since by virtue of the Jensen inequality and Lemma~\ref{LemmaFcoeffs}
	\[
	\left(  \frac{1}{n'} \sum_{j=1}^{n'}  \sqrt{ \prod_{\aarr \in \{ \barr, \carr \}} \sum_{\ell=j}^\infty [ \wt{f}_{\ell,0}^{(n)}( \aarr ) ]^2  }  \right)^2 
	\le
	\frac{1}{n'} \sum_{j=1}^{n'} \prod_{\aarr \in \{ \barr, \carr \}} \sum_{\ell=j}^\infty [ \wt{f}_{\ell,0}^{(n)}( \aarr ) ]^2.
	\]
	Further,
	\[
	\EE[ \Delta L_j( \barr ) \Delta K_j( \barr ) \mid \calF_{n,m'}  ] 
	= \wt{f}_{0,0}^{(n)}( \barr) E( \epsilon_{nk}^3) \sum_{\ell=j}^\infty \wt{f}_{\ell,0}^{(n)}( \carr ) \epsilon_{n,m'+j-\ell}
	\]
	leading to the estimate
	\[
	\| \EE[ \Delta L_j( \barr ) \Delta K_j( \carr ) \mid \calF_{n,m'}  ]  \|_{L_1}
	\le \sup_{n,k} \EE | \epsilon_{nk} |^3 \sup_{n,k} \EE( \epsilon_{nk}^2) \wt{f}_{\ell,0}^{(n)}( \barr )\sqrt{ \sum_{\ell=j}^\infty [ \wt{f}_{\ell,0}^{(n)}( \carr ) ]^2 }
	\stackrel{n,n'}{\ll} K_n (n')^{1-\theta/2}.
	\]
	Lastly, a direct calculation using similar arguments as above shows that 
	\begin{align*}
	\EE[ \Delta_j( \barr ) \Delta L_j( \carr ) | \calF_{n,m'}  ] & = \wt{f}_{0,0}^{(n)}( \barr ) \wt{f}_{0,0}^{(n)}( \carr ) [ \EE( \epsilon_{n,m'+j}^4 ) - \sigma_{n,m'+j-\ell}^4 ] +
	\sum_{\ell =1}^{j-1} \wt{f}_{\ell,0}^{(n)}( \barr ) \wt{f}_{\ell,0}^{(n)}( \carr ) \sigma_{n,m'+j}^2 \sigma_{n,m'+j-\ell}^2,
	\end{align*}
	which is the $j$th term of the first sum of $ (n') \beta_{n,n',m'}^2( \barr) $ as defined in the proof of Lemma~1. %\ref{CorollaryBetaSq}. 
	Since there it was shown that $ | \beta_{n,n',m'}^2( \barr) - \beta_n^2(\barr) | \stackrel{n,n'}{\ll}  \| \vecv_n \|_{\ell_1}^2 \| \vecw_n \|_{\ell_1}^2 (n')^{-\theta} $, we eventually obtain
	\[
	\sum_{j=1}^{n'} \EE[ \Delta_j( \barr ) \Delta L_j( \carr ) | \calF_{n,m'}  ] - ( n')  \beta_n^2( \barr, \carr ) 
	\ll  K_n (n')^{1-\theta/2}.
	\]
	Putting together the above estimates completes the proof of the first assertion. The second assertion is shown as in 
	\cite[(4.23)]{Kouritzin1995} and is omitted for brevity.
\end{proof}

\subsection{Proofs of Subsection~7.3}

\begin{lemma} 
	\label{MeansofCusums}
	Under the change-point model (9) % (\ref{CPM}) 
	it holds
	\begin{equation}
	\label{Mean1}
	\EE \left( \matS_{nk} - \frac{k}{n} \matS_{nn} \right)
	= \left\{
	\begin{array}{cc}
	\frac{k(n-\tau)}{n}( \bfSigma_{n0} - \bfSigma_{n1} ), & \qquad k \le \tau, \\
	\tau \frac{n-k}{n} ( \bfSigma_{n0} - \bfSigma_{n1} ), & \qquad k > \tau.
	\end{array}
	\right.
	\end{equation}
	and
	\begin{equation}
	\label{Mean2}
	m_n(k) := \EE \left( U_{nk} - \frac{k}{n} U_{nn} \right)
	= \left\{
	\begin{array}{cc}
	\frac{k(n-\tau)}{n} \Delta_n, & \qquad k \le \tau, \\
	\tau \frac{n-k}{n} \Delta_n, & \qquad k > \tau.
	\end{array}
	\right.
	\end{equation}
\end{lemma}

\begin{proof}[{\bf Proof of  Lemma~\ref{MeansofCusums}}]
	Since $ \EE( \vecY_{ni} \vecY_{ni}^\top ) = \bfSigma_{n0} \eins( i \le \tau ) + \bfSigma_{n1} \eins( i > \tau ) $, we have 
	$
	\EE\left( \frac{1}{n} \matS_{nn} \right) = \EE\left( \frac{1}{n} \sum_{i=1}^n \vecY_{ni} \vecY_{ni}^\top \right) = \frac{\tau}{n} \bfSigma_{n0} + \frac{n-\tau}{n} \bfSigma_{n1}.
	$
	Therefore, for $ k \le \tau $ 
	%	\begin{align*}
	$	\EE\left( \matS_{nk} - \frac{k}{n} \matS_{nn} \right) = 
	k \bfSigma_{n0} - k\left( \frac{\tau}{n} \bfSigma_{n0} + \frac{n-\tau}{n} \bfSigma_{n1}  \right) 
	= \frac{k(n-\tau)}{n}(\bfSigma_{n0} - \bfSigma_{n1}), $
	%	\end{align*}
	whereas for $ k > \tau $ 
	%	\begin{align*}
	$	\EE\left( \matS_{nk} - \frac{k}{n} \matS_{nn} \right) = 
	\tau \bfSigma_{n0} + (\tau-k) \bfSigma_{n1} - k\left( \frac{\tau}{n} \bfSigma_{n0} + \frac{n-\tau}{n} 	\bfSigma_{n1}  \right) 
	= \tau \frac{n-k}{n}(\bfSigma_{n0} - \bfSigma_{n1}).$
	%	\end{align*}
	This verifies (\ref{Mean1}). Recalling that $ U_{nk} = \vecv_n^\top \matS_{nk} \vecw_n $ and
	$ \Delta_n = \vecv_n^\top\bfSigma_{n0} \vecw_n - \vecv_n^\top\bfSigma_{n1} \vecw_n $, 
	(\ref{Mean2}) follows by linearity.
\end{proof}

\begin{proof}[{\bf Proof of Theorem~2}]
	Observe that
	$
	\max_{k \le n} \left| U_{nk} - \frac{k}{n} U_{nn} \right| 
	= \max_{k \le n} \frac{1}{\sqrt{n}} \left|  D_{nk} - \frac{k}{n}  D_{nn} + m_n(k) \right|.
	$
	In  Theorem~\ref{Mart_for_Dnk} it is shown that $ \wt{D}_{nk} $, defined in (\ref{DefDTilde}), is a martingale which approximates $ D_{nk} $, since $ \EE( \wt{D}_{nk} - D_{nk} )^2 \stackrel{n,k}{\ll} n^{1-\theta} $,
	so that  $ \EE( \wt{D}_{nk} )^2 \stackrel{n,k}{\ll} k $ as well as $ \EE( \wt{D}_{nk}  - \frac{k}{n} \wt{D}_{nn} )^2 \stackrel{n,k}{\ll} n $ hold. Using this fact, (\ref{Dnk_Mapprox2}) and the triangle inequality, we obtain for any constant $ C > 0 $ and $ k \le n $
	\begin{align*}
	\PP\left( \frac{1}{\sqrt{n}} \left|  D_{nk} - \frac{k}{n} D_{nn} \right| > C \right)
	& \le \PP \left( \left|  \wt{D}_{nk} - \frac{k}{n} \wt{D}_{nn} \right| > \frac{C \sqrt{n} }{2}  \right) 
	+ \PP \left( \left| D_{nk} - \frac{k}{n} D_{nn} - [ \wt{D}_{nk} - \frac{k}{n} \wt{D}_{nn} ]  \right| >  \frac{C\sqrt{n}}{2}  \right) \\
	& \stackrel{n}{\ll} \frac{ \EE\left( \wt{D}_{nk} - \frac{k}{n} \wt{D}_{nn} \right)^2 }{ C n } + \frac{ n^{1-\theta} }{ C n  },
	\end{align*}
	which entails $ \frac{1}{\sqrt{n}} \left|  D_{nk} - \frac{k}{n} D_{nn} \right| = O_\PP(1) $.  W.l.o.g. assume $ \Delta_n > 0 $ for large $n$ and observe that by (11) %(\ref{ThereisAChangeCondition})  
	it holds
	$
	m_n( \tau + 1 ) =\frac{ \trunc{n \vartheta} }{ \sqrt{n} } \frac{n - \trunc{n  \vartheta} - 1 }{n} \Delta_n
	\to +\infty, 
	$
	as $ n \to \infty $. Consequently, we have 
	\begin{align*}
	\max_{k \le n} \frac{1}{\sqrt{n}} \left| U_{nk} - \frac{k}{n} U_{nn} \right| 
	& \ge \frac{1}{\sqrt{n}} \left| D_{n,\tau+1} - \frac{\tau+1}{n} D_{nn} + m_n(\tau+1)  \right| \\
	& \ge \frac{1}{\sqrt{n}} \left|  | D_{n,\tau+1} - \frac{\tau+1}{n} D_{nn} | - | m_n( \tau+1) | \right| 
	\underset{n \to \infty}{\overset{\PP}{\to}} +\infty.
	\end{align*}
\end{proof}

\begin{proof}[{\bf Proof of Theorem~3}]
	Observe that $ \Cov( Y_{ni}^{(\nu)}(\barr), Y_{ni'}^{(\mu)}( \carr )  ) = \sum_{\ell=0}^\infty b_{n\ell}^{(\nu)} c_{n, i'-i+\ell}^{(\mu)} \sigma_{i-j}^2 $ for $ i \le i' $ and arbitrary coefficient arrays $ \barr, \carr $. By the strengthened decay condition, we have
	$ \EE( \vecv_n^\top \vecY_{ni}(\aarr) )^2 = O(1) $ uniformly in $i$,  
	$ \EE\left( n^{-1/2} \sum_{i=1}^k \vecv_n^\top \vecY_{ni}(\aarr) \right)^2 = O(1) $ and
	\[
	\EE( \vecv_n^\top \overline{\vecY}_{n}(\aarr) )^2 = O\left( \| \vecv_n \|_{\ell_1}^2 \frac{1}{n^2} \sum_{i,j=1}^n \Cov(Y_{ni}^{(\nu)}(\aarr), Y_{nj}^{(\nu)}(\aarr)  )  \right) \stackrel{n}{\ll} n^{-1},
	\]
	for $ \aarr \in \{ \barr, \carr \} $. It follows that 
	\begin{align*}
	\EE \left| \vecw_n^\top \overline{\vecY}_n \sum_{i \le k} \vecv_n^\top \vecY_{ni}( \aarr ) \right| 
	& = \EE \left| n^{-1/2} \sum_{i \le k} \vecv_n^\top \vecY_{ni} n^{-1/2} \sum_{i \le k} \vecw_n^\top \vecY_{ni} \right|  = O(1) 
	\end{align*}
	and therefore 
	\[ R_{ni}( \aarr ) = - \vecw_n^\top \overline{\vecY}_{n}( \aarr ) \vecv_n^\top \vecY_{ni}(\aarr)  - \vecv_n^\top \overline{\vecY}_{n}( \aarr )  \vecw_n^\top \vecY_{ni}(\aarr)  + \vecv_n^\top \overline{\vecY}_{n}(\aarr)  \vecw_n^\top \overline{\vecY}_{n}(\aarr) 
	\]
	satisfies $ \sup_{i \ge 1} \EE| R_{ni}( \aarr ) | \ll n^{-1/2} $ and $ \EE \left| \sum_{i \le k} R_{ni}( \aarr ) \right| = O(k/n) $. 
	Put
	\[
	\wt\bfxi_{ni} = \wt\bfxi_{ni}( \avb, \avc ) = \left( \begin{array}{cc} \wt Y_{ni}( \avb ) \wt Y_{ni}( \awb ) - \EE[ Y_{ni}( \avb ) Y_{ni}( \awb ) ]\\ \wt Y_{ni}( \avc ) \wt Y_{ni}( \awc ) - \EE[  Y_{ni}( \avc ) Y_{ni}( \awc ) ]\end{array}  \right),
	\]
	where $ \wt{Y}_{ni}(\cdot) = Y_{ni}(\cdot) - \overline{Y}_{n}( \cdot ) $. 
	Since  $ \wt\bfxi_{ni} = \bfxi_{ni} + \vecR_{ni} $, $ \vecR_{ni} = ( R_{n}(\barr), R_{n}(\carr) )^\top $,  we have for $ k \le n $ the estimates $ \EE \| \bfxi_{ni} - \wt\bfxi_{ni} \|_\infty = \EE \| \vecR_{ni} \|_\infty \ll n^{-1/2} $ uniformly in $ i $. Consider the decomposition
	$ 	\wt\vecD_{nk} = \vecD_{nk} + \vecR_n $ if $ \vecR_n = \sum_{i \le k} \vecR_{ni} $.
	By Markov's inequality
	$ \PP( | \sum_{i \le k} R_{ni}( \aarr ) | > \delta n^{\lambda'} ) \le \delta^{-1} n^{-\lambda'} \EE | \sum_{i \le k} R_{ni}( \aarr ) | = O(  n^{-\lambda'}  ) $, $ \delta, \lambda' > 0$, such that
	$ \vecR_n  = o_{\PP}( n^{\lambda'} ) $ for any $ \lambda'>0$. It follows that 
	\[ 
	\wt\vecD_{nk} = \vecD_{nk} + \vecR_n, \qquad \| \vecR_n \|_\infty = o_{\PP}( n^{1/2} ).  
	\]
	In view of (55) %(\ref{StrongApproxBivariate}), 
	we may conclude that, on a new probability space for equivalent versions, $ \| \vecD_{nt} - \vecB_n(t) \|_2 \le C_n t^{1/2-\lambda} $, $ t > 0 $, a.s.. By virtue of \cite[Sec.~21, Lemma~2]{Billingsley1999} this strong approximation can be constructed on the original probability space $ (\Omega, \calA, \PP) $. Consequently, we obtain $ \| \wt\vecD_{nt} - \vecB_n(t) \|_2 \le C_n t^{1/2-\lambda} + o_{\PP}( n^{1/2} ) $, $ t > 0 $, a.s.. Now it follows easily that  assertions (i) and (ii) of Theorem~1 %Theorem~\ref{BasicStrongApprox} 
	hold true with an additional error term $  o_{\PP}( n^{1/2} )  $ and (iii)-(vi) with an additional $ o_{\PP}(1) $ term. Finally, (vii) and (viii) hold in probability, if $ C_n n^{-\lambda} = o(1) $.
\end{proof}

\subsection{Proofs of Subsection~7.5}

\begin{theorem}
	\label{Mart_for_Dnk} Under the change-point alternative model (9) %(\ref{CPM}) 
	with $ \tau = \lfloor n \vartheta \rfloor $, $ \vartheta \in (0, 1) $, there exist a $ \calF_{nk} $-martingale array $ \wt{D}_{nk} $, $ 1 \le k $, $ n \ge 1 $, such that
	\begin{equation}
	\label{Dnk_Mapprox1}
	\EE( D_{nk} - \wt{D}_{nk} )^2 \stackrel{n,k}{\ll} k^{1-\theta}.
	\end{equation}
	and hence for  $ k \le n $ and $ n \ge 1 $
	\begin{equation}
	\label{Dnk_Mapprox2}
	\EE\left( D_{nk} - \frac{k}{n} D_{nn} - [\wt{D}_{nk} - \frac{k}{n} \wt{D}_{nn} ] \right)^2 \stackrel{n,k}{\ll} k^{1-\theta}.
	\end{equation}
	Further, if $ 0 < \beta < 1/2 $, then
	\begin{equation}
	\label{Dnk_MapproxWeighted}
	\EE \left(  \left( \frac{n}{k} \right)^\beta \left|  D_{nk} - \frac{k}{n} D_{nn} - [ \wt{D}_{nk} - \frac{k}{n} \wt{D}_{nn} ] \right| \right)^2 \stackrel{n,k}{\ll} n^{1-\theta}.
	\end{equation}
\end{theorem} 

\begin{proof}[{\bf Proof of  Theorem~\ref{Mart_for_Dnk}}]
	Recall (\ref{Martingale_m'0}) and put for each $n \ge 1 $
	\begin{equation}
	\label{DefDTilde}
	\wt{D}_{nk} = \delta M_k^{(n)}( \barr)  \eins( k \le \tau ) + [ \delta M_\tau^{(n)}( \barr ) + \delta M_k^{(n)}( \carr ) - \delta M_\tau^{(n)}( \carr )  ] \eins( k > \tau ), \qquad k \ge 1.
	\end{equation}
	It is clear that $ \EE( \wt{D}_{nk} | \calF_{n,k-1} ) = 0 $ holds if $ k \le \tau $ and $ k > \tau+1 $. In addition, for $ k = \tau + 1 $ we have
	\begin{align*}
	\EE[ \wt{D}_{nk} - \wt{D}_{n,k-1} | \calF_{n,k-1}  ] 
	&= \EE[ \delta M_\tau^{(n)}( \barr ) + \delta M_{\tau+1}^{(n)}( \carr ) - \delta M_\tau^{(n)}( \carr ) - \delta M_\tau^{(n)}( \barr ) \mid \calF_{n,\tau} ] \\
	& = \EE[ M_{n,\tau+1}^{(n)}( \carr ) - M_{n,\tau}^{(n)}( \carr ) \mid \calF_{n,\tau} ] = 0,
	\end{align*}
	because $ M_{nk}^{(n)}( \carr ) $ is a $ \calF_{nk} $-martingale array.
	Since 
	\[
	D_{nk} = \vecD_{nk}^{(1)} \eins( k \le \tau ) + [ \vecD_{n\tau}^{(1)} + \vecD_{nk}^{(2)} - \vecD_{n\tau}^{(2)} ] \eins(k > \tau  ),
	\]
	the triangle inequality provides the upper bound
	\[
	\| \vecD_{nk}^{(1)} - \delta M_k^{(n)}(\barr)  \|_{L_2} 
	+ \| \vecD_{n\tau}^{(1)} - \delta M_\tau^{(n)}( \barr )   \|_{L_2}	
	+ \| \vecD_{nk}^{(2)} - \delta M_k^{(n)}( \carr ) \|_{L_2}  + \| \vecD_{n\tau}^{(2)} - M_\tau^{(n)}( \carr ) \|_{L_2} 
	\]
	for $ \| D_{nk} - \wt{D}_{nk} \|_{L_2} $, which is $ \ll k^{1/2-\theta/2} $ by virtue of
	Lemma~2, see (49) and (50),
	% Lemma~\ref{LemmaMartingaleApprox}, see (\ref{MartingaleApproxD1}) and (\ref{MartingaleApproxD2}). 
	This verifies (\ref{Dnk_Mapprox1}), i.e.,
	\[
	\| D_{nk} - \wt{D}_{nk} \|_{L_2} \stackrel{n,k}{\ll} k^{1/2-\theta/2}.
	\]
	As a consequence, for $ k \le n $ and $ n \ge 1 $
	\begin{align*}
	& \left\|  D_{nk} - \frac{k}{n} D_{nn} - [ \wt{D}_{nk} - \frac{k}{n} \wt{D}_{nn} ] \right\|_{L_2} 
	\stackrel{k,n}{\ll} \| D_{nk} - \wt{D}_{nk} \|_{L_2} + \frac{k}{n} \| D_{nn} - \wt{D}_{nn} \|_{L_2} 
	\stackrel{k,n}{\ll} k^{1/2-\theta/2}.
	\end{align*}
	Lastly, since we may assume that $ \theta $ is small enough to ensure $ \theta + 2 \beta < 1 $, it holds for $ 1 \le k \le n $
	\begin{equation}
	\label{WeightedApproxBound}
	\EE \left(  \left( \frac{n}{k} \right)^\beta \left|  D_{nk} - \frac{k}{n} D_{nn} - [ \wt{D}_{nk} - \frac{k}{n} \wt{D}_{nn} ] \right| \right)^2 \stackrel{n,k}{\ll} n^{2 \beta} k^{-2\beta} k^{1-\theta} \le n^{1-\theta},
	\end{equation}
	%	which completes the proof.
\end{proof}

\begin{proof}[{\bf Proof of Theorem~7}]
	The proof is completed by considering the decomposition
	\begin{align*}
	\sum_{|h| \le m_n} w_{mh} \wt{\Gamma}_n(u;h,d) - \alpha^2(u, \barr, \carr )
	& = \sum_{|h| \le m_n} w_{mh} \wt{\Gamma}_n(u;h,d) - \sum_{h \in \Z} \Gamma(u;h,d) + o(1) \\
	& = A_n(u;d) + B_n(u;d) + C_n(u;d) + D_n(u;d) + o(1),
	\end{align*}
	where  the $o(1)$ term is uniform over $ d \in \N $  and
	\begin{align*}
	A_n(u;d) & = \sum_{|h| \le m_n} w_{mh} [ \wt{\Gamma}_n(u;h,d) - \EE(\wt{\Gamma}_n(u;h,d) ) ],  \quad 
	B_n(u;d)  = \sum_{|h| \le m_n} w_{mh} [ \EE(\wt{\Gamma}_n(u;h,d)) - \Gamma(u;h,d) ], \\
	C_n(u;d) & = \sum_{|h| \le m_n} [w_{mh} - 1 ] \Gamma(u;h,d),  \quad
	D_n(u;d)  = -\sum_{|h| > m_n} \Gamma(u;h,d).
	\end{align*}
	$ B_n(u;d) $ has been already estimated in (60) and (62) %(\ref{EstimationBn}) and (\ref{ConsistencySeriesConverges}) 
	implies $ \sup_{u \in [\varepsilon,1]} \sup_{d \in \N} | D_n(u;h,d) | = o(1) $, as $ n \to \infty $.  Denote the counting measure on $ \Z $ by $ d \nu $. Then by Fubini and (65) % (\ref{UnifConvGammaTilde})
	\begin{align*}
	\sup_{u \in [\varepsilon,1]} \sup_{d \in \N} \EE| A_n(u;d) | 
	& \le  \sup_{u \in [\varepsilon,1]} \sup_{d \in \N} \int_{\Z} w_{mh} \EE | \wt{\Gamma}_n(u;h,d) - \EE\wt{\Gamma}_n(u;h,d) | \eins( |h| \le m_n) \, d\nu(h) \\
	& \le 2W m_n \sup_{u \in [\varepsilon,1]} \sup_{d \in \N} \max_{|h|\le m_n} \EE | \wt{\Gamma}_n(u;h,d) - \EE\wt{\Gamma}_n(u;h,d) | = o(1).
	\end{align*}
	Lastly, $ \sup_{u \in [\varepsilon,1]} \sup_{d \in \N} | C_n(u;h,d) | = o(1) $, as $ n \to \infty $, follows by dominated convergence. 
\end{proof}

\section*{Acknowledgments}

The author acknowledges support from Deutsche Forschungsgemeinschaft (grants STE 1034/11-1,  1034/11-2).

%\section*{References}

\bibliographystyle{apalike}
\bibliography{lit}

\begin{thebibliography}{}

\bibitem[Aue et~al., 2009]{Aue2009}
Aue, A., H\"{o}rmann, S., Horv\'{a}th, L., and Reimherr, M. (2009).
\newblock Break detection in the covariance structure of multivariate time
  series models.
\newblock {\em Ann. Statist.}, 37(6B):4046--4087.

\bibitem[Avanesov and Buzun, 2018]{Avanesov2018}
Avanesov, V. and Buzun, N. (2018).
\newblock Change-point detection in high-dimensional covariance structure.
\newblock {\em arxiv:1610.03783}.

\bibitem[Bar-Joseph et~al., 2012]{BarJosephGitterEtAl2012}
Bar-Joseph, Z., Gitter, A., and Simon, I. (2012).
\newblock Studying and modelling dynamic biological processes using time-series
  gene expression data.
\newblock {\em Nature Reviews Genetics}, 13:552--564.

\bibitem[Berkes et~al., 2009]{BerkesGombayHorvath2009}
Berkes, I., Gombay, E., and Horv\'{a}th, L. (2009).
\newblock Testing for changes in the covariance structure of linear processes.
\newblock {\em J. Statist. Plann. Inference}, 139(6):2044--2063.

\bibitem[Billingsley, 1999]{Billingsley1999}
Billingsley, P. (1999).
\newblock {\em Convergence of probability measures}.
\newblock Wiley Series in Probability and Statistics: Probability and
  Statistics. John Wiley \& Sons, Inc., New York, second edition.
\newblock A Wiley-Interscience Publication.

\bibitem[Birnbaum et~al., 2013]{BirnbaumJohnstonNadlerPaul2013}
Birnbaum, A., Johnstone, I.~M., Nadler, B., and Paul, D. (2013).
\newblock Minimax bounds for sparse {PCA} with noisy high-dimensional data.
\newblock {\em Ann. Statist.}, 41(3):1055--1084.

\bibitem[Breitung and Eickmeier, 2011]{BreitungEickmeier2011}
Breitung, J. and Eickmeier, S. (2011).
\newblock Testing for structural breaks in dynamic factor models.
\newblock {\em J. Econometrics}, 163(1):71--84.

\bibitem[Brodie et~al., 2009]{BrodieEtAl2009}
Brodie, J., Daubechies, I., De~Mol, C., Giannone, D., and Loris, I. (2009).
\newblock Sparse and stable {M}arkowitz portfolios.
\newblock {\em Proceedings the National Academy of Sciences of the United
  States of America}, 106(30):12267--12272.

\bibitem[Cai et~al., 2015]{CaiMaWu2015}
Cai, T., Ma, Z., and Wu, Y. (2015).
\newblock Optimal estimation and rank detection for sparse spiked covariance
  matrices.
\newblock {\em Probab. Theory Related Fields}, 161(3-4):781--815.

\bibitem[Cho and Fryzlewicz, 2015]{ChoFryzlewicz2015}
Cho, H. and Fryzlewicz, P. (2015).
\newblock Multiple-change-point detection for high dimensional time series via
  sparsified binary segmentation.
\newblock {\em J. R. Stat. Soc. Ser. B. Stat. Methodol.}, 77(2):475--507.

\bibitem[Cs\"{o}rg\H{o} and Horv\'{a}th, 1993]{CsoergoeHorvath1993}
Cs\"{o}rg\H{o}, M. and Horv\'{a}th, L. (1993).
\newblock {\em Weighted approximations in probability and statistics}.
\newblock Wiley Series in Probability and Mathematical Statistics: Probability
  and Mathematical Statistics. John Wiley \& Sons, Ltd., Chichester.
\newblock With a foreword by David Kendall.

\bibitem[Cs\"{o}rg\H{o} and Horv\'{a}th, 1997]{CsoergoeHorvath1997}
Cs\"{o}rg\H{o}, M. and Horv\'{a}th, L. (1997).
\newblock {\em Limit theorems in change-point analysis}.
\newblock Wiley Series in Probability and Statistics. John Wiley \& Sons, Ltd.,
  Chichester.
\newblock With a foreword by David Kendall.

\bibitem[Davis and Kahan, 1970]{DavisKahan1970}
Davis, C. and Kahan, W.~M. (1970).
\newblock The rotation of eigenvectors by a perturbation. {III}.
\newblock {\em SIAM J. Numer. Anal.}, 7:1--46.

\bibitem[Erichson et~al., 2018]{ErichsonEtAl2018}
Erichson, N.~B., Zheng, P., Manohar, K., Brunton, S.~L., Kutz, J.~N., and
  Aravkin, A.~A. (2018).
\newblock Sparse principal component analysis via variable projection.
\newblock {\em arXiv:1804.00341}.

\bibitem[Ferger, 2018]{Ferger2018}
Ferger, D. (2018).
\newblock On the supremum of a {B}rownian bridge standardized by its maximizing
  point with applications to statistics.
\newblock {\em Statist. Probab. Lett.}, 134:63--69.

\bibitem[Galeano and Pe\~{n}a, 2007]{GaleanoPena2007}
Galeano, P. and Pe\~{n}a, D. (2007).
\newblock Covariance changes detection in multivariate time series.
\newblock {\em J. Statist. Plann. Inference}, 137(1):194--211.

\bibitem[Han and Inoue, 2015]{HanInoue2015}
Han, X. and Inoue, A. (2015).
\newblock Tests for parameter instability in dynamic factor models.
\newblock {\em Economet. Theor.}, 31(5):1117--1152.

\bibitem[Horv\'{a}th and Rice, 2019]{HorvathRice2019}
Horv\'{a}th, L. and Rice, G. (2019).
\newblock Asymptotics for empirical eigenvalue processes in high-dimensional
  linear factor models.
\newblock {\em J. Multivariate Anal.}, 169:138--165.

\bibitem[Johnson and Lindenstrauss, 1984]{JohnsonLindenstraussJoram1984}
Johnson, W.~B. and Lindenstrauss, J. (1984).
\newblock Extensions of {L}ipschitz mappings into a {H}ilbert space.
\newblock In {\em Conference in modern analysis and probability ({N}ew {H}aven,
  {C}onn., 1982)}, volume~26 of {\em Contemp. Math.}, pages 189--206. Amer.
  Math. Soc., Providence, RI.

\bibitem[Johnstone and Lu, 2009]{JohnstoneLu2009}
Johnstone, I.~M. and Lu, A.~Y. (2009).
\newblock On consistency and sparsity for principal components analysis in high
  dimensions.
\newblock {\em J. Amer. Statist. Assoc.}, 104(486):682--693.

\bibitem[Jolliffe et~al., 2003]{JollifeEtAl2003}
Jolliffe, L., Trendafilov, N., and Uddin, M. (2003).
\newblock A modified principal component technique based on the lasso.
\newblock {\em J. Comput. Graph. Statist.}, 12:531--547.

\bibitem[Jung and Marron, 2009]{JungMarron2009}
Jung, S. and Marron, J.~S. (2009).
\newblock P{CA} consistency in high dimension, low sample size context.
\newblock {\em Ann. Statist.}, 37(6B):4104--4130.

\bibitem[Kane and Nelson, 2014]{KaneNelson2014}
Kane, D.~M. and Nelson, J. (2014).
\newblock Sparser {J}ohnson-{L}indenstrauss transforms.
\newblock {\em J. ACM}, 61(1):Art. 4, 23.

\bibitem[Kouritzin, 1995]{Kouritzin1995}
Kouritzin, M.~A. (1995).
\newblock Strong approximation for cross-covariances of linear variables with
  long-range dependence.
\newblock {\em Stochastic Process. Appl.}, 60(2):343--353.

\bibitem[Ledoit and Wolf, 2004]{LedoitWolf2004}
Ledoit, O. and Wolf, M. (2004).
\newblock A well-conditioned estimator for large-dimensional covariance
  matrices.
\newblock {\em J. Multivariate Anal.}, 88(2):365--411.

\bibitem[Li and Zhao, 2013]{LiZhao2013}
Li, X. and Zhao, Z. (2013).
\newblock Testing for changes in autocovariances of nonparametric time series
  models.
\newblock {\em J. Statist. Plann. Inference}, 143(2):237--250.

\bibitem[Ma, 2013]{Ma2011}
Ma, Z. (2013).
\newblock Sparse principal component analysis and iterative thresholding.
\newblock {\em Ann. Statist.}, 41(2):772--801.

\bibitem[Na et~al., 2011]{NaLeeLee2011}
Na, O., Lee, Y., and Lee, S. (2011).
\newblock Monitoring parameter change in time series models.
\newblock {\em Stat. Methods Appl.}, 20(2):171--199.

\bibitem[Paul, 2007]{Paul2007}
Paul, D. (2007).
\newblock Asymptotics of sample eigenstructure for a large dimensional spiked
  covariance model.
\newblock {\em Statist. Sinica}, 17(4):1617--1642.

\bibitem[Paul and Johnstone, 2007]{PJ2007}
Paul, D. and Johnstone, I.~M. (2007).
\newblock Augmented sparse principal component analysis for high dimensional
  data.
\newblock Technical Report arXiv:1202.1242.

\bibitem[Philipp, 1986]{Philipp1986}
Philipp, W. (1986).
\newblock A note on the almost sure approximation of weakly dependent random
  variables.
\newblock {\em Monatsh. Math.}, 102(3):227--236.

\bibitem[Sancetta, 2008]{Sancetta2008}
Sancetta, A. (2008).
\newblock Sample covariance shrinkage for high dimensional dependent data.
\newblock {\em J. Multivariate Anal.}, 99(5):949--967.

\bibitem[Schweinberger et~al., 2017]{SchweinbergerEtAl2017}
Schweinberger, M., Babkin, S., and Ensor, K.~B. (2017).
\newblock High-dimensional multivariate time series with additional structure.
\newblock {\em J. Comput. Graph. Statist.}, 26(3):610--622.

\bibitem[Shen et~al., 2013]{ShenShenMarron2013}
Shen, D., Shen, H., and Marron, J.~S. (2013).
\newblock Consistency of sparse {PCA} in high dimension, low sample size
  contexts.
\newblock {\em J. Multivariate Anal.}, 115:317--333.

\bibitem[Shorack and Wellner, 1986]{ShorackWellner1986}
Shorack, G.~R. and Wellner, J.~A. (1986).
\newblock {\em Empirical processes with applications to statistics}.
\newblock Wiley Series in Probability and Mathematical Statistics: Probability
  and Mathematical Statistics. John Wiley \& Sons, Inc., New York.

\bibitem[Steland, 2005]{Steland2005}
Steland, A. (2005).
\newblock Random walks with drift---a sequential approach.
\newblock {\em J. Time Ser. Anal.}, 26(6):917--942.

\bibitem[Steland, 2018]{Steland2018}
Steland, A. (2018).
\newblock Shrinkage for covariance estimation: Asymptotics, confidence
  intervals, bounds and applications in sensor monitoring and finance.
\newblock {\em Statistical Papers}, 59:1441--1462.

\bibitem[Steland and von Sachs, 2017]{StelandSachs2017}
Steland, A. and von Sachs, R. (2017).
\newblock Large-sample approximations for variance-covariance matrices of
  high-dimensional time series.
\newblock {\em Bernoulli}, 23(4A):2299--2329.

\bibitem[Steland and von Sachs, 2018]{StelandSachs2018}
Steland, A. and von Sachs, R. (2018).
\newblock Asymptotics for high-dimensional covariance matrices and quadratic
  forms with applications to the trace functional and shrinkage.
\newblock {\em Stochastic Process. Appl.}, 128(8):2816--2855.

\bibitem[van~der Vaart, 1998]{vanderVaart1998}
van~der Vaart, A.~W. (1998).
\newblock {\em Asymptotic {S}tatistics}.
\newblock Cambridge University Press, Cambridge.

\bibitem[Vostrikova, 1981]{Vostrikova1981}
Vostrikova, L.~J. (1981).
\newblock Discovery of disorder in multidimensional random processes.
\newblock {\em Dokl. Akad. Nauk SSSR}, 259(2):270--274.

\bibitem[Wang and Samworth, 2018]{WangSamworth2018}
Wang, T. and Samworth, R.~J. (2018).
\newblock High dimensional change point estimation via sparse projection.
\newblock {\em J. R. Stat. Soc. Ser. B. Stat. Methodol.}, 80(1):57--83.

\bibitem[Wang and Fan, 2017]{WangFan2017}
Wang, W. and Fan, J. (2017).
\newblock Asymptotics of empirical eigenstructure for high dimensional spiked
  covariance.
\newblock {\em Ann. Statist.}, 45(3):1342--1374.

\bibitem[Yu et~al., 2015]{WangSamworth2015}
Yu, Y., Wang, T., and Samworth, R.~J. (2015).
\newblock A useful variant of the {D}avis-{K}ahan theorem for statisticians.
\newblock {\em Biometrika}, 102(2):315--323.

\end{thebibliography}

\end{document}